\documentclass[12pt]{amsart}
\usepackage{amssymb,latexsym}
\usepackage[all]{xy}
\usepackage{graphicx}

\numberwithin{equation}{section}

\newtheorem{theorem}{Theorem}[section]
\newtheorem{proposition}[theorem]{Proposition}

\newtheorem{lemma}[theorem]{Lemma}

\theoremstyle{definition}

\newtheorem{remark}[theorem]{Remark}
\newtheorem{example}[theorem]{Example}
\newtheorem{definition}[theorem]{Definition}

\newcommand{\prinA}{\mathcal{A}_{\bullet}}

\def\aa{\mathbf{a}}
\def\bb{\mathbf{b}}
\def\cc{\mathbf{c}}
\def\dd{\mathbf{d}}

\def\ee{\mathbf{e}}

\def\gg{\mathbf{g}}

\def\vv{\mathbf{v}}
\def\ww{\mathbf{w}}
\def\xx{\mathbf{x}}

\def\ZZ{\mathbb{Z}}

\def\PP{\mathbb{P}}

\def\QQ{\mathbb{Q}}

\def\Xcal{\mathcal{X}}

\def\Trop{\operatorname{Trop}}

\def\sgn{\operatorname{sgn}}

\renewcommand{\eqref}[1]{{\rm (\ref{#1})}}

\begin{document}


\title[Quantum F-polynomials in Classical Types]
{Quantum F-polynomials in Classical Types}

\author{Thao Tran}
\address{Department of Mathematics, Northeastern University,
Boston, MA 02115, USA} \email{tran.thao1@neu.edu}


\begin{abstract}  In their ``Cluster Algebras IV" paper, Fomin and Zelevinsky defined $F$-polynomials and $\gg$-vectors, and they showed that the cluster variables in any cluster algebra can be expressed in a formula involving the appropriate $F$-polynomial and $\gg$-vector. In ``$F$-polynomials in Quantum Cluster Algebras," the predecessor to this paper, we defined and proved the existence of quantum $F$-polynomials, which are analogs of $F$-polynomials in quantum cluster algebras in the sense that cluster variables in any quantum cluster algebra can be expressed in a similar formula in terms of quantum $F$-polynomials and $\gg$-vectors.   In this paper, we give formulas for both $F$-polynomials and quantum $F$-polynomials for cluster algebras of classical type when the initial exchange matrix is acyclic. 

\end{abstract}



\maketitle

\tableofcontents

\section{Introduction}

Cluster algebras were defined by Fomin and Zelevinsky in \cite{ca1}.  Since then, connections have been found in a variety of areas including algebraic combinatorics, quiver representations, Poisson geometry, and others.  Let $n \in \ZZ, n\geq 1$.  Roughly speaking, a \emph{cluster algebra} is a commutative algebra which is generated by a distinguished set of generators called \emph{cluster variables}.  To obtain these cluster variables, one starts with an initial \emph{seed}, which contains $n$ cluster variables as well as an \emph{exchange matrix}.  Using the entries from the exchange matrix, a process called \emph{mutation} is performed on the initial seed, which yields new seeds containing new cluster variables and exchange matrices.  Finally, the mutation process is iterated on the resulting seeds.  The total set of resulting cluster variables generates the cluster algebra.   \emph{Quantum cluster algebras}, which were defined by Berenstein and Zelevinsky in \cite{quantum}, are certain noncommutative deformations of cluster algebras.  

In \cite{coefficients}, Fomin and Zelevinsky defined $F$-polynomials and $\gg$-vectors. They gave a formula for cluster variables in terms of the corresponding $F$-polynomials and $\gg$-vectors which only depends on the data from the initial seed.  As a consequence, once the appropriate $F$-polynomial and $\gg$-vector are computed, a cluster variable may be expressed as a Laurent polynomial in the initial cluster variables.  

A (quantum) cluster algebra is of \emph{finite type} if the total set of cluster variables is finite.  In \cite{ca2}, it was proven that  finite type cluster algebras are classified by the same Cartan-Killing types as semisimple Lie algebras or finite root systems.    It was proven in \cite{quantum} that an identical classification also holds for quantum cluster algebras of finite type.  In particular, finite type cluster algebras include those of classical types $\mbox{A}_{n}$, $\mbox{B}_{n}$, $\mbox{C}_{n}$, and $\mbox{D}_{n}$.

In this paper, formulas for $F$-polynomials are given in classical types for the case where the initial matrix is acyclic (Theorem \ref{thm:fpoly-classical}).     Different formulas were given in the cases where the initial matrix is bipartite \cite{coefficients}, acyclic \cite{yangzel}, and in general  in \cite{mw}.   Also, a formula was given in \cite{msw} for $F$-polynomials and $\gg$-vectors for cluster algebras arising from surfaces with marked points as defined in \cite{triangulated}, and from this formula, one may compute the $F$-polynomials corresponding to type $\mbox{A}_{n}$ and $\mbox{D}_{n}$.     Theorem \ref{thm:fpoly-classical} differs from these other formulas in the statement of the formula, the methods for proving it, or both.  The formula for the finite type $F$-polynomials in this paper gives a combinatorial recipe for determining which monomials occur with nonzero coefficient and what the coefficient of the monomial is in that case.   The formula for type $\mbox{A}_{n}$ was already proven in \cite{qps2} using a formula for $F$-polynomials given in terms of quiver representations.  To prove the result in type $\mbox{D}_{n}$ here, the same formula is used.  To finish the proof in the other types, we show that the $F$-polynomials can be obtained as certain ``projections" of $F$-polynomials from type $\mbox{A}$ and $\mbox{D}$.

In \cite{quantumfpoly}, the predecessor to this paper, \emph{quantum $F$-polynomials} were defined; these are analogs to $F$-polynomials in that there exists a similar formula for computing cluster variables in quantum cluster algebras in terms of quantum $F$-polynomials and $\gg$-vectors.  Formulas for certain quantum $F$-polynomials were given in \cite{quantumfpoly}; in particular, this included all quantum $F$-polynomials of type $\mbox{A}_{n}$.  In this paper, Theorem \ref{thm:quantum-classical} gives a formula for quantum $F$-polynomials for the classical types when the initial matrix is acyclic.

\section{Cluster Algebras of Geometric Type} \label{section:ca}
Following \cite[Section 2]{coefficients}, we give the definition of a cluster algebra of geometric type as well as some properties of these cluster algebras.  The proofs of any statements given in this section can be found in \cite[Section 2]{coefficients}.
\begin{definition} \label{def:semifield-tropical}
Let $J$ be a finite set of labels,
and let $\Trop (u_j: j \in J)$ be an abelian group (written
multiplicatively) freely generated by the elements $u_j \, (j \in J)$.
We define the addition~$\oplus$ in $\Trop (u_j: j \in J)$ by
\begin{equation}
\label{eq:tropical-addition}
\prod_j u_j^{a_j} \oplus \prod_j u_j^{b_j} =
\prod_j u_j^{\min (a_j, b_j)} \,,
\end{equation}
and call $(\Trop (u_j: j \in J),\oplus,\cdot)$ a \emph{tropical
 semifield}.   If $J$ is empty,  we obtain the trivial semifield consisting of a single element~$1$.   The group ring of $\Trop (u_j: j \in J)$ is the ring of Laurent polynomials in the variables~$u_j\,$.
\end{definition}

Fix two positive integers $m$, $n$ with $m \geq n$.  Let $\mathbb{P} = \Trop (x_{n + 1}, \ldots, x_{m})$, and let $\mathcal{F}$ be the field of rational functions in $n$ independent variables with coefficients in $\QQ\PP$, the field of fractions of the integral group ring $\ZZ\PP$.  (Note that the definition of $\mathcal{F}$ does not depend on the auxiliary addition $\oplus$ in $\mathbb{P}$.)  The group ring $\ZZ\mathbb{P}$ will be the ground ring for the cluster algebra $\mathcal{A}$ to be defined, and $\mathcal{F}$ will be the ambient field, with $n$ being the \emph{rank} of $\mathcal{A}$.

\begin{definition} \label{def:labeled-seed}
A \emph{labeled seed} in $\mathcal{F}$ is a pair $(\tilde{\xx}, \tilde{B})$ where
\begin{itemize}
  \item $\tilde{\xx} = (x_{1}, \ldots, x_{m})$, where $x_{1}, \ldots, x_{n}$ are algebraically independent over $\QQ\PP$, and $\mathcal{F} = \QQ\PP(x_{1}, \ldots, x_{n})$, and 
  \item $\tilde{B}$ is an $m \times n$ integer matrix such that the submatrix $B$ consisting of the top $n$ rows and columns of $\tilde{B}$ is skew-symmetrizable (i.e., $DB$ is skew-symmetric for some $n \times n$ diagonal matrix $D$ with positive integer diagonal entries).
\end{itemize}
We call $\tilde{\xx}$ the \emph{extended cluster} of the labeled seed, $(x_{1}, \ldots, x_{n})$ the \emph{cluster}, $\tilde{B}$ the \emph{exchange matrix}, and the matrix $B$ the \emph{principal part} of $\tilde{B}$.
\end{definition}
We fix some notation to be used throughout the paper.  For $x \in \QQ$,  
\begin{align*}
[x]_+ &= \max(x,0); \\
\sgn(x) &=
\begin{cases}
-1 & \text{if $x<0$;}\\
0  & \text{if $x=0$;}\\
 1 & \text{if $x>0$;}
\end{cases}\\
\end{align*}
Also, for $i, j \in \ZZ$, write $[i, j]$ for the set $\{ k \in \ZZ : i \leq k \leq j \}$.  In particular, $[i, j] = \varnothing$ if $i > j$.

\begin{definition} \label{def:matrix-mut} Let $k \in [1, n]$.  We say that an $m \times n$ matrix $\tilde{B}'$ is obtained from an $m \times n$ matrix $\tilde{B} = (b_{ij})$ by \emph{matrix mutation} in direction $k$ if the entries of $\tilde{B}'$ are given by 
\begin{equation}
\label{eq:matrix-mutation}
b'_{ij} =
\begin{cases}
-b_{ij} & \text{if $i=k$ or $j=k$;} \\[.05in]
b_{ij} + \sgn(b_{ik}) \ [b_{ik}b_{kj}]_+
 & \text{otherwise.}
\end{cases}
\end{equation}
\end{definition} 

\begin{definition} \label{def:seed-mut} Let $(\tilde{\xx}, \tilde{B})$ be a labeled seed in $\mathcal{F}$ as in Definition \ref{def:labeled-seed}, and write $\tilde{B} = (b_{ij})$.  The \emph{seed mutation} $\mu_{k}$ in direction $k$ transforms $(\tilde{\xx}, \tilde{B})$ into the labeled seed $\mu_{k}(\tilde{\xx}, \tilde{B}) = (\tilde{\xx}', \tilde{B}')$, where 
\begin{itemize}
   \item $\tilde{\xx}'=(x_1',\dots,x_m')$, where $x'_{j} = x_{j}$ for $j \neq k$, and 
      \begin{equation}
\label{eq:exchange-relation}
x'_k = x_k^{-1}\left(
\displaystyle\prod_{i=1}^m x_i^{[b_{ik}]_+}
+ \displaystyle\prod_{i=1}^m x_i^{[-b_{ik}]_+}\right) ,
\end{equation}
   \item $\tilde{B}'$ is obtained from $\tilde{B}$ by matrix mutation in direction $k$.
\end{itemize}
\end{definition}
One may check that the pair $(\tilde{\xx}', \tilde{B}')$ obtained is again a labeled seed.  Furthermore, the seed mutation $\mu_{k}$ is involutive, i.e., applying $\mu_{k}$ to $(\tilde{\xx}, \tilde{B})$ yields the original labeled seed $(\tilde{\xx}, \tilde{B})$.

\begin{definition} Let $\mathbb{T}_{n}$ be the $n$-regular tree whose edges are labeled with $1, \ldots, n$ in such a way that for each vertex, the $n$ edges emanating from that vertex each receive different labels.  Write $t \frac{k}{\hspace{1cm}} t'$ to indicate $t, t' \in \mathbb{T}_{n}$ are joined by an edge with label $k$.
\end{definition}

\begin{definition} \label{def:cluster-pattern} A \emph{cluster pattern} is an assignment of a labeled seed $(\tilde{\xx}_{t}, \tilde{B}_{t})$ to every vertex $t \in \mathbb{T}_{n}$ such that if $t \frac{k}{\hspace{1cm}} t'$, then the labeled seeds assigned to $t$, $t'$ may be obtained from one another by seed mutation in direction $k$.   Write $\tilde{\xx}_{t}  = (x_{1; t}, \ldots, x_{m; t})$, $\tilde{B}_{t}  =  (b^t _{ij})$, and denote by $B_{t}$ the principal part of $\tilde{B}_{t}$.  
\end{definition}

\begin{definition} For a given cluster pattern, write 
\begin{equation}
\Xcal = \{ x_{j; t} : j \in [1, n], t \in \mathbb{T}_{n} \}.
\end{equation}
The elements of $\Xcal$ are the \emph{cluster variables}.  The \emph{cluster algebra} $\mathcal{A}$ associated to this cluster pattern is the $\ZZ\PP$-subalgebra of $\mathcal{F}$ generated by all cluster variables.  That is, $\mathcal{A} = \ZZ\PP[\Xcal]$.
\end{definition}

\section{$F$-polynomials and $\gg$-vectors} \label{section:f-polys}

For this section, fix an $n \times n$ skew-symmetrizable integer matrix $B^{0} = (b^{0}_{ij})$ and an initial vertex $t_{0} \in \mathbb{T}_{n}$.   The information in this section can be found in \cite{coefficients}.

\begin{definition} We say that a cluster pattern $t \mapsto (\tilde{\xx}_{t}, \tilde{B}_{t})$ (or its corresponding cluster algebra) has \emph{principal coefficients} at $t_{0}$ if $\tilde{\xx}_{t_{0}} = (x_{1}, \ldots, x_{2n})$ (i.e., $m = 2n$) and the exchange matrix at $t_{0}$ is the \emph{principal matrix} corresponding to $B^{0}$ given by  
\begin{equation}
\tilde{B}^{t_{0}} = \left( \begin{array}{c} B^{0}  \\ I
                             \end{array} \right)
\end{equation}
where $I$ is the $n \times n$ identity matrix.  Denote the corresponding cluster algebra by $\mathcal{A}_{\bullet} = \mathcal{A}_{\bullet}(B^{0}, t_{0})$. 
\end{definition}

Using the initial extended cluster of $\mathcal{A}_{\bullet}$, define
\begin{eqnarray} \label{def:yhat}
 \hat{y}_{j} = x_{n + j} \prod_{i = 1}^{n} x_{i}^{b^{0}_{ij}} \hspace{0.5cm}
\end{eqnarray}
for $j \in [1, n]$.

\begin{definition}  A polynomial $F \in \ZZ[u_{1}, \ldots, u_{n}]$ is \emph{primitive} if no $u_{i}$ divides $F$ for $i \in [1, n]$.
\end{definition}

\begin{proposition} \label{thm:cluster-var-formula} Let $\ell \in [1, n]$, $t \in \mathbb{T}_{n}$.  There exists a unique primitive polynomial 
\begin{eqnarray}
F_{\ell; t} = F_{\ell; t}^{B^{0}; t_{0}} \in \ZZ[u_{1}, \ldots, u_{n}]
\end{eqnarray}
and a unique vector $\gg_{\ell; t} = \gg_{\ell; t}^{B^{0}; t_{0}} = (g_{1}, \ldots, g_{n}) \in \ZZ^{n}$ such that  the cluster variable $x_{\ell; t} \in  \prinA(B^{0}, t_{0})$ is given by the equation
\begin{eqnarray} \label{eq:cluster-var-fpoly}
x_{\ell; t} = F^{B^{0}; t_{0}}_{\ell; t}(\hat{y}_{1}, \ldots, \hat{y}_{n})x_{1}^{g_{1}}\cdots x_{n}^{g_{n}}.
\end{eqnarray}
\end{proposition}
\begin{proof}
The proposition follows immediately from Corollary 6.3, equation (5.5), and Proposition 7.8 of \cite{coefficients}.
\end{proof}

The polynomial $F_{\ell; t}$ given in Proposition \ref{thm:cluster-var-formula} is called an \emph{$F$-polynomial}, and $\gg_{\ell; t}$ is called a \emph{$\gg$-vector}.

\section{$F$-polynomials  in Classical Types} \label{fpoly-classical}

Let $B^{0} = (b^{0}_{ij})$ be an acyclic $n \times n$ exchange matrix of type $\mbox{A}_{n}$, $\mbox{B}_{n}$, $\mbox{C}_{n}$, or $\mbox{D}_{n}$.  In particular, $b_{ij} = \pm a_{ij}$ for $i \neq j$, where $(a_{ij})$ is the Cartan matrix of the corresponding type.    We will use the convention for Cartan matrices given in \cite{kac} which is different from the one given in \cite{bourbaki}.

We recall from  \cite{ca1} the next definition.  Let $\mathcal{A}$ be any cluster algebra whose initial exchange matrix has principal part $B^{0}$.  Suppose $\mathcal{A}$ has initial extended cluster $(x_{1}, \ldots, x_{m})$.  By the Laurent phenomenon (\cite[Theorem~3.1]{ca1}), any cluster variable $x_{j; t} \in \mathcal{A}$ may be expressed as 
\begin{eqnarray}
x_{j; t} = \displaystyle \frac{N(x_{1}, \ldots, x_{n})}{x_{1}^{d_{1}} \ldots x_{n}^{d_{n}}  },
\end{eqnarray}
where $N(x_{1}, \ldots, x_{n})$ is a polynomial with coefficients in $\ZZ[x^{\pm 1}_{n + 1}, \ldots, x^{\pm 1}_{m}]$ which is not divisible by any $x_{i}$.  Let $\dd_{j; t}^{B^{0}; t_{0}} = \left[   \begin{array}{c} d_{1} \\ \vdots \\ d_{n}   \end{array}  \right]$, and call this the \emph{denominator vector} of the cluster variable $x_{j; t}$.    One may show that $\dd_{j; t}$ does not depend on the choice of coefficients.

Denote by $\Phi_{+} = \Phi_{+}(B^{0})$ the set of all denominator vectors of cluster variables in $\prinA(B^{0}, t_{0})$ which do not occur in the initial cluster.  It was proven in \cite{yangzel} that these cluster variables are in bijective correspondence with the positive roots corresponding to the type of $B^{0}$, and this correspondence is via denominator vectors.   To be more precise, if a cluster variable corresponds to $\alpha = \sum d_{i} \alpha_{i}$, where the simple roots are given by $\alpha_{1}, \ldots, \alpha_{n}$, then the denominator vector of the cluster variable is $(d_{1}, \ldots, d_{n}) \in \ZZ^{n}$.    

Let $\ee_{i}$ be the $i$th elementary vector in $\ZZ^{n}$.  The sets $\Phi_{+}$ are given in each type by the following:

\

\textbf{Type $\mbox{A}_{n}$}: 
$\Phi_{+} = \{ \ee_{i} + \ldots + \ee_{j} : 1 \leq i \leq j \leq n \}$

\textbf{Type $\mbox{B}_{n}$:} 
$\Phi_{+} = \{  \ee_{i} + \ldots + \ee_{j} : 1 \leq i \leq j \leq n \}  \cup \{ \ee_{i} + \ldots + \ee_{j - 1} + 2\ee_{j} + \ldots + 2\ee_{n} : 1 \leq i < j \leq n \}$

\textbf{Type $\mbox{C}_{n}$}:
$\Phi_{+} = \{ \ee_{i} + \ldots + \ee_{j} : 1 \leq i \leq j \leq n \} \cup
\{ \ee_{i} + \ldots + 2\ee_{j} + \cdots + 2\ee_{n - 1} +  \ee_{n} : 1 \leq i < j < n \} \cup \{ 2\ee_{i} + \cdots + 2\ee_{n - 1} + \ee_{n} : 1 \leq i \leq n - 1\}$

\textbf{Type $\mbox{D}_{n}$:} 
$\Phi_{+} = \{ \ee_{i} + \ldots + \ee_{j} : 1 \leq i \leq j \leq n \} \cup
\{ \ee_{i} + \ldots + \ee_{n - 2} + \ee_{n} : 1 \leq i \leq n - 2 \} \cup
\{ \ee_{i} + \ldots + \ee_{j - 1} + 2\ee_{j} + \ldots + 2\ee_{n - 2} + \ee_{n - 1} + \ee_{n} : 1 \leq i < j \leq n - 2 \}
$

\

Let $F_{\dd} = F_{\dd}^{B^{0}; t_{0}}$ be the $F$-polynomial corresponding to the cluster variable with denominator vector $\dd$.  

For any $n \times n$ skew-symmetrizable matrix $B$, let $Q(B)$ be the quiver with vertices $[1, n]$ and an arrow from $i \rightarrow j$ whenever $b_{ji} > 0$.  The matrix $B$ is acyclic if $Q(B)$ is acyclic (i.e., it doesn't contain any directed cycles).  Write $Q^{0} = Q(B^{0})$.  

Define a partial order $\geq$ on $\ZZ^{n}$ by 
\begin{eqnarray} 
\aa \geq \aa' \mbox{ if } \aa - \aa' \in \ZZ^{n}_{\geq 0}.
\end{eqnarray} 

Let $\dd = (d_{1}, \ldots, d_{n}) \in \Phi_{+}(B^{0})$ and $\textbf{e} = (e_1, \ldots, e_n) \in \mathbb{Z}^{n}$ such that $0 \leq \ee \leq \dd$. 

\begin{definition} \label{def:acceptable} We call an arrow $i \to j$ in $Q^0$ \emph{acceptable} (with respect to the pair $(\dd, \ee)$) if $e_i - e_j
\leq [d_i - d_j]_+$.  Also, an arrow $i \to j$ is called \emph{critical} (with respect to $(\dd, \ee)$) if either
$$(d_i, e_i) = (2,1), \quad (d_j,e_j) = (1,0),$$
or
$$(d_j, e_j) = (2,1), \quad (d_i,e_i) = (1,1).$$
\end{definition}

Note that a critical arrow is always acceptable.   Let $S$ be the induced subgraph of $Q^0$ on the set of vertices $\{i: (d_i,e_i) = (2,1)\}$.   For a connected component $C$ of $S$, define $\nu(C)$ as the number of critical arrows having a vertex in $C$.

\begin{theorem} \label{thm:fpoly-classical} Fix $\dd = (d_{1}, \ldots, d_{n}) \in \Phi_{+}(B^{0})$, $\ee = (e_{1}, \ldots, e_{n}) \in \ZZ^{n}$, and define $S$ as above. The coefficient of the monomial $u_{1}^{e_{1}}\ldots u_{n}^{e_{n}}$ in $F_{\dd}$ is nonzero if and only if 
\begin{enumerate}
 \item $0 \leq \ee \leq \dd$;
 \item all arrows in $Q^{0}$ are acceptable;
 \item $\nu(C) \leq 1$ for all components $C$ of $S$;
 \item if $B^{0}$ is of type $\mbox{C}_{n}$, then 
  \begin{enumerate} 
     \item $e_{n} = 1$, $d_{n - 1} = 2$, and  $n \rightarrow n - 1$  in $Q^{0}$ implies that $e_{n - 1} = 2$;
     \item $e_{n - 1} \geq 1$, $d_{n} = 1$, and $n - 1 \rightarrow n$ in $Q^{0}$ imply that $e_{n} = 1$.
  \end{enumerate}
  \item If $B^{0}$ is of type $\mbox{B}_{n}$, $S$ consists of a single component which contains the vertex $n - 1$, and there is a critical arrow in $S$, then $S$ does not contain the vertex $n$.
\end{enumerate}
If the conditions above are all satisfied, then the coefficient of $u_{1}^{e_{1}}\ldots u_{n}^{e_{n}}$ is $2^{c}$, where $c$ is the number of components $C$ such that $\nu(C) = 0$.
\end{theorem} 
\begin{remark} If $0 \leq d_{i} \leq 1$ for all $i \in [1, n]$, then $S$ is empty, so the third condition is automatically satisfied.   
\end{remark}

The formula for $\gg$-vectors in the next theorem will be useful in computing quantum $F$-polynomials for the classical types.

\begin{theorem} \cite{yangzel} \label{thm:g-vec-classical} The $\gg$-vector $\gg_{\dd}$ corresponding to $\dd = (d_{1}, \ldots, d_{n}) \in \Phi_{+}(B^{0})$ is given by 
\begin{eqnarray}
-\sum_{i \in [1, n]} d_{i}\textbf{e}_{i} + \sum_{i, j \in [1, n]}  d_{i}[-b^{0}_{ji}]_{+}\textbf{e}_{j}.
\end{eqnarray}
\end{theorem}
\begin{proof}
This follows from Theorems 1.8 and 1.10 of \cite{yangzel}.
\end{proof}

\begin{example} \label{example:typeB} Let
\begin{eqnarray}
B^{0} = \left(  \begin{array}{rrrr} 0 & -1 & 0 & 0 \\
1 & 0 & 1 & 0 \\
0 & -1 & 0 & -1 \\
0 & 0 & 2 & 0  
\end{array}
\right)
\end{eqnarray}
Then $B^{0}$ is of type $\mbox{B}_{4}$, and $Q^{0}$ is the quiver below:

\[
\xymatrix{
1 \ar[r] & 2 & 3\ar[l] \ar[r] & 4
}
\]

\

Suppose that $\dd = \ee_{1} + 2\ee_{2} + 2\ee_{3} + 2\ee_{4}$.   Note that the only possible critical arrow in $Q^{0}$ is $1 \rightarrow 2$, and this arrow is critical if $e_{1} = e_{2} = 1$.   Let $\ee = (e_{1}, e_{2}, e_{3}, e_{4}) \in \ZZ^{4}$.  Then $u_{1}^{e_{1}}u_{2}^{e_{2}}u_{3}^{e_{3}}u_{4}^{e_{4}}$ occurs with nonzero coefficient in $F_{\dd}$ if and only if $0 \leq \ee \leq \dd$, the three  inequalities $e_{1} \leq e_{2}$, $e_{2} \geq e_{3}$, and $e_{4} \geq e_{3}$ are satisfied, and $\ee \neq (1, 1, 1, 1)$.  Using these conditions, it is easy to check that the $F$-polynomial corresponding to $\dd$ is 
\begin{eqnarray}
F_{\ee_{1} + 2\ee_{2} + 2\ee_{3} + 2\ee_{4}} & = & u_{1}u_{2}^{2}u_{3}^{2}u_{4}^{2} + 2u_{1}u_{2}^{2}u_{3}u_{4}^{2} + 2u_{2}^{2}u_{3}^{2}u_{4}^{2} + 2u_{1}u_{2}^{2}u_{3}u_{4}   \\
\nonumber & & \hspace{0.1cm}  +\, u_{1}u_{2}u_{3}u_{4}^{2}  + \, u_{1}u_{2}^{2}u_{4}^{2} + 2u_{1}u_{2}^{2}u_{4} + u_{1}u_{2}^{2} + u_{1}u_{2}u_{4}^{2}  \\ 
\nonumber & & \hspace{0.1cm}  +\, u_{2}^{2}u_{3}^{2}u_{4}^{2} + \, 2u_{2}^{2}u_{3}u_{4} + 2u_{2}u_{3}u_{4}^{2} + 2u_{2}u_{4}^{2} + 2u_{2}^{2}u_{4}   \\
\nonumber & & \hspace{0.1cm}    +\, u_{2}^{2}u_{4}^{2}  +\, u_{2}^{2} + u_{4}^{2} + 2u_{2}u_{3}u_{4} + 2u_{1}u_{2}u_{4} + u_{1}u_{2}  \\
\nonumber & & \hspace{0.1cm}  + \, 4u_{2}u_{4} + 2u_{4} + 2u_{2} + 1.
\end{eqnarray}

Some other $F$-polynomials corresponding to $B^{0}$ are given below.
\begin{eqnarray}
F_{\ee_{2}} & = & u_{2} + 1 \\
F_{\ee_{2} + \ee_{3}} & = & u_{2}u_{3} + u_{2} + 1 \\
F_{\ee_{2} + \ee_{3} + \ee_{4}} & = & u_{2}u_{3}u_{4} + u_{2}u_{4} + u_{4} + u_{2} + 1 \\
F_{\ee_{2} + 2\ee_{3} + 2\ee_{4}} & = & u_{2}u_{3}^{2}u_{4}^{2} + 2u_{2}u_{3}u_{4}^{2} + 
2u_{2}u_{3}u_{4} + u_{3}u_{4}^{2} + u_{2}u_{4}^{2} \\  
\nonumber &  &  \hspace{1cm} + \, 2u_{2}u_{4} + u_{2} + u_{4}^{2} + 2u_{4} + 1 
\end{eqnarray}

The corresponding $\gg$-vectors are given below.
\begin{eqnarray}
\gg_{\ee_{2}} & = & \ee_{1} - \ee_{2} + \ee_{3} \\
\gg_{\ee_{2} + \ee_{3}} & = & \ee_{1} - \ee_{2} \\
\gg_{\ee_{2} + \ee_{3} + \ee_{4}} & = & \ee_{1} - \ee_{2} + \ee_{3} - \ee_{4} \\
\gg_{\ee_{2} + 2\ee_{3} + 2\ee_{4}} & = & \ee_{1} - \ee_{2} + \ee_{3} - 2\ee_{4} \\
\gg_{\ee_{1} + 2\ee_{2} + 2\ee_{3} + 2\ee_{4}} & = & \ee_{1} - 2\ee_{2} + 2\ee_{3} - 2\ee_{4} 
\end{eqnarray}
\end{example}

Theorem \ref{thm:fpoly-classical} was proven in \cite{quantumfpoly} to be true for type $\mbox{A}_{n}$.  The remainder of this section is devoted to proving Theorem \ref{thm:fpoly-classical} in the remaining types.  The following easily verified proposition will be useful in checking if an arrow $i \rightarrow j$ in $Q^{0}$ is acceptable.

\begin{proposition} \label{prop:acceptable} Let $\dd = (d_{1}, \ldots, d_{n}) \in \Phi_{+}(B^{0})$, $\ee = (e_{1}, \ldots, e_{n}) \in \ZZ^{n}$ such that $0 \leq e_{i} \leq d_{i}$ for all $i \in [1, n]$, and let $i \rightarrow j$ be an arrow in $Q^{0}$.  If at least one of the equalities $d_{i} = 0$, $d_{j} = 0$, $e_{i} = 0$, or $e_{j} = d_{j}$ holds, then $i \rightarrow j$ is an acceptable arrow. 
\end{proposition}

\subsection{Type $\mbox{D}_{n}$}

In this subsection, let $B^{0}$ be an acyclic $n \times n$ exchange matrix of type $\mbox{D}_{n}$.   The underlying undirected graph of the quiver $Q^{0} = Q(B^{0})$ is simply the Dynkin diagram of the corresponding type.
For these types, we will use the formulas for $F$-polynomials given in \cite{qps2}.  First, recall some terminology and notation.  

A \emph{representation} $M$ (over $\mathbb{C}$) of a quiver $Q$ is given by assigning $\mathbb{C}$-vector space to each vertex $i \in [1, n]$ and assigning a $b$-tuple $(\varphi_{ji}^{(1)}, \ldots, \varphi_{ji}^{(b)})$ of linear maps $\varphi_{ji}^{(k)}: M_{i} \rightarrow M_{j}$ to every arrow $i \rightarrow j$ of multiplicity $b = b_{ji}$.  
The \emph{dimension} of a representation $M$ is the vector $\dd = (d_{1}, \ldots, d_{n}) \in \ZZ^{n}$ given by $d_{i} = \mbox{dim}(M_{i})$.  A \emph{subrepresentation} $N$ of a representation $M$ is given by a collection of subspaces $N_{i} \subset M_{i}$ for each $i \in [1, n]$ such that $\varphi_{ji}^{(k)}(N_{i}) \subset N_{j}$ for all $i, j, k$.  For a representation $M$ of dimension $\dd$ and for an integer vector $\ee = (e_{1}, \ldots, e_{n})$ such that $0 \leq \ee \leq \dd$ (i.e., $0 \leq e_{i} \leq d_{i}$ for all $i$), let $\mbox{Gr}_{\ee}(M)$ denote the variety of all subrepresentations of $M$ of dimension $\ee$.  Thus, $\mbox{Gr}_{\ee}(M)$ is a closed subvariety of the product of Grassmanians $\prod \mbox{Gr}_{e_{i}}(M_{i})$.  Denote by $\chi_{\ee}(M)$ the Euler-Poincare characteristic of $\mbox{Gr}_{\ee}(M)$ (see \cite[Section 4.5]{fulton}).

 To compute $F_{\dd}$, we use the following result adapted from \cite{qps2}:

\begin{theorem}  \label{fpoly-quiver-rep} Assume $B^{0}$ is as above. Let $\dd \in \Phi_{+}(B^{0})$, and let $M$ be an indecomposable representation of $Q^{0}$ of dimension $\dd$.  Then 
\begin{eqnarray}
F_{\dd} = \displaystyle \sum \chi_{\ee}(M)u_{1}^{e_{1}}\ldots u_{n}^{e_{n}},
\end{eqnarray}
where the summation ranges over $\ee = (e_{1}, \ldots, e_{n}) \in \ZZ^{n}$ such that $0 \leq \ee \leq \dd$.
\end{theorem} 

The following lemma will be useful in determining when there exist subrepresentations of $M$ of dimension $\ee$:

\begin{lemma} \label{acceptable-arrows}
Let $M'$ and $M''$ be vector spaces of dimensions $d'$ and $d''$,
respectively, and $\varphi: M' \to M''$ be a linear map of maximal
possible rank $\min(d',d'')$.  Let $e'$ and $e''$ be two integers such
that $0 \leq e' \leq d'$ and $0 \leq e'' \leq d''$. Then the following
conditions are equivalent:
\begin{enumerate}
\item There exist subspaces $N' \subseteq M'$ and $N'' \subseteq M''$
such that $\dim N' = e'$, $\dim N'' = e''$, and $\varphi (N')
\subseteq N''$.
\item $e' - e'' \leq [d' - d'']_+$.
\end{enumerate}
\end{lemma}
\begin{proof} First, observe that 
\begin{eqnarray}
\mbox{dim}(\mbox{Ker}\,\varphi) = d' - \mbox{min}(d', d'') = [d' - d'']_{+}.
\end{eqnarray}
For the proof that (1) implies (2), 
\begin{eqnarray}
[d' - d'']_{+} \geq \mbox{dim}(\mbox{Ker}\,\varphi |_{N'}) = e' - \mbox{dim}(\mbox{Im}\, \varphi |_{N'}) \geq e' - e''.
\end{eqnarray}
Now assume that (2) holds.   Choose a subspace $N'$ of $M'$ of dimension $e'$ so that $N' \supset \mbox{Ker}\,\varphi$ if $e' \geq [d' - d'']_{+}$, or $N' \subset \mbox{Ker}\,\varphi$ if $e' \leq [d' - d'']_{+}$.  Then 
\begin{eqnarray}
\hspace{1cm} \mbox{dim}(\mbox{Im}\, \varphi |_{N'}) & = & e' - \mbox{dim}(\mbox{Ker}\, \varphi |_{N'}) \\
& = & \left\{
                \begin{array}{ll} 
                     e' - [d' - d'']_{+} & \mbox{ if } e' \geq [d' - d'']_{+} \\
                     0 & \mbox{ otherwise.} 
                \end{array}
             \right. 
\end{eqnarray}
Since $\mbox{dim}(\mbox{Im}\, \varphi |_{N'}) \leq e''$ in either case, any $e''$-dimensional subspace $N''$ of $M''$ which contains $\mbox{Im}\, \varphi |_{N'}$ will work.
\end{proof}

We may restrict attention to $\ee \in \ZZ^{n}$ which satisfy the first two conditions given in Theorem \ref{thm:fpoly-classical}.  (If these conditions do not hold, Theorem \ref{fpoly-quiver-rep} together with Lemma \ref{acceptable-arrows} imply that the coefficient of $u_{1}^{e_{1}}\ldots u_{n}^{e_{n}}$ in $F_{\dd}$ equals 0.)   If $d_{i} = 0$ or 1 for all $i$, then Theorem \ref{thm:fpoly-classical} follows from \cite[Theorem 7.4]{quantumfpoly}.   If $d_{i} > 1$ for some $i$, then $\dd = \ee_{p} + \ldots + \ee_{r - 1} + 2\ee_{r} + \ldots + 2\ee_{n - 2} + \ee_{n - 1} + \ee_{n}$ for some $p, r \in \ZZ$ satisfying $1 \leq p < r \leq n - 2$.   An indecomposable representation $M = (M_{i})$ of dimension $\dd$ can be selected so that for $p \leq i \leq r - 2$ and for $r \leq i \leq n - 3$, the map between $M_{i}$ and $M_{i + 1}$ (which is either $\varphi_{i, i + 1}$ or $\varphi_{i + 1, i}$) is the identity map.  If $r - 1 \rightarrow r$ in $Q^{0}$, then let $V_{r} = Im(\varphi_{r, r - 1})$; otherwise, let $V_{r} = Ker(\varphi_{r - 1, r})$.  For each value $j = n - 1$ and $j = n$,  if $n - 2 \rightarrow j$ in $Q^{0}$, then let $V_{j} = Ker(\varphi_{j, n - 2})$; otherwise, let $V_{j} = Im(\varphi_{n - 2, j})$.   Then $V_r, V_{n-1}$, and $V_n$ are three distinct one-dimensional subspaces of the two-dimensional space $M_{n-2}$.

To compute $\chi_{\ee}(M)$, we will demonstrate how to construct all possible subrepresentations which have dimension vector $\textbf{e}$ (if there are any), from which it will be easy to compute $\chi_{\textbf{e}}(M)$.  In order to verify that a given $N = (N_{i})$ (with dimension $\ee$) is a subrepresentation of $M$, we need to check that for each $i, j \in [1, n]$ which are connected by an edge in $Q^{0}$, 
\begin{eqnarray} \label{subrepprop}
\varphi_{ji}(N_{i}) \subset N_{j} \mbox{ if } i \rightarrow j \mbox{ in } Q^{0}, \mbox{ or } \varphi_{ij}(N_{j}) \subset N_{i} \mbox{ if } j \rightarrow i \mbox{ in } Q^{0}. 
\end{eqnarray}
Observe that for any $i \in [1, n] - S$, there is only one possible subspace of $M_{i}$ of dimension $e_{i}$.  Consider a component $C$ of $S$.  If $i, j$ are vertices in $C$ and $i \rightarrow j$, then (\ref{subrepprop}) implies that $N_{i} = N_{j}$.  Thus, once a subspace of dimension 1 is chosen for one vertex of the component, all of the vertices in that component must be assigned that same subspace.   If $i, j \in [p, n]$ such that $i \rightarrow j$ in $Q^{0}$ and so that at least one of $i, j$ is in $[r, n - 2] - S$, then it is easy to check that $\varphi_{ji}(N_{i}) \subset N_{j}$.  For example, if $i \in [r, n - 2] - S$, then $e_{i} = 0$ or $e_{i} = 2$.  If $e_{i} = 0$, then $\varphi_{ji}(N_{i}) = 0$.  If $e_{i} = 2$, then the fact that $i \rightarrow j$ is acceptable implies $e_{j} = d_{j}$, so $N_{j} = M_{j} \supset \varphi_{ji}(N_{i})$.  One can also easily check that for if $i \rightarrow j$ with $i, j \in [1, r - 1]$, then $\varphi_{ji}(N_{i}) \subset N_{j}$.  Thus, it only remains to verify for each proposed subrepresentation $(N_{i})$ that the property (\ref{subrepprop}) holds for $(i, j) =  (r - 1, r)$ if $r \in S$, and that the property holds for $(i, j) = (n - 2, n - 1),(n - 2, n)$ if $n - 2 \in S$.

Next, consider which 1-dimensional subspaces may be assigned to each component.  Observe that when $(i, j) = (r - 1, r)$ is a critical pair, property (\ref{subrepprop}) is satisfied for $(i, j)$ by $N$ if and only if $N_{r} = V_{r}$.   Also, when $(i, j) = (n - 2, n)$ (resp. $(i, j) = (n - 2, n - 1)$) is a critical pair, property (\ref{subrepprop}) is satisfied for $(i, j)$ if and only if $N_{n - 2} = V_{n}$ (resp. $N_{n - 2} = V_{n - 1}$).

Let $C$ be a component of $S$.  If $\nu(C) = 0$, then any 1-dimension subspace of $\mathbb{C}^{2}$ may be assigned to the vertices of $C$.   Suppose $\nu(C) = 1$.  If $(r - 1, r)$ is a critical pair and $r \in C$, then the subspace on the vertices of $C$ must by $V_{r}$.  If $(n - 2, n)$ (resp. $(n - 2, n - 1)$) is a critical pair and $n - 2 \in C$, then the subspace on the vertices of $C$ must $V_{n}$ (resp. $V_{n - 1}$).  Finally, if $\nu(C) \geq 2$, then there is no 1-dimensional space which can be assigned to the vertices of $C$ so that (\ref{subrepprop}) is satisfied.   Since the Euler-Poincare characteristic of the projective line  $\mathbb{P}^{1}$ is 2, it follows that the contribution to $\chi_{\textbf{e}}(M)$ is a multiplicative factor of 0, 1, or 2 when $\nu(C) = 0$, $\nu(C) = 1$, or $\nu(C) \geq 2$, respectively.  This concludes the proof of Theorem \ref{thm:fpoly-classical} in type $\mbox{D}_{n}$.

\subsection{Projections of $F$-polynomials and $\textbf{g}$-vectors} 

To prove Theorem \ref{thm:fpoly-classical} in the remaining types, we will show that $F$-polynomials and $\gg$-vectors of type $\mbox{B}_{n}$ and $\mbox{C}_{n}$ can be obtained as certain ``projections" of $F$-polynomials and $\gg$-vectors of type $\mbox{D}_{n + 1}$ and $\mbox{A}_{2n - 1}$, respectively.  
In this subsection, we prove a more general result (Theorem \ref{thm:fpoly-proj}) concerning such projections.
For this subsection, let $B^{0}$ be an arbitrary $n \times n$ skew-symmetric integer matrix  such that $B^{0}$ is acyclic.     The methods are similar to those used by Dupont in \cite{dupont} for coefficient-free cluster algebras, but to give the desired results, we will need to work with cluster algebras with principal coefficients.  

We begin by recalling some terminology and notation from \cite{dupont}.  For $\sigma \in S_{n}$, let $\tilde{\sigma} \in S_{2n}$ be given by $\tilde{\sigma}(i) = \sigma(i)$ and $\tilde{\sigma}(i + n) = \sigma(i) + n$ for $i \in [1, n]$.  (Thus, $\sigma$ acts on the set $[1, 2n]$ via $\tilde{\sigma}$.)  If $\tilde{B} = (b_{ij})$ is in $M_{2n, n}(\ZZ)$, then define an action of $\sigma$ on $\tilde{B}$ by $\sigma \tilde{B} = (b_{\tilde{\sigma}^{-1}i, \tilde{\sigma}^{-1}j})$.  We may also define an action of $\sigma$ on matrices $B$ in $M_{n, n}(\ZZ)$ in a similar way.   Say that a subgroup $G$ of $S_{n}$ is an \emph{automorphism group} for $\tilde{B}$ (resp. $B$) if $\sigma \tilde{B} = \tilde{B}$ (resp. $\sigma B = B$) for all $\sigma \in G$.   

Fix a group $G$ of automorphisms of the matrix $B^{0}$ for the remainder of this subsection.  (Note that $G$ is also a group of automorphisms for $\tilde{B}^{0}$.)

Let $I = [1, n]$.  Write $\overline{I}$ for the set of orbits under the action of $G$ on $I$, and $\overline{i}$ for the orbit $Gi$. Fix some ordering $\overline{I} = \{ \overline{j_{1}}, \ldots, \overline{j_{r}} \}$, where $r = |\overline{I}|$.   For a matrix $B \in M_{n, n}(\ZZ)$, define the \emph{quotient matrix} $\overline{B} = (\overline{b}_{\overline{i}, \overline{j}})$ to be an $|\overline{I}| \times |\overline{I}|$ matrix whose entries are given by 
\begin{eqnarray} \label{def:quotient-matrix}
\overline{b}_{\overline{i}, \overline{j}} = \sum_{\ell \in \overline{i}} b_{\ell, j},
\end{eqnarray}
where $j$ is fixed representative of $\bar j$.
It is easy to show that the definition of quotient matrix does not depend on the choice of representative from the orbit $\overline{j}$.  Furthermore, one may verify that if $B$ is skew-symmetric and $D$ is the $|\overline{I}| \times |\overline{I}|$ diagonal matrix with entries $d_{\overline{i}} = | \mbox{stab}_{G}(i) |$, where $\mbox{stab}_{G}(i)$ is the stabilizer of $i$ under the action of $G$, then $D\overline{B}$ is skew-symmetric.  That is, $\overline{B}$ is skew-symmetrizable.

\begin{definition}  Let $\tilde{B} = (b_{ij})$ be a $2n \times n$ matrix with automorphism group $G$.  We say that $\tilde{B}$ is \emph{admissible} if $b_{ij} = 0$ whenever $i, j$ are in the same $G$-orbit of $I$.  (Equivalently, if $B$ is the principal part of $\tilde{B}$, then admissibility means that there is no arrow between $i$ and $j$ in $Q(B)$ whenever $i, j$ are in the same $G$-orbit.)   
\end{definition}

For a $G$-orbit $\Omega = \{ i_{1}, \ldots, i_{s} \}$ of $I$, let $\mu_{\Omega} = \mu_{i_{s}} \ldots \mu_{i_{1}}$ be a composition of mutations; we call $\mu_{\Omega}$ an \emph{orbit mutation}.  By using the same argument as in \cite{dupont}, it is not difficult to check that orbit mutations are well-defined, i.e. independent of the order of the elements in $\Omega$.

Let the variables for the $F$-polynomials corresponding to $\overline{B^{0}}$ be given by $u_{\overline{i}}$ for $\overline{i} \in \overline{I}$.  Define the \emph{projection} $\pi$ to be the 
homomorphism of algebras given by 
\begin{eqnarray}
\pi: \ZZ[u_{i} : i \in I] & \rightarrow & \ZZ[u_{\overline{i}} : \overline{i} \in \overline{I}] \\
u_{i} & \mapsto & u_{\overline{i}}.
\end{eqnarray}
Also, for $\gg = (g_{1}, \ldots, g_{n}) \in \ZZ^{n}$, define the \emph{quotient} of $\gg$ to be 
\begin{eqnarray}
\overline{\gg} = \left(\sum_{\ell \in \overline{j_{1}}} g_{\ell}, \ldots, \sum_{\ell \in \overline{j_{r}}} g_{\ell} \right) \in \ZZ^{|\overline{I}|}
\end{eqnarray}

Let $\overline{\mathbb{T}_{n}}$ be the $|\overline{I}|$-regular tree such that the $|\overline{I}|$ edges emanating from a vertex are labeled by distinct elements of $\overline{I}$; write $\overline{t}_{0}$ for the initial vertex in $\overline{\mathbb{T}_{n}}$. 

Theorem \ref{thm:fpoly-proj} says that certain $F$-polynomials and $\textbf{g}$-vectors corresponding to the initial matrix $\overline{B^{0}}$ can be obtained as ``projections" or ``quotients" of $F$-polynomials or $\textbf{g}$-vectors corresponding to the initial matrix $B^{0}$. 

\begin{theorem} \label{thm:fpoly-proj} Suppose $\Omega_{1}, \ldots, \Omega_{s}$ are $G$-orbits with $\Omega_{i} = \overline{k_{i}}$.  For $p \in [1, s]$, let
\begin{eqnarray}
B^{p} = (b^{p}_{ij}) = \mu_{\Omega_{p}}\ldots \mu_{\Omega_{1}}(B^{0}).
\end{eqnarray}
Suppose that $k_{p}$ is a sink or source in $Q(B^{p - 1})$ for each $p = 1, \ldots, s$.  Let $t \in \mathbb{T}_{n}$ be the vertex obtained by a path from $t_{0}$ whose edges are labeled by all of the elements of $\Omega_{s}$, followed by all the edges in $\Omega_{s - 1}$, and so on, to finally the edges in $\Omega_{1}$.  Let $\overline{t} \in \overline{\mathbb{T}_{n}}$ be the path obtained from the initial vertex $\overline{t}_{0}$ via the edges $\overline{k_{s}}, \ldots, \overline{k_{1}}$.   Let $j \in [1, n]$.

Then
\begin{eqnarray}
F^{\overline{B^{0}}; \overline{t}_{0}}_{\overline{j}; \overline{t}} = \pi(F^{B^{0}; t_{0}}_{j; t}) \\
\gg_{\overline{j}; \overline{t}}^{\overline{B^{0}}; \overline{t}_{0}} = \overline{\gg_{j; t}^{B^{0}; t_{0}}}
\end{eqnarray}
\end{theorem}
\begin{proof}  To prove the theorem, we need to first consider the cluster algebras $\prinA(B^{0}; t_{0})$ and $\prinA(\overline{B^{0}}; \overline{t}_{0})$, and show that the seed at $\overline{t}$ in $\prinA(\overline{B^{0}}; \overline{t}_{0})$ is the ``projection" of the seed at $t$ in $\prinA(B^{0}; t_{0})$.   In \cite{dupont}, the seeds of the coefficient-free cluster algebras with initial exchange matrices $B^{0}$ and $\overline{B^{0}}$ were related.    Some of the results proven in \cite{dupont} easily carry over to the current sitation with principal coefficients (namely, Lemma \ref{lemma:g-inv} and Lemma \ref{lemma:proj-mut}), so the proofs are omitted here.  However, for finishing the proof of the result about the projection of the seeds, some stronger conditions are needed (see Proposition \ref{prop:projection}). 

Let the initial seed of $\prinA(B^{0}, t_{0})$ be given by $(\tilde{\textbf{x}}, \tilde{B}^{0})$, where
\begin{eqnarray}
\tilde{B}^{0} & = & \left( \begin{array}{c} B^{0} \\ I_{n} \end{array} \right) \\    
\tilde{\textbf{x}} & = & (x_{1}, \ldots, x_{n}, y_{1}, \ldots, y_{n}).
\end{eqnarray}

By abuse of notation, let the \emph{projection} $\pi$ be the the
homomorphism of algebras given by 
\begin{eqnarray}
\pi: \ZZ[x_{i}^{\pm 1}, y_{i}^{\pm 1} : i \in I] & \rightarrow & \ZZ[x_{\overline{i}}^{\pm 1}, y_{\overline{i}}^{\pm 1} : \overline{i} \in \overline{I}] \\
x_{i} & \mapsto & x_{\overline{i}} \\
y_{i} & \mapsto & y_{\overline{i}}
\end{eqnarray}
We also apply $\pi$ to ordered $2n$-tuples $\tilde{\textbf{v}} = (v_{1}, \ldots, v_{n}, y_{1}, \ldots, y_{n})$ (with each $v_{j}$ in $\ZZ[x_{i}^{\pm 1}, y_{i}^{\pm 1} : i \in I]$):
\begin{eqnarray}
\pi(\tilde{\textbf{v}}) = (\pi(v_{j_{1}}), \ldots, \pi(v_{j_{r}}), y_{\overline{j_{1}}}, \ldots, y_{\overline{j_{t}}}).
\end{eqnarray}
Note that $\pi(\tilde{\textbf{v}})$ is a $2|\overline{I}|$-tuple of elements from $\ZZ[x_{\overline{i}}^{\pm 1}, y_{\overline{i}}^{\pm 1} : \overline{i} \in \overline{I}]$.    

Let $\prinA(\overline{B^{0}}, \overline{t}_{0})$ be the cluster algebra with principal coefficients such that the initial extended cluster is given by 
\begin{eqnarray}
\pi(\tilde{\textbf{x}}) = (x_{\overline{j_{1}}}, \ldots, x_{\overline{j_{r}}}, y_{\overline{j_{1}}}, \ldots, y_{\overline{j_{r}}}).
\end{eqnarray}

By the Laurent phenomenon, 
\begin{eqnarray}
\prinA(B^{0}, t_{0}) \subset \ZZ[x^{\pm 1}_{i}, y_{i}^{\pm 1} : i \in I] \\
\prinA(\overline{B^{0}}, \overline{t}_{0}) \subset  \ZZ[x_{\overline{i}}^{\pm 1}, y_{\overline{i}}^{\pm 1} : \overline{i} \in \overline{I}]
\end{eqnarray}

If $\tilde{B} = (b_{ij}) \in M_{2n, n}(\ZZ)$, then define the quotient matrix $\overline{\tilde{B}} = (\overline{b}_{\overline{i}, \overline{j}})$ using (\ref{def:quotient-matrix}).   (Equivalently, if $\tilde{B} = \left( \begin{array}{c} B \\ C \end{array} \right)$, then $\overline{\tilde{B}} = \left( \begin{array}{c} \overline{B} \\ \overline{C} \end{array} \right)$).  The first $|\overline{I}|$ rows of this matrix are indexed by $\overline{I}$, the next $n$ rows are indexed by $\overline{I}' = \{ \overline{j_{1}} + n, \ldots, \overline{j_{r}}+n \}$ (i.e., the $G$-orbits of $[n + 1, 2n]$), and the columns are indexed by $\overline{I}$.

Define an action of $G$ on  $\ZZ[x^{\pm 1}_{i}, y_{i}^{\pm 1} : i \in I]$ by
\begin{eqnarray}
\sigma x_{i} =  x_{\sigma i} \\
\sigma y_{i} = y_{\sigma i}
\end{eqnarray}
\begin{definition}
A seed $(\tilde{\textbf{v}} = (v_{1}, \ldots, v_{n}, y_{1}, \ldots, y_{n}), \tilde{B})$ in $\prinA(B^{0}, t_{0})$ is \emph{$G$-invariant} if for every $\sigma \in G$, we have $\sigma \tilde{B} = \tilde{B}$ and $\sigma v_{i} = v_{\sigma i}$.
\end{definition}

\begin{lemma}  \label{lemma:g-inv} Let $\Omega$ be a $G$-orbit of $I$,  and let $(\tilde{\textbf{v}}, \tilde{B})$ a seed in $\prinA(B^{0}; t_{0})$ such that $\tilde{B}$ is admissible.  If $(\tilde{\textbf{v}}, \tilde{B})$ is a $G$-invariant seed, then so is $(\tilde{\textbf{v}}', \tilde{B}') =  \mu_{\Omega}(\tilde{\textbf{v}}, \tilde{B})$.
\end{lemma}

\begin{definition}  Let $\tilde{B} = (b_{ij})$ be a $2n \times n$ matrix with automorphism group $G$.  We say that $\tilde{B}$ is \emph{strongly admissible} if it is admissible and satisfies the following conditions:
\begin{itemize}
\item[(1)] For $i, j \in [1, 2n]$ in the same orbit and $\ell \in [1, n]$, $b_{i\ell}, b_{j\ell}$ are both nonnegative or both nonpositive.  
\item[(2)] For each $i, j \in [1, n]$ in the same orbit and $\ell \in [1, 2n]$, $b_{\ell i}, b_{\ell j}$ are both nonnegative or both nonpositive.\end{itemize}
\end{definition}

\begin{lemma}  \label{lemma:proj-mut} Let $\Omega = \overline{k}$ be a $G$-orbit of $I$, and let $(\tilde{\textbf{v}}, \tilde{B})$ be a $G$-invariant seed of $\prinA(B^{0}, t_{0})$ such that $\tilde{B}$ is strongly admissible.  Put $(\tilde{v}', \tilde{B}') = \mu_{\Omega}(\tilde{v}, \tilde{B})$.  Then 
\begin{eqnarray}
(\pi(\tilde{v}'), \overline{\tilde{B}'}) = \mu_{\overline{k}}(\pi(\tilde{\textbf{v}}), \overline{\tilde{B}}),
\end{eqnarray}
i.e., projection commutes with orbit mutation when applied to $(\tilde{\textbf{v}}, \tilde{B})$.
\end{lemma}

\begin{proposition}  \label{prop:projection} Suppose $\Omega_{1}, \ldots, \Omega_{s}$ are $G$-orbits with $\Omega_{i} = \overline{k_{i}}$.  Let 
\begin{eqnarray}
\tilde{B}^{p} = (b^{p}_{ij}) = \mu_{\Omega_{p}}\ldots \mu_{\Omega_{1}}(\tilde{B}^{0}),
\end{eqnarray}
and let $B^{p}$ be the principal part of $\tilde{B}^{p}$.   Suppose that $k_{p}$ is a sink or source in $Q(B^{p - 1})$ for each $p = 1, \ldots, s$.  Put 
\begin{eqnarray}
(\tilde{\textbf{x}}', \tilde{B}') = \mu_{\Omega_{s}} \ldots \mu_{\Omega_{1}}(\tilde{\textbf{x}}, \tilde{B}^{0}).
\end{eqnarray}
Then each $\tilde{B}^{p}$ is strongly admissible, and 
\begin{eqnarray}
(\pi(\tilde{\textbf{x}}'), \overline{\tilde{B}'}) =  \mu_{\overline{k_{1}}}\ldots \mu_{\overline{k_{s}}}(\pi(\tilde{\textbf{x}}), \overline{\tilde{B}^{0}}).
\end{eqnarray}
Thus, with the notation of Theorem \ref{thm:fpoly-proj}, the seed at $\overline{t}$ in $\prinA(\overline{B^{0}}, \overline{t}_{0})$ can be obtained by applying the projection to the seed $(\tilde{\textbf{x}}', \tilde{B}')$ in $\prinA(B^{0}; t_{0})$.
\end{proposition}
\begin{proof}
We only need to prove that each $\tilde{B}^{p}$ is strongly admissible.  Once this is done, the proposition follows from Lemmas \ref{lemma:g-inv} and \ref{lemma:proj-mut} by induction on $s$, the number of orbit mutations.

First, we check that $Q(B^{p})$ is acyclic.  If $k_{p}$ is a sink in $Q(B^{p - 1})$, then $b_{i, k_{p}}^{p - 1} \leq 0$ and $b_{k_{p}, j}^{p - 1} \geq 0$ for all $i, j \in [1, n]$.  If $k_{p}$ is a source in $Q(B^{p - 1})$, then $b_{i, k_{p}}^{p - 1} \geq 0$ and $b_{k_{p}, j}^{p - 1} \leq 0$ for all $i, j \in [1, n]$.  It follows that 
\begin{eqnarray}
b_{ij}^{p} = \left\{ \begin{array}{ll} 
                         -b_{ij}^{p - 1} & \mbox{ if } i = k_{p} \mbox{ or } j = k_{p} \\
                          & \\
                         b_{ij}^{p - 1} & \mbox{ otherwise.}
             \end{array} \right.
\end{eqnarray}
That is, $Q(B^{p})$ is obtained from $Q(B^{p - 1})$ by reversing the arrows which contain the vertex $k_{p}$.  By an inductive argument, it is easy to show that $Q(B^{p})$ is acyclic.  

Let $i, j \in I$ such that $i, j$ are in the same $G$-orbit.  Then $j = \sigma i$ for some $\sigma \in G$.   Since $G$ is a finite group, there exists some $r > 0$ such that $\sigma^{r} = 1$.

If $i \rightarrow j$ in $Q(B^{p})$, then the $G$-invariance of $\tilde{B}$ implies that the following cycle occurs in $Q(B^{p})$:
\begin{eqnarray}  
i \rightarrow \sigma i \rightarrow \sigma^{2} i \rightarrow \cdots \rightarrow \sigma^{r}i = i.
\end{eqnarray}
Thus, $i \nrightarrow j$ in $Q(B^{p})$, which proves that $B^{p}$ is admissible.

Let $\ell \in [1, n]$.  If $i \rightarrow \ell \rightarrow j$ in $Q(B^{p})$, then the following cycle occurs in $Q(B^{p})$:
\begin{eqnarray}
i \rightarrow \ell \rightarrow \sigma i \rightarrow \sigma \ell \rightarrow \sigma^{2} i \rightarrow \sigma^{2} \ell \rightarrow \cdots \rightarrow \sigma^{r - 1} \ell \rightarrow \sigma^{r} i = i.
\end{eqnarray}
Thus, the path $i \rightarrow \ell \rightarrow j$ cannot occur in $Q(B^{p})$, and by similar reasoning, $j \rightarrow \ell \rightarrow i$ cannot occur in $Q(B^{p})$.  This forces $b_{i\ell}^{p}, b_{j\ell}^{p}$ to be both nonnegative or both nonpositive, and the same thing is true of $b_{\ell i}^{p}, b_{\ell j}^{p}$.

From \cite{qps2}, it is known that every $F$-polynomial $F^{B^{0}; t_{0}}_{j; t}$ has constant term 1 (under the assumption that $B^{0}$ is skew-symmetric).  From \cite[Proposition 5.6]{coefficients}, it follows that the entries in the bottom $n$ rows of the $\ell$th column of $\tilde{B}^{p}$ are all nonnegative or nonpositive.  This shows condition (1) for the strong admissibility of $\tilde{B}^{p}$ is true.  

Finally, suppose that $\ell \in [n + 1, 2n]$.  Then $b^{p}_{\ell j} = b^{p}_{\ell, \sigma i} = b^{p}_{\sigma^{-1}\ell, i}$.  Since $b^{p}_{\sigma^{-1}\ell, i}$, $b^{p}_{\ell i}$ are both nonnegative or both nonpositive, this proves condition (2) for strong admissibility. 
\end{proof}

\begin{proposition} \label{prop:projection2} Let $\Omega_{i}$ be $G$-orbits, and let $t \in \mathbb{T}_{n}$, $\overline{t} \in  \overline{\mathbb{T}_{n}}$, with the same assumptions as in Theorem \ref{thm:fpoly-proj}.    Let $j \in [1, n]$, and write $x_{\overline{j}; \overline{t}}$, $x_{j; t}$ for the cluster variables in  $\prinA(B^{0}; t_{0})$ and  $\prinA(\overline{B}^{0}; \overline{t}_{0})$, respectively.  Then
\begin{eqnarray}
x_{\overline{j}; \overline{t}} = \pi(x_{j; t}).
\end{eqnarray}
\end{proposition}
\begin{proof} The proposition follows immediately from Proposition \ref{prop:projection}.
\end{proof}

Finally, the proof of Theorem \ref{thm:fpoly-proj} may be concluded.  By (\ref{eq:cluster-var-fpoly}),
\begin{eqnarray}
x_{j; t} = F_{j; t}^{B^{0}; t_{0}}(\hat{y}_{1}, \ldots, \hat{y}_{n})x_{1}^{g_{1}}\ldots x_{n}^{g_{n}}, 
\end{eqnarray}
where $\gg_{j; t}^{B^{0}; t_{0}} = (g_{1}, \ldots, g_{n})$.  Applying $\pi$ to both sides, we get
\begin{eqnarray}
x_{\overline{j}; \overline{t}} = \pi(F_{j; t}^{B^{0}; t_{0}}(\hat{y}_{1}, \ldots, \hat{y}_{n}))x_{\overline{1}}^{g_{1}}\ldots x_{\overline{n}}^{g_{n}}.
\end{eqnarray}
Let 
\begin{eqnarray}
\hat{\hat{y}}_{j} = y_{\overline{j}}\prod_{\overline{i} \in \overline{I}} x_{\overline{i}}^{\overline{b}^{0}_{\overline{i}, \overline{j}}}.
\end{eqnarray}
These are the $\hat{y}_{j}$ elements corresponding to $\prinA(\overline{B^{0}}; \overline{t}_{0})$ (see (\ref{def:yhat})).  It is straightforward to check that $\pi(\hat{y}_{j}) = \hat{\hat{y_{j}}}$ for $j \in [1, n]$.  Thus,
\begin{eqnarray}
x_{\overline{j}; \overline{t}} = \pi(F_{j; t}^{B^{0}; t_{0}})(\hat{\hat{y}}_{1}, \ldots, \hat{\hat{y}}_{n})\prod_{\overline{j} \in \overline{I}}x_{\overline{j}}^{\sum_{\ell \in \overline{j}} g_{\ell}}.
\end{eqnarray}
From \cite{qps2}, it is known that $F_{j; t}^{B^{0}; t_{0}}$ has constant term 1, so the same is true of $\pi(F_{j; t}^{B^{0}; t_{0}})$.  The theorem follows from Proposition \ref{thm:cluster-var-formula}.
\end{proof}

The following result is proven in \cite{yangzel}:

\begin{theorem} \label{thm:sinks} If $B^{0}$ is of finite type, then every cluster variable in $\prinA(B^{0}; t_{0})$ may be obtained from the initial cluster by a sequence of mutations, each of which is either a sink or seed mutation. 
\end{theorem}

\subsection{Type $\mbox{C}_{n}$}
In this subsection, we use the results of the previous subsection to finish the proof of Theorem \ref{thm:fpoly-classical}.  Let $\overline{B^{0}} = (\overline{b}_{ij})$ be an $n \times n$ acyclic exchange matrix of type $\mbox{C}_{n}$.  Let $G$ be the subgroup of $S_{2n - 1}$ generated by the involution $\sigma$, where $\sigma(i) = 2n - i$ for all $i \in [1, 2n - 1]$.  Then $G$ is a group of automorphisms for the matrix $B^{0} = (b_{ij})$ of type $\mbox{A}_{2n - 1}$ whose entries are given below: 
\begin{eqnarray}
b_{ij} & = & b_{2n - i, 2n - j} = \overline{b}_{ij} \mbox{ if } 1 \leq i, j \leq n, (i, j) \neq (n - 1, n) \\
b_{n - 1, n} & = & b_{n + 1, n} = \mbox{sgn}(\overline{b}_{n - 1, n}) \in \{ +1, -1 \} \\
b_{ij} & = & b_{2n - i, 2n - j} = 0 \mbox{ if } n + 1 \leq i \leq 2n - 1 \mbox{ and } 1 \leq j \leq n - 1 
\end{eqnarray}
The matrix $\overline{B^{0}}$ is indeed the quotient matrix of $B^{0}$, which justifies the choice of notation.   

\begin{lemma} \label{lemma:typebc-proj} Suppose the cluster variable $x \in \prinA(B^{0})$ has denominator vector $\dd'$.  Then  $\pi(x)$ is a cluster variable in $\prinA(\overline{B^0})$ with denominator vector $\overline{\dd'}$.  
\end{lemma}
\begin{proof} The fact that $\pi(x)$ is a cluster variable in $\prinA(\overline{B^{0}})$ follows from Proposition \ref{prop:projection2}  and Theorem \ref{thm:sinks}.  Define the function $\psi: \Phi_{+}(B^{0}) \rightarrow \Phi_{+}(\overline{B^{0}})$ in the following way: if $x \in \prinA(B^{0})$ has denominator vector $\dd' \in \Phi_{+}(B^{0})$, then let $\psi(\dd')$ be the denominator vector of the cluster variable $\pi(x)$.   Thus, we must show that 
\begin{eqnarray} \label{eqn:overlinedd}
\psi(\dd') = \overline{\dd'}.
\end{eqnarray}
Clearly, $\psi(\dd') \leq \overline{\dd'}$ and $\psi(\dd') = \psi(\sigma(\dd'))$.   It is easy to verify that the number of $\sigma$-orbits of $\Phi_{+}(B^{0})$ is equal to $|\Phi_{+}(\overline{B^{0}})|$.  It follows that $\psi(\dd') = \psi(\dd'')$ if and only if $\dd' = \dd''$ or $\dd' = \sigma(\dd'')$.

We use the natural ordering on $\Phi_{+}(B^{0})$ to prove (\ref{eqn:overlinedd}).   The least elements of $\Phi_{+}(B^{0})$ are $\ee_{i}$ $(i \in [1, n])$.  Since $\psi(\ee_{i}) \leq \ee_{i}$, equality must hold.  Now suppose that $\dd' \in \Phi_{+}(B^{0})$ such that $\psi(\dd'') = \overline{\dd''}$ for all $\dd'' \in \Phi_{+}(B^{0})$ satisfying $\dd'' < \dd'$.    If $\dd' = \ee_{i} + \cdots + \ee_{j}$ for some $1 \leq i < j \leq n$, then any denominator vector in  $\Phi_{+}(\overline{B^{0}})$ that is less than $\dd'$ is of the form $\ee_{i'} + \cdots + \ee_{j'}$, where $i \leq i' \leq j' \leq j$, and this is equal to $\psi(\ee_{i'} + \cdots + \ee_{j'})$ by assumption.  Thus, (\ref{eqn:overlinedd}) must hold in this case.  If $\dd' = \ee_{i} + \cdots + \ee_{2n - j}$ for some $1 \leq i < j \leq n - 1$, then any denominator vector in $\Phi_{+}(\overline{B^{0}})$ that is less than $\overline{\dd'} = \ee_{i} + \cdots + 2\ee_{j} + \cdots + 2\ee_{n - 1} + \ee_{n}$ is either of the form $\ee_{i'} + \cdots + \ee_{j'}$ or of the form $\ee_{i'} + \cdots + 2\ee_{j'} + \cdots + 2\ee_{n - 1} + \ee_{n}$, where $i \leq i'$ and $j' \geq j$.  The former expression is equal to $\psi(\ee_{i'} + \cdots + \ee_{j'})$, while the latter is equal to $\psi(\ee_{i'} + \cdots + \ee_{2n - j'})$.  Again, (\ref{eqn:overlinedd}) holds.  The remaining denominator vectors in $\Phi_{+}(B^{0})$ are obtained from the ones already mentioned by applying $\sigma$, so the result again follows by using the fact that $\psi(\sigma(\dd')) = \psi(\dd')$ for $\dd' \in \Phi_{+}(B^{0})$.
\end{proof}

To calculuate $F^{\overline{B^{0}}; \overline{t}_{0}}_{\dd}$ for $\dd \in \Phi_{+}(\overline{B^{0}})$, we need to use the formula for $F_{\dd'}^{B^{0}; t_{0}}$ for $\dd' \in \Phi_{+}(B^{0})$ such that $\overline{\dd'} = \dd$, and then apply the projection $\pi$.  In the current setting, $\pi(u_{i}) = \pi(u_{2n - i}) = u_{i}$ for $i \in [1,n]$.  

Let $\dd = (d_{1}, \ldots, d_{n}) \in \Phi_{+}(B^{0})$.  If $\dd = \sum_{i = p}^{r} \ee_{i}$ for some $1 \leq p \leq r \leq n$, then we can let $\dd' = \dd$.  In this case, 
\begin{eqnarray}
F^{B^{0}; t_{0}}_{\dd} = \pi(F_{\dd}^{B^{0}; t_{0}}).
\end{eqnarray}
It is easy to verify the theorem works in this case.

Now suppose that $\dd = \sum_{i = p}^{n} \ee_{i} + \sum_{i = r}^{n - 1} \ee_{i}$ for some $1 \leq p \leq r \leq n - 1$.  Then we can take $\dd' = \sum_{i = p}^{2n - r} \ee_{i} = (d_{1}', \ldots, d_{2n - 1}').$  Let $\ee = (e_{1}, \ldots, e_{n})$ such that $u_{1}^{e_{1}}\ldots u_{n}^{e_{n}}$ occurs with nonzero coefficient in $\pi(F_{\dd'}^{B^{0}; t_{0}})$.    Then there exists $(e_{1}', \ldots, e_{2n - 1}') \in \ZZ^{2n - 1}$ such that $u_{1}^{e_{1}'}\ldots u_{n}^{e_{2n - 1}'}$ occurs with nonzero coefficient in $F_{\dd'}^{B^{0}; t_{0}}$ and $\pi(u_{1}^{e_{1}'}\ldots u_{n}^{e_{2n - 1}'}) = u_{1}^{e_{1}}\ldots u_{n}^{e_{n}}$ (i.e., $e_{i} = e_{i}' + e_{2n - i}'$ for $i \in [1, n - 1]$ and $e_{n} = e_{n}'$). 

First, we will check that properties (1)-(4) in Theorem \ref{thm:fpoly-classical} hold for $u_{1}^{e_{1}}\ldots u_{n}^{e_{n}}$.   (1) is immediate from the fact that (1) holds for $F_{\dd'}^{B^{0}; t_{0}}$.   

Let $Q^{0} = Q(B^{0})$ and $\overline{Q^{0}} = Q(\overline{B^{0}})$.   Observe that $i \rightarrow j$ in $\overline{Q^{0}}$ if and only if $i \rightarrow j$ and $2n - i \rightarrow 2n - j$ in $Q^{0}$.  Also, if $i \rightarrow j$ in $Q^{0}$ and $d_{i}' \leq d_{j}'$, then the fact that arrows in $Q^{0}$ are acceptable implies that $e_{i}' \leq e_{j}'$.

For (4), part (a), if $n \rightarrow n - 1$ in $\overline{Q^{0}}$ and $e_{n} = 1$, then $e_{n}' = 1$ implies that $e_{n - 1}' = 1$, $e_{n + 1}' = 1$.  Thus, $e_{n - 1} = e_{n - 1}' + e_{n + 1}' = 2$.  For part (b), suppose that $n - 1 \rightarrow n$ in $\overline{Q^{0}}$ and $e_{n - 1} \geq 1$.  Then $e_{n - 1}' = 1$ or $e_{n + 1}' = 1$.  Since $e_{n}' \geq e_{n - 1}'$, $e_{n}' \geq e_{n + 1}'$, part (b) follows.

For (3), observe that (4) implies that the arrow connecting $n - 1$ and $n$ in $\overline{Q^{0}}$ cannot be critical, so there is only one possible critical arrow which may occur between $r - 1$ and $r$.  

For (2), let $i \rightarrow j$ in $\overline{Q^{0}}$.  If $i, j$ are both in $[1, p - 1]$, both in $[p, r - 1]$, or both in $[r, n - 1]$, then $d_{i}' = d_{j}'$, $d_{2n - i}' = d_{2n - j}'$, and $d_{i} = d_{j}$.  Since $i \rightarrow j$ and $2n - i \rightarrow 2n - j$ in $Q^{0}$, it follows that $e_{i}' - e_{j}' \leq 0$ and $e_{2n - i}' - e_{2n - j}' \leq 0$, which means that $e_{i} - e_{j} \leq 0 = [d_{i} - d_{j}]_{+}$.   If $i = p - 1$ or $j = p - 1$, then use the fact that if $d_{i} = 0$ or $d_{j} = 0$, then condition (2) is automatically satisfied.   Now suppose that $r > p$ and $\{i, j \} = \{ r - 1, r \}$.  Note that $d_{r - 1}' = 1$, $d'_{2n - r + 1} = 0$, $d_{r}' = 1$, $d'_{2n - r} = 1$, $d_{r - 1} = 1$, $d_{r} = 2$.  If $r - 1 \rightarrow r$ in $\overline{Q}^{0}$, then $e'_{r - 1} - e_{r}' \leq 0$, $e_{2n - r + 1}' - e_{2n - r}' \leq 0$, so $e_{r - 1} - e_{r} \leq 0$.  If $r \rightarrow r - 1$ in $\overline{Q}^{0}$, then $e_{r}' - e_{r - 1}' \leq 0$, $e_{2n - r}' - e_{2n - r + 1}' \leq 1$, so $e_{r} - e_{r - 1} \leq 1$.  Finally, if $\{ i, j \} = \{ n - 1, n \}$, then (4) implies that $i \rightarrow j$ is acceptable.  

Now assume that $\ee = (e_{1}, \ldots, e_{n}) \in \ZZ^{n}$ satisfies conditions (1)-(4).   Since all the coefficients in $F^{B^{0}; t_{0}}_{\dd'}$ are either 0 or 1, we need to construct all $\ee' = (e_{1}', \ldots, e_{2n - 1}') \in \ZZ^{n}_{\geq 0}$ such that $u_{1}^{e_{1}'}\ldots u_{n}^{e_{n}'}$ occurs with nonzero coefficient in $F^{B^{0}; t_{0}}_{\dd'}$, $e_{i} = e_{i}' + e_{2n - i}'$ for $i \in [1, n - 1]$, and $e_{n} = e_{n}'$.  

If $e_{i} = 0$, then this forces $e_{i}' = e_{2n - i}' = 0$.  If $e_{i} = d_{i}$, then we must have $e_{i}' = d_{i}'$ and $e_{2n - i}' = d_{2n - i}'$.  Thus, if $(d_{i}, e_{i}) \neq (2, 1)$, then there is only one possible choice for $e_{i}'$ and $e_{2n - i}'$.  

Let $S$ be the subgraph of $\overline{Q^{0}}$ induced by the vertex set $ \{ i \in [1, n] : (d_{i}, e_{i}) = (2, 1) \}$.  If $i, j$ are vertices in $S$ such that $i \rightarrow j$ in $\overline{Q^{0}}$, then $e_{i}' \leq e_{j}'$ and $e_{2n - i}' \leq e_{2n - j}'$.  It follows that for each component $C$ of the graph $S$, exactly one of the following possibilities holds:  either $e_{i}' = 1$ and $e_{2n - i}' = 0$ for all vertices $i$ in $C$, or $e_{i}' = 0$ and $e_{2n - i}' = 1$ for all vertices $i$ in $C$.  Thus, there are two possible choices for each such $C$, except in the case where $r$ is a vertex of $C$ and $(r - 1, r)$ is a critical pair.  In this case, we must have $e_{r}' = 1$ and $e_{2n - r}' = 0$ if $r - 1 \rightarrow r$ in $\overline{Q^{0}}$, or $e_{r}' = 0$ and $e_{2n - r}' = 1$ if $r \rightarrow r - 1$ in $\overline{Q^{0}}$.  

Now suppose that $\textbf{e}' = (e_{1}', \ldots, e_{2n - 1}')$ is chosen as above.  It remains to verify that all arrows $i \rightarrow j$ in $Q^{0}$ are acceptable with respect to $\textbf{e}'$ and $\textbf{d}'$.    If $i, j \in [r, n - 1]$ and both $i, j$ are vertices in $S$, note that $i \rightarrow j$ and $2n - i \rightarrow 2n - j$ are both acceptable by construction

Let $i \rightarrow j$ in $\overline{Q^{0}}$.  We will prove that $i \rightarrow j$ and $2n - i \rightarrow 2n - j$ are acceptable arrows in $Q^{0}$.   If $i, j \in [1, p]$, then Proposition \ref{prop:acceptable} implies that these arrows are acceptable.  If $i, j \in [p, r - 1]$, then $e_{2n - i}' = e_{2n - j}' = 0$, so $2n - i \rightarrow 2n - j$ is acceptable.  Also, $e_{i} = e_{i}'$, $e_{j} = e_{j}' $, so the fact that $i \rightarrow j$ is acceptable in $\overline{Q^{0}}$ implies that the arrow is acceptable in $Q^{0}$.

Suppose that $i$ is in $[r, n - 1]$ but is not a vertex of $S$.   If $e_{i} = 0$, then $e_{i}' = e_{2n - i}' = 0$. If $e_{i} = 2$, then $e_{j} = d_{j}$, so $e_{j}' = d_{j}'$ and $e_{2n - j}' = d_{2n - j}'$.  In either case, Proposition \ref{prop:acceptable} implies that $i \rightarrow j$ and $2n - i \rightarrow 2n - j$ are acceptable arrows in $Q^{0}$.  The argument in the case that $j$ is in $[r, n - 1]$ but not in $S$ is similar.

By Proposition \ref{prop:acceptable}, the arrow between $2n - r + 1$ and $2n - r$ is acceptable since $d_{2n - r + 1}' = 0$.  Next, we need to consider the arrow in $Q^{0}$ between the vertices $r - 1$ and $r$.   If $(r - 1, r)$ is a critical pair in $\overline{Q^{0}}$, note that the construction implies that $r - 1 \rightarrow r$ or $r \rightarrow r - 1$ in $Q^{0}$.    If $r - 1 \rightarrow r$ in $Q^{0}$, $e_{r} = 1$, and $(r - 1, r)$ is not a critical pair, then $e_{r - 1} = 0$, so $e_{r - 1}' = 0$.  This means that $r - 1 \rightarrow r$ in $Q^{0}$ is acceptable.  If $r \rightarrow r - 1$ in $Q^{0}$, $e_{r} = 1$ and $(r - 1, r)$ is not critical, then $e_{r - 1} = 1 = e_{r - 1}' = d_{r - 1}'$, so $r \rightarrow r - 1$ is acceptable in $Q^{0}$.  

Finally, we consider the arrow in $\overline{Q^{0}}$ between $n - 1$ and $n$.    If $n \rightarrow n - 1$ in $Q^{0}$, then $e_{n - 1} = 1$, then (4) implies that $e_{n}' = 0$.  If $n - 1 \rightarrow n$ in $\overline{Q^{0}}$ and $e_{n - 1} = 1$, then (4) implies that $e_{n}' = 1 = d_{n}'$.  Thus, in either case, the arrow between $n - 1$ and $n$ and the arrow between $n + 1$ and $n$ are both acceptable in $Q^{0}$.  

This concludes the proof of Theorem \ref{thm:fpoly-classical} in type $\mbox{C}_{n}$.

\subsection{Type $\mbox{B}_{n}$}
In this subsection, we finish the proof of Theorem \ref{thm:fpoly-classical} using the same strategy as in the previous subsection.  Let $\overline{B^{0}} = (\overline{b}_{ij})$ be an $n \times n$ acyclic exchange matrix of type $\mbox{B}_{n}$.  Let $G$ be the subgroup of $S_{n + 1}$ generated by the involution $\sigma$, where $\sigma(i) = i$ for all $i \in [1, n - 1]$, $\sigma(n + 1) = n$, $\sigma(n) = n + 1$.  Then $G$ is a group of automorphisms for the matrix $B^{0} = (b_{ij})$ of type $\mbox{D}_{n + 1}$ whose entries are given below: 
\begin{eqnarray}
b_{ij} & = & \overline{b}_{ij} \mbox{ if } 1 \leq i, j \leq n \mbox{ and } (i, j) \neq (n, n - 1)  \\
b_{n, n - 1} & = & -b_{n - 1, n} = \mbox{sgn}(\overline{b}_{n - 1, n}) \in \{ +1, -1 \} \\
b_{n + 1, n - 1} & = & -b_{n - 1, n + 1} = \mbox{sgn}(\overline{b}_{n - 1, n}) \in \{ +1, -1 \} \\
b_{n + 1, j} & = & b_{j, n + 1} = 0 \mbox{ if } 1 \leq j \leq n + 1, j \neq n - 1.
\end{eqnarray}
The matrix $\overline{B^{0}}$ is indeed the quotient matrix of $B^{0}$, which justifies the choice of notation.  Write $Q^{0} = Q(B^{0})$ and $\overline{Q^{0}} = Q(\overline{B^{0}})$. 

In this case, it is easy to verify Lemma \ref{lemma:typebc-proj} holds using a similar argument.  Let $\dd \in \Phi_{+}$.  If $\dd = \sum_{i = 1}^{n} d_{i}\ee_{i}$ with each $d_{i} = 0$ or $1$, then we can take $\dd' = \dd$ and the same reasoning as in type $\mbox{C}_{n}$ holds here.  

Assume for the remainder of this subsection that $\dd = \sum_{i = p}^{r - 1} \ee_{i} + \sum_{i = r}^{n} 2\ee_{i}$ for some $1 \leq p < r \leq n$. Then we can take $\dd' = \sum_{i = p}^{r - 1} \ee_{i} + \sum_{i = r}^{n - 2} 2\ee_{i} + \ee_{n - 1} + \ee_{n}$.  Write $\dd = (d_{1}, \ldots, d_{n})$ and $\dd' = (d_{1}', \ldots, d_{n + 1}')$.  Then $d_{i} = d_{i}'$ for all $i \in [1, n - 1]$ and $d_{n} = d_{n}' + d_{n + 1}'$.  Let $(e_{1}', \ldots, e_{n + 1}') \in \ZZ^{n + 1}$ such that $u_{1}^{e_{1}'} \ldots u_{n + 1}^{e_{n + 1}'}$ occurs with nonzero coefficient in $F_{\dd'}^{B^{0}; t_{0}}$.  Let $\ee = (e_{1}, \ldots, e_{n}) \in \ZZ^{n}$ be given by $e_{i} = e_{i}'$ for $i \in [1, n - 1]$, $e_{n} = e_{n}' + e_{n + 1}'$.  Then $u_{1}^{e_{1}}\ldots u_{n}^{e_{n}}$ occurs with nonzero coefficient in $F_{\dd}^{\overline{B^{0}}; \overline{t}_{0}} = \pi(F_{\dd'}^{B^{0}; t_{0}})$, so we need to verify that $\ee$ satisfies conditions (1)-(3) and (5) in Theorem \ref{thm:fpoly-classical} with respect to $\overline{Q^{0}}$ and $\dd$.  (1) is immediate since $0 \leq e_{i}' \leq d_{i}'$ for all $i \in [1, n + 1]$.  The only possible critical pair is $(r - 1, r)$, so (3) also follows easily.  

For (2), let $i \rightarrow j$ be an arrow in $\overline{Q^{0}}$.  If $i, j \in [1, n - 1]$, then the fact that $i \rightarrow j$ is an acceptable arrow in $Q^{0}$ implies the same thing is true of $i \rightarrow j$ in $Q^{0}$.  

If $(i, j) = (n - 1, n)$, then the fact that $n - 1 \rightarrow n$ and $n - 1 \rightarrow n + 1$ are acceptable arrows in $Q^{0}$ implies that $e_{n - 1}' - e_{n}' \leq 1$ and $e_{n - 1}' - e_{n + 1}' \leq 1$.  If $e_{n - 1}' = 2$, the inequalities imply that $e_{n}' = e_{n + 1}' = 1$.  If $e_{n - 1}' = 1$, then $e_{n}' \geq 1$ or $e_{n + 1}' \geq 1$, since otherwise $n - 1 \rightarrow n$ and $n - 1 \rightarrow n$ are both critical arrows in $Q^{0}$.  In any case, it is clear that $e_{n - 1}' - e_{n}' - e_{n + 1}' \leq 0$, which means that $n - 1 \rightarrow n$ is an acceptable arrow in $\overline{Q^{0}}$.

If $(i, j) = (n, n - 1)$, then the fact that $n \rightarrow n - 1$ and $n + 1 \rightarrow n - 1$ are acceptable arrows in $Q^{0}$ means that $e_{n}' - e_{n - 1}' \leq 0$ and $e_{n + 1}' - e_{n - 1}' \leq 0$.  If $e_{n - 1}' = 0$, then the inequalities imply that $e_{n + 1}' = e_{n}' = 0$.  If $e_{n - 1}' = 1$, then $e_{n}' + e_{n + 1}' \leq 1$, since $(n - 1, n)$ and $(n - 1, n + 1)$ cannot both be critical pairs in $Q^{0}$.  If $e_{n - 1}' = 2$, then $e_{n}' \leq 1$ and $e_{n + 1}' \leq 1$, so $e_{n}' + e_{n + 1}' - e_{n - 1}' \leq 0$.    In any case, it holds that $i \rightarrow j$ is an acceptable arrow in $\overline{Q^{0}}$.

Let $S$ be the subgraph of $\overline{Q^{0}}$ induced by the set of vertices $\{ i \in [r, n] : e_{i} = 1 \}$, and let $S'$ be the subgraph of $Q^{0}$ induced by the set of vertices $\{ i \in [r, n - 1] : e_{i}' = 1 \}$.  Note that $S'$ is obtained from $S$ by deleting the vertex $n$ (if $n$ is in $S$) and any edges that contain $n$.

For condition (5), suppose that $S$ consists of a single component containing $n - 1$, and the arrow between $r - 1$ and $r$ is critical in $\overline{Q^{0}}$.   Then the arrow between $r - 1$ and $r$ is also critical in $Q^{0}$, and the equalities $e_{r} = \cdots = e_{n - 1} = 1$ and $e'_{r} = \cdots = e'_{n - 1} = 1$ hold.   Assume for the sake of contradiction that $e_{n} = 1$.  Then $e_{n}' = 1$ and $e_{n + 1}' = 0$, or $e_{n}' = 0$ and $e_{n + 1}' = 1$.  In either case, it is easy to check that one of the arrows between $n - 1$ and $n$ or between $n - 1$ and $n + 1$ in $Q^{0}$ is critical with respect to $(\dd', \ee')$, which contradicts (3) for $F_{\dd'}$.

Now suppose that $\ee = (e_{1}, \ldots, e_{n}) \in \ZZ^{n}$ satisfies conditions (1)-(3) and (5) in Theorem \ref{thm:fpoly-classical}.  We compute the coefficient of $u_{1}^{e_{1}}\ldots u_{n}^{e_{n}}$ in $\pi(F_{\dd'}^{B^{0}; t_{0}})$.  

If $e_{n} = 0$ or $e_{n} = 2$, then there is exactly one $(e_{1}', \ldots, e_{n + 1}') \in \ZZ^{n + 1}$ such that $u_{1}^{e_{1}'}\ldots u_{n + 1}^{e_{n + 1}'}$ occurs with nonzero coefficient in $F_{\dd'}^{B^{0}; t_{0}}$ and $\pi(u_{1}^{e_{1}'}\ldots u_{n + 1}^{e_{n + 1}'}) = u_{1}^{e_{1}}\ldots u_{n}^{e_{n}}$.  Note that $(n - 1, n)$ is not critical in $\overline{Q^{0}}$, and $(n - 1, n)$, $(n, n + 1)$ are not critical in $Q^{0}$.  Since $(r - 1, r)$ is critical in $Q^{0}$ if and only if it is critical in $\overline{Q^{0}}$, and the components of $S$ are the same as the components of $S'$, the theorem follows in this case from the type $\mbox{D}_{n}$ case.

Assume for the remainder of this proof that $e_{n} = 1$.  Then there are two possible $\ee' = (e_{1}', \ldots, e_{n + 1}')$ such that $u_{1}^{e_{1}'}\ldots u_{n + 1}^{e_{n + 1}'}$ occurs with nonzero coefficient in $F_{\dd'}^{B^{0}; t_{0}}$ and $\pi(u_{1}^{e_{1}'}\ldots u_{n + 1}^{e_{n + 1}'}) = u_{1}^{e_{1}}\ldots u_{n}^{e_{n}}$.  To be more precise, $e_{i}' = e_{i}$ for $i \in [1, n - 1]$, and either $e_{n}' = 1$ and $e_{n + 1}' = 0$, or $e_{n}' = 0$ and $e_{n + 1}' = 1$.  Note that the coefficient of $u_{1}^{e_{1}'}\ldots u_{n + 1}^{e_{n + 1}'}$ in $F_{\dd}^{B^{0}; t_{0}}$ for either $\ee'$ is equal to $2^{c}$, where $c$ is the number of components of $S'$, so we need to show that the number of components of $S'$ is $c + 1$.  

Suppose that $n - 1 \rightarrow n$ in $\overline{Q^{0}}$.  Then the fact that this arrow is acceptable implies that $e_{n - 1} = 0$ or $1$.  If $e_{n - 1} = 0$, then the arrows $n - 1 \rightarrow n$ and $n - 1 \rightarrow n + 1$ in $Q^{0}$ are not critical.  If $e_{n - 1} = 1$, then one of the arrows $n - 1 \rightarrow n$ or $n - 1 \rightarrow n + 1$ is critical in $Q^{0}$.  In either case, the components $C$ of $S$ for which $\nu(C) = 0$ are the components of $S'$ for which $\nu(C) = 0$, plus an additional component containing the vertex $n$.   By similar reasoning, the same statement may be proven when $n \rightarrow n - 1$ in $\overline{Q^{0}}$.  This finishes the proof of Theorem \ref{thm:fpoly-classical}.

\section{Quantum $F$-polynomials}

All of the definitions and results in this section can be found in \cite{quantumfpoly}.  For this section, fix an $n \times n$ skew-symmetrizable integer matrix $B^{0} = (b^{0}_{ij})$ at $t_{0}$, and let $D$ be a diagonal matrix with positive integer diagonal entries such that $DB^{0}$ is skew-symmetric.  
Let $\hat{\delta} = (\delta_{1}, \ldots, \delta_{n}) \in \ZZ^{n}$ be the vector consisting of the entries of the diagonal of $D$.   We will not explicitly define the quantum $F$-polynomials  $F_{j; t} = F_{j; t}^{B^{0}; D; t_{0}}$ ($j \in [1, n]$, $t \in \mathbb{T}_{n}$) but simply recall results about quantum $F$-polynomials  which will be used in the next section to compute them.  

 The variables for quantum $F$-polynomials will be $Z_{1}, \ldots, Z_{n}$, where the $Z_{i}$ satisfy the quasi-commutation relation

\begin{eqnarray}
Z_{i}Z_{j} = q^{\delta_{i}b^{0}_{ij}}Z_{j}Z_{i}
\end{eqnarray}
for all $i, j \in [1, n]$.  Note that the quantum $F$-polynomials will be contained in the ring
\begin{eqnarray} \mathcal{R} = \mathbb{Z}[q^{\pm \frac{1}{2}}, Z_{1}, \ldots, Z_{n}].
\end{eqnarray}  

Some basic facts about quantum $F$-polynomials are given in the next proposition.

\begin{proposition} \cite{quantumfpoly} \begin{enumerate}
   \item For $j \in [1, n]$, $F_{j; t_{0}} = 1$. 
   \item Let $t \in \mathbb{T}_{n}$ such that $t_{0} \frac{k}{\hspace{1cm}} t$.   Then $F_{k; t} = q^{\delta_{k}/2}Z_{k} + 1$.
   \item Let $t, t' \in \mathbb{T}_{n}$ such that $t \frac{k}{\hspace{1cm}} t'$.  Then $F_{j; t} = F_{j; t'}$ for $j \neq k$.
\end{enumerate}
\end{proposition}

For $\textbf{c} = (c_{1}, \ldots, c_{n}) \in \ZZ^{n}$, define 
\begin{eqnarray} \label{eqn:Z-monomial}
Z^{\textbf{c}} = q^{\frac{1}{2} \sum_{1 \leq i < j \leq n} \delta_{j}b^{0}_{ji}c_{i}c_{j}} Z_{1}^{c_{1}} \cdots Z_{n}^{c_{n}}.
\end{eqnarray}

 Let $G_{1}, \ldots, G_{p} \in \mathcal{R}$.  Write 
\begin{eqnarray}
\prod^{\rightarrow}_{i \in [1, p]} G_{i} = G_{1} \ldots G_{p}.
\end{eqnarray}
The elements $G_{1}, \ldots, G_{n}$ do not necessarily commute, so the product notation establishes a fixed order in which the elements are to be multiplied.

For $\aa = (a_{1}, \ldots, a_{n}), \bb = (b_{1}, \ldots, b_{n}), \cc = (c_{1}, \ldots, c_{n}) \in \ZZ^{n}$, define the notation 
\begin{eqnarray}
\aa \cdot \bb \cdot \cc = \sum_{i = 1}^{n} a_{i}b_{i}c_{i} \in \ZZ.
\end{eqnarray}

For $\textbf{a} = (a_{1}, \ldots, a_{n}) \in \mathbb{Z}^{n}$, define a
$\ZZ[q^{\pm \frac{1}{2}}]$-linear operator $L[\textbf{a}]: \mathcal{R} \rightarrow \mathcal{R}$ by setting
\begin{eqnarray} \label{def:L-function}
L[\textbf{a}](Z^{\textbf{b}}) = q^{-\aa \cdot \bb \cdot \hat{\delta}}Z^{\textbf{b}},
\end{eqnarray}
for $\textbf{b} = (b_{1}, \ldots, b_{n}) \in \mathbb{Z}^{n}$.

Suppose that for some $t \in \mathbb{T}_{n}$, the quantum $F$-polynomials $F_{j; t}$ are known for all $j \in [1, n]$, and $t' \in \mathbb{T}_{n}$ such that $t \frac{k}{\hspace{1cm}} t'$.  If $t'$ satisfies a certain condition, then the following result may be used to find the general form of $F_{k; t'}$, and the propositions that follow may be used to sharpen what is known about $F_{k; t'}$.

\begin{theorem} \cite[Theorem 6.7]{quantumfpoly} \label{thm:quantum-fpoly-rec} Let $k_{1}, \ldots, k_{\ell} \in [1, n]$ be a sequence of mutations such that $k_{\ell} \notin \{ k_{1}, \ldots, k_{\ell - 1} \}$, and set $k = k_{\ell}$. Write $t \in \mathbb{T}_{n}$ for the vertex obtained from $t_{0}$ along the path with edges labeled by $k_{1}, \ldots, k_{\ell - 1}$, and let $t' \in \mathbb{T}_{n}$ such that $t \frac{k}{\hspace{1cm}} t'$.  Let $B = (b_{ij})$ be the principal part of the exchange matrix at $t$.   For $\epsilon \in \{ + , - \}$, let
\begin{eqnarray}
I_{\epsilon} = \{ j : j = k_{i} \mbox{ for some } i, \mbox{and }  \epsilon b_{jk} > 0 \}
\end{eqnarray}
Let $k_{1}^{+}, \ldots, k_{p}^{+}$ be a list of the elements of $I_{+}$ in some order, and let $k_{1}^{-}, \ldots, k_{r}^{-}$ be a list of the elements of $I_{-}$.   For $\epsilon \in \{ +, - \}$, let
\begin{eqnarray}
G_{j}^{(\epsilon)} = \prod^{\rightarrow}_{i \in [1, [\epsilon b_{k_{j}^{\epsilon}; k}]]_{+}} L[(i - 1)\gg_{k_{j}^{\epsilon}; t}](F_{k_{j}^{\epsilon}; t})
\end{eqnarray}
for $j \in [1, p]$ if $ \epsilon = +$ and for $j \in [1, r]$ if $\epsilon = -$.  Also, let
\begin{eqnarray}
\hat{F}_{j}^{(\epsilon)} = L\left[ \sum_{i \in [1, j - 1]} [\epsilon b_{k_{j}^{\epsilon}, k}]_{+} \gg_{k_{j}^{\epsilon}; t}\right](G_{j}^{(\epsilon)})
\end{eqnarray}
Then there exist $\lambda^{+}, \lambda^{-} \in \frac{1}{2}\ZZ$ and $\aa^{+}, \aa^{-} \in \ZZ^{n}$ such that 
\begin{eqnarray}
F_{k; t'} = q^{\lambda^{+}}\left( \prod^{\rightarrow}_{j \in [1, p]} \hat{F}_{j}^{(+)}\right)Z^{\aa^{+}} + q^{\lambda^{-}}\left( \prod^{\rightarrow}_{j \in [1, r]} \hat{F}_{j}^{(-)}\right)Z^{\aa^{-}}
\end{eqnarray}

\end{theorem}

\begin{proposition} \label{prop:fpoly1} \cite[Theorem 5.3]{quantumfpoly} Let $u_{1}, \ldots, u_{n}$ be independent commuting variables.  Setting $Z_{i} = u_{i}$ for $i \in [1, n]$ and $q = 1$ in $F_{j; t}^{B^{0}; D; t_{0}}$ yields the (non-quantum) $F$-polynomial $F_{j; t}^{B^{0}; t_{0}}$.

\end{proposition}

\begin{proposition} \cite[Corollary 6.5]{quantumfpoly} \label{prop:fpoly-coeff} Let $j \in [1, n]$, $t \in \mathbb{T}_{n}$.  For $\aa \in \ZZ^{n}$,  let $P_{\aa} \in \ZZ[x, x^{-1}]$ such that $P_{\aa}(q^{\frac{1}{2}})$ is the coefficient of $Z^{\aa}$ in the quantum $F$-polynomial $F_{j; t}$.  

Then
\begin{eqnarray}
P_{\aa}(q^{\frac{1}{2}}) = q^{-\gg_{j; t} \cdot \aa \cdot \hat{\delta}}P_{\aa}(q^{-\frac{1}{2}}).
\end{eqnarray}
In particular, if $P_{\aa}(q) = q^{c}$ for some $c \in \frac{1}{2}\ZZ$, then $c = -\frac{\gg_{j; t}\cdot \aa \cdot \hat{\delta}}{2}$.
\end{proposition}

\section{Quantum $F$-polynomials in classical types}

Let $B^{0} = (b^{0}_{ij})$ be an acyclic $n \times n$ exchange matrix of type $\mbox{A}_{n}$, $\mbox{B}_{n}$, $\mbox{C}_{n}$, or $\mbox{D}_{n}$.  We resume using the same notation established in Section \ref{fpoly-classical}. For type $\mbox{A}_{n}$ and type $\mbox{D}_{n}$, let $\delta = (1, \ldots, 1) \in \ZZ^{n}$.  For type $\mbox{C}_{n}$, let $\delta = (1, \ldots, 1, 2) \in \ZZ^{n}$.  For type $\mbox{B}_{n}$, let $\delta = (2, \ldots, 2, 1) \in \ZZ^{n}$.   Let $d$ be a positive integer, and let $D$ be the $n \times n$ diagonal matrix whose entries are given by the vector $d\delta$.  Thus, $DB^{0}$ is a skew-symmetric matrix.

Write $F_{\dd} = F_{\dd}^{B^{0}; D; t_{0}}, F_{\dd}^{cl} = F_{\dd}^{B^{0}; t_{0}}, \gg_{\dd} = \gg_{\dd}^{B^{0}; t_{0}}$ for the quantum $F$-polynomial, $F$-polynomial, and $\gg$-vector corresponding to $\dd \in \Phi_{+}(B^{0})$.  By convention, set $F_{0} = 1$, $\gg_{0} = 0$.

For $\aa = (a_{1}, \ldots, a_{n}) \in \ZZ^{n}_{\geq 0}$ such that $u_{1}^{a_{1}}\ldots u_{n}^{a_{n}}$ has nonzero coefficients in $F_{\dd}^{cl}$, let $\phi_{\dd}(\aa) \in \ZZ_{\geq 0}$ such that $2^{\phi_{\dd}(\aa)}$ is the coefficient of $u_{1}^{a_{1}}\ldots u_{n}^{a_{n}}$ in $F_{\dd}^{cl}$ (i.e., $\phi_{\dd}(\aa) = c$, where $c$ is given in Theorem \ref{thm:fpoly-classical}).  Also, let $\rho_{\dd}(\aa) = 1$ if $B^{0}$ is of type $\mbox{B}_{n}$, the $n$th component of $\aa$ is 1, and $\phi_{\dd}(\aa) \geq 1$, and let $\rho_{\dd}(\aa) = 0$ otherwise.

\begin{theorem}  \label{thm:quantum-classical}  Let $\dd \in \Phi_{+}(B^{0})$.  In types $\mbox{A}_{n}$, $\mbox{C}_{n}$ and $\mbox{D}_{n}$, the quantum $F$-polynomial $F_{\dd}$ is given by 

\begin{eqnarray}
F_{\dd} 
            & = & \sum q^{-\frac{d}{2}\delta \cdot \gg_{\dd} \cdot \aa}(q^{\frac{d}{2}} + q^{-\frac{d}{2}})^{\phi_{\dd}(\aa)}Z^{\aa}.
\end{eqnarray}
In type $\mbox{B}_{n}$, 
\begin{eqnarray}
F_{\dd}  
& = & \sum q^{-\frac{d}{2}\delta \cdot \gg_{\dd} \cdot \aa}(q^{-\frac{d}{2}} + q^{\frac{d}{2}})^{\rho_{\dd}(\aa)}(q^{-d} + q^{d})^{\phi_{\dd}(\aa) - \rho_{\dd}(\aa)}Z^{\aa}
\end{eqnarray}
In either case, each summation ranges over $\aa = (a_{1}, \ldots, a_{n}) \in \ZZ^{n}_{\geq 0}$ such that $u_{1}^{a_{1}}\ldots u_{n}^{a_{n}}$ has nonzero coefficient in $F_{\dd}^{cl}$.

\end{theorem}

\begin{example} Use the same $B^{0}$ as in Example \ref{example:typeB}, and let $D$ be the $4 \times 4$ diagonal matrix with diagonal entries $2\delta = (4, 4, 4, 2)$.  Then the quantum $F$-polynomials for the same denominator vectors are given below.

\begin{eqnarray}
F_{\ee_{2}} & = & q^{2}Z_{2} + 1 \\
F_{\ee_{2} + \ee_{3}} & = & q^{2}Z^{\ee_{2} + \ee_{3}} + q^{2}Z_{2} + 1 \\
F_{\ee_{2} + \ee_{3} + \ee_{4}} & = & qZ^{\ee_{2} + \ee_{3} + \ee_{4}} + q^{3}Z^{\ee_{2} + \ee_{4}} + qZ_{4} + q^{2}Z_{2} + 1 \\
F_{\ee_{2} + 2\ee_{3} + 2\ee_{4}} & = & q^{2}Z^{\ee_{2} + 2\ee_{3} + 2\ee_{4}} + q^{2}Z^{\ee_{3} + 2\ee_{4}} + q^{4}Z^{2\ee_{4}}  \\
\nonumber & & \hspace{0.5cm}   + \, (q + q^{3})Z^{\ee_{2} + \ee_{3} + \ee_{4}} + (q^{2} + q^{6})Z^{\ee_{2} + \ee_{3} + 2\ee_{4}} \\
\nonumber & & \hspace{0.5cm} + \, q^{6}Z^{\ee_{2} + 2\ee_{4}} + (q^{3} + q^{5})Z^{\ee_{2} + \ee_{4}} \\
\nonumber & & \hspace{0.5cm} + \, (q + q^{3})Z_{4} + q^{2}Z_{2} + 1 \\
F_{\ee_{1} + 2\ee_{2} + 2\ee_{3} + 2\ee_{4}} & = & q^{2}Z^{\ee_{1} + 2\ee_{2} + 2\ee_{3} + 2\ee_{4}} + q^{4}Z^{2\ee_{2} + 2\ee_{3} + 2\ee_{4}}   \\
\nonumber & & \hspace{0.5cm} +  \, (q^{4} + q^{8})Z^{\ee_{1} + 2\ee_{2} + \ee_{3} + 2\ee_{4}}  + (q^{6} + q^{10})Z^{2\ee_{2} + \ee_{3} + 2\ee_{4} } \\
\nonumber & & \hspace{0.5cm} + \, q^{10}Z^{\ee_{1} + 2\ee_{2} + 2\ee_{4}}  +  q^{12}Z^{2\ee_{2} + 2\ee_{4}} \\
\nonumber & & \hspace{0.5cm} + \, (q^{3} + q^{5})Z^{\ee_{1} + 2\ee_{2} + \ee_{3} + \ee_{4}} + (q^{7} + q^{9})Z^{\ee_{1} + 2\ee_{2} + \ee_{4}}   \\
\nonumber & & \hspace{0.5cm} + \, (q^{5} + q^{7})Z^{2\ee_{2} + \ee_{3} + \ee_{4}} + q^{6}Z^{\ee_{1} + 2\ee_{2}}  \\
\nonumber & & \hspace{0.5cm}  + \, (q^{9} + q^{11})Z^{2\ee_{2} + \ee_{4}}  + q^{8}Z^{2\ee_{2}} + q^{2}Z^{\ee_{1} + \ee_{2} + \ee_{3} + 2\ee_{4}}  \\
\nonumber & & \hspace{0.5cm}  + \, q^{6}Z^{\ee_{1} + \ee_{2} + 2\ee_{4}} + (q^{2} + q^{6})Z^{\ee_{2} + \ee_{3} + 2\ee_{4}}  \\
\nonumber & & \hspace{0.5cm}   + \, (q^{6} + q^{10})Z^{\ee_{2} + 2\ee_{4}} + q^{4}Z^{2\ee_{4}} + (q^{3} + q^{5})Z^{\ee_{1} + \ee_{2} + \ee_{4}} \\
\nonumber & & \hspace{0.5cm}  + \, (q + q^{3})Z^{\ee_{2} + \ee_{3} + \ee_{4}} + (q^{3} + q^{5} + q^{7} + q^{9})Z^{\ee_{2} + \ee_{4}}  \\
\nonumber & & \hspace{0.5cm} + \, q^{2}Z^{\ee_{1} + \ee_{2}} + (q + q^{3})Z_{4} + (q^{2} + q^{6})Z_{2} + 1 
\end{eqnarray}
For example, let $\dd = \ee_{1} + 2\ee_{2} + 2\ee_{3} + 2\ee_{4}$.  To verify the coefficient of $Z^{\ee_{2} + \ee_{4}}$ in $F_{\dd}$ is $q^{3} + q^{5} + q^{7} + q^{9} = q^{6}(q + q^{-1})(q^{2} + q^{-2})$, observe that Example \ref{example:typeB} shows that  $\phi_{\dd}(\ee_{2} + \ee_{4}) = 2$ and $\rho_{\dd}(\ee_{2} + \ee_{4}) = 1$.  Using the $\gg$-vector $\gg_{\dd} = \ee_{1} - 2\ee_{2} + 2\ee_{3} - 2\ee_{4}$ from the same example, it is easy to check that $-\frac{d}{2}\delta \cdot \gg_{d} \cdot (\ee_{2} + \ee_{4}) = 6$, as desired.
\end{example}

Theorem \ref{thm:quantum-classical} was proven for type $\mbox{A}_{n}$ in \cite{quantumfpoly}.  The proof in the remaining types will be given in the subsections below.

Some notation which will be used below: For $S \subset [1, n]$, write 
\begin{eqnarray}
\ee_{S} = \sum_{i \in S} \ee_{i} \in \ZZ^{n}
\end{eqnarray}

\subsection{Combinatorial realizations of cluster algebras of types $\mbox{B}_{n}, \mbox{C}_{n}, \mbox{D}_{n}$}

For each of these types, we recall from \cite{ca2} certain combinatorial realizations of cluster variables and clusters which will be useful in the proofs below. 

For types $\mbox{B}_{n}$ and $\mbox{C}_{n}$, the realization is given in terms of diagonals of the regular $(2n + 2)$-gon $\mathbb{P}_{2n + 2}$.    In order to deal with both types at once, set $r = 1$ for type $\mbox{B}_{n}$, and $r = 2$ for type $\mbox{C}_{n}$.  Let $\theta$ be the 180$^\circ$ rotation of $\PP_{2n + 2}$.  Then $\theta$ also acts on the diagonals of $\PP_{2n + 2}$.   If $a, b$ are endpoints of a diagonal, then we write $[ab]$ for the diagonal.  If $a$ is a vertex in $\PP_{2n + 2}$, then write $\overline{a} = \theta(a)$ for the opposite vertex on $\PP_{2n + 2}$.

The cluster variables are in bijective correspondence with the $\theta$-orbits of diagonals of $\PP_{2n + 2}$.  Two cluster variables are in some cluster together if and only if the diagonals in their corresponding $\theta$-orbits do not cross each other.

\begin{definition}  \label{def:max-diag-set}
Call $\mathcal{D}$ a \emph{maximal diagonal set} if $\mathcal{D}$ is a set of diagonals of $\mathbb{P}_{2n + 2}$ such that 
\begin{itemize}
\item no two elements of $\mathcal{D}$ cross;
\item $\mathcal{D}$ is closed under the action of $\theta$;
\item every diagonal outside of $\mathcal{D}$ crosses at least one of the elements of $\mathcal{D}$.
\end{itemize}
\end{definition} 
 
\noindent Maximal diagonal sets of $\mathbb{P}_{2n + 2}$ are in bijective correspondence with clusters.
  The initial cluster corresponding to $B^{0}$ is represented by the set of diagonals $\alpha_{1}, \ldots, \alpha_{n}$ and their $\theta$-orbits, which are constructed as follows:   Let $\alpha_{1}$ be any diagonal of shortest length.  Then we construct $\alpha_{2}, \ldots, \alpha_{n}$ in order so that $\alpha_{i + 1}$ shares an endpoint with $\alpha_{i}$, the other endpoints of $\alpha_{i}$, $\alpha_{i + 1}$ are the endpoints of a side of $\mathbb{P}_{2n + 2}$, and $\alpha_{i + 1}$ is clockwise (by an acute angle) to $\alpha_{i}$ if $i \rightarrow i + 1$ in $Q^{0}$ and  counterclockwise to $\alpha_{i}$ otherwise.    Note that $\alpha_{n}$ is a diameter of $\mathbb{P}_{2n + 2}$.  Let $\mathcal{D}^{0} = \{ \alpha_{1}, \ldots, \alpha_{n}, \theta(\alpha_{1}), \ldots, \theta(\alpha_{n - 1}) \}$.
 
In type $\mbox{B}_{n}$ (resp. $\mbox{C}_{n}$), the cluster variable with denominator vector $\sum_{i \in [1, n]} a_{i} \ee_{i} \in \Phi_{+}(B^{0})$ is represented by the unique $\theta$-orbit $\{ \beta, \theta(\beta) \}$ such that for each $i \in [1, n]$, the diagonal $\alpha_{i}$ (resp.  $\beta$) crosses the diagonals in $\{ \beta, \theta(\beta) \}$ (resp. $\{ \alpha_{i}, \theta(\alpha_{i}) \}$) at a total of $a_{i}$ points.
 
Next, we describe all possible cluster mutations in terms of maximal diagonal sets.  Consider a cluster which corresponds to a maximal diagonal set $\mathcal{D}$, and let $B = (b_{ij})$ be the principal part of the exchange matrix.  Let $k \in [1, n]$.    Suppose that the $k$th cluster variable corresponds to a diagonal $[ac]$ in $\mathcal{D}$.    Then there exist vertices  $c, f$ of $\mathbb{P}_{2n + 2}$ such that $[ab], [bc], [cf]$, and $[af]$ are each either a diagonal in $\mathcal{D}$ or side of $\mathbb{P}_{2n + 2}$, and $a, b, c, f$ are distinct vertices in counterclockwise order on $\mathbb{P}_{2n + 2}$.  Mutation in direction $k$ then corresponds to replacing the diagonals $[ac], \theta([ac])$ by $[bf], \theta([bf])$.  

Let $i \in [1, n]$.  If $i$ corresponds to a diameter, then
\begin{eqnarray} \label{eqn:bc1}
b_{ik} = \left\{      \begin{array}{cl} 
2/r & \mbox{if } i \mbox{ corresponds to } [cf] \mbox{ or } [ab]  \\
-2/r & \mbox{if } i \mbox{ corresponds to } [bc] \mbox{ or } [af] \\
0 & \mbox{otherwise} 
\end{array}
\right.
\end{eqnarray}
If $k$ corresponds to a diameter, then 
\begin{eqnarray} \label{eqn:bc2}
b_{ik} = \left\{      \begin{array}{cl} 
r & \mbox{if } i \mbox{ corresponds to } [cf] \mbox{ or } [ab]  \\
-r & \mbox{if } i \mbox{ corresponds to } [bc] \mbox{ or } [af] \\
0 & \mbox{otherwise} 
\end{array}
\right.
\end{eqnarray}
Otherwise, 
\begin{eqnarray} \label{eqn:bc3}
b_{ik} = \left\{      \begin{array}{cl} 
1 & \mbox{if } i \mbox{ corresponds to } [cf] \mbox{ or } [ab]  \\
-1 & \mbox{if } i \mbox{ corresponds to } [bc] \mbox{ or } [af] \\
0 & \mbox{otherwise} 
\end{array}
\right.
\end{eqnarray}

For type $\mbox{D}_{n}$, the cluster variables are in bijective correspondence with the $\theta$-orbits of diagonals of $\PP_{2n}$, where diameters may be of two different ``colors," which are denoted by $[a\overline{a}]$ and $\widetilde{[a\overline{a}]}$.   Two cluster variables are in some cluster together if and only if the diagonals in their corresponding $\theta$-orbits do not cross each other, with the additional assumption that diameters of the same color do not cross.   

Define \emph{maximal diagonal sets} in an analogous manner to Definition \ref{def:max-diag-set} so that $\mathcal{D}$ is a subset of diagonals and (colored) diameters of $\mathbb{P}_{2n}$, and ``crossing" takes into account the convention just given.  Maximal diagonal sets of $\mathbb{P}_{2n}$ are in bijective correspondence with clusters.

The initial cluster corresponding to $B^{0}$ is represented by the set of diagonals $\alpha_{1}, \ldots, \alpha_{n}$ and their $\theta$-orbits, which are constructed as follows:   Construct $\alpha_{1}, \ldots, \alpha_{n - 1}$ in $\mathbb{P}_{2n}$ in a similar manner as in the type $\mbox{C}_{n}$ setting.   Note that $\alpha_{n - 1}$ is a diameter of the first color in $\mathbb{P}_{2n}$.  If the arrow between $n - 2$ and $n$ in $Q^{0}$ is in the same direction as the arrow between $n - 2$ and $n - 1$, then let $\alpha_{n}$ be the diagonal $\tilde{\alpha}_{n - 1}$; otherwise, let $\alpha_{n}$ be the diameter with the same color as $\alpha_{n - 1}$ which shares the other endpoint of $\alpha_{n - 2}$.  Let $\mathcal{D}^{0} = \{ \alpha_{1}, \ldots, \alpha_{n}, \theta(\alpha_{1}), \ldots, \theta(\alpha_{n - 2}) \}$.
 
 The cluster variable with denominator vector $\sum_{i \in [1, n]} a_{i} \ee_{i} \in \Phi_{+}(B^{0})$ is represented by the unique $\theta$-orbit such that for each $i \in [1, n]$, the diagonals in this orbit cross the diagonals representing $\alpha_{i}$ at $a_{i}$ pairs of centrally symmetric points (where we count an intersection of two diameters of different colors and location as one such pair).

Next, we describe certain cluster mutations in terms of maximal diagonal sets.  Consider a cluster which corresponds to a maximal diagonal set $\mathcal{D}$, and let $B = (b_{ij})$ be the principal part of the exchange matrix.  Let $k \in [1, n]$.    Suppose that the $k$th cluster variable corresponds to a diagonal $\beta = [ac] \in \mathcal{D}$, or to $\beta = \widetilde{[ac]} \in \mathcal{D}$, where $c = \overline{a}$.    Assume there exist vertices  $b, f$ of $\mathbb{P}_{2n + 2}$ such that $[ab], [bc], [cf]$, and $[af]$ are each either a diagonal in $\mathcal{D}$ or side of $\mathbb{P}_{2n + 2}$, and $a, b, c, f$ are distinct vertices in counterclockwise order on $\mathbb{P}_{2n + 2}$.  Furthermore, assume that if $[ab]$ is a diameter, then $\widetilde{[ab]}$ is also in $\mathcal{D}$, and assume similar conditions with $[bc], [cf]$, and $[af]$.
If $\beta = \widetilde{[ac]}$ (resp. $\beta = [ac]$), and $\beta$ is a diameter, then $c = \overline{a}$, $b = \overline{f}$, and mutation in direction $k$ corresponds to replacing $\beta$ by the diameter $[bf]$ (resp. $\widetilde{[bf]}$).   If $\beta$ is not a diameter, then mutation in direction $k$ corresponds to replacing the $\theta$-orbit $\{\beta$, $\theta(\beta)\}$ by $\{[bf]$, $\theta([bf])\}$.

Let $i \in [1, n]$.  Then
\begin{eqnarray} \label{exchmatD}
b_{ik} = \left\{      \begin{array}{cl} 
1 & \mbox{if } i \mbox{ corresponds to } [cf], [ab], \widetilde{[cf]}, \mbox{ or } \widetilde{[ab]}  \\
-1 & \mbox{if } i \mbox{ corresponds to } [bc], [af], \widetilde{[bc]}, \mbox{ or } \widetilde{[af]} \\
0 & \mbox{otherwise} 
\end{array}
\right.
\end{eqnarray}
Note that in the above, if $i$ corresponds to $\widetilde{[ab]}$, for example, then we assume that $[ab]$ is a diameter, i.e., $b = \overline{a}$.  Likewise, the same convention applies to $\widetilde{[cf]}, \widetilde{[bc]}$, and $\widetilde{[af]}$.

\begin{remark}  In \cite{triangulated}, Fomin, Shapiro, and Thurston associate cluster algebras to  punctured Riemann surfaces.  In particular, cluster algebras of type $\mbox{D}_{n}$ may be realized in terms of triangulations of a once-punctured $n$-gon; see \cite[Example 6.7]{triangulated} for additional details.  
\end{remark}

\subsection{Type $\mbox{D}_{n}$}

Let $B^{0} = (b^{0}_{ij})$ be an acyclic $n \times n$ exchange matrix of type $\mbox{D}_{n}$.


The theorem has already been proven in \cite{quantumfpoly} for $\dd \in \Phi_{+}(B^{0})$ such that the components of $\dd$ are all either 0 or 1.  Thus, it remains to prove the theorem for $\dd$ of the form $\sum_{i = p}^{n} \ee_{i} + \sum_{j = r}^{n - 2} \ee_{j}$ ($1 \leq p < r \leq n - 1$).  Proceed by induction on $n - 2 - p$.  The base of the induction occurs when $n - 2 - p = 0$, so that $\dd = \ee_{n - 2} + \ee_{n - 1} + \ee_{n}$, and in this case, the theorem has already been proven.

Fix $\dd = \sum_{i = p}^{n} \ee_{i} + \sum_{j = r}^{n - 2} \ee_{j}$ ($1 \leq p < r \leq n - 2$), and suppose that the theorem has been proven for all denominator vectors $\sum_{i = p'}^{n} \ee_{i} + \sum_{j = r'}^{n - 2} \ee_{j}$, where $p < p' < r' \leq n - 2$.   Write $\dd = (d_{1}, \ldots, d_{n})$.

\begin{lemma} \label{typeDrec} At least one of the two statements below  is true:
\begin{enumerate}
\item For some $\lambda', \lambda'' \in \frac{1}{2}\ZZ$, $\aa', \aa'' \in \ZZ^{n}$, and $\dd', \dd'' \in \Phi_{+}(B^{0}) \cup \{ 0 \}$ such that $\dd', \dd'' < \dd$ and the $p$th component of both vectors is 0, the quantum $F$-polynomial $F_{\dd}$ satisfies

\begin{eqnarray} \label{eqn:fpolyD-rec1}
F_{\dd} = q^{\lambda'}F_{\dd'}Z^{\aa'} + q^{\lambda''}F_{\dd''}Z^{\aa''}
\end{eqnarray}

\item Let 
\begin{eqnarray} \label{eqn:dd-1}
\dd'' = \ee_{r} + \cdots + \ee_{n - 1}, \hspace{0.5cm} \dd''' = \ee_{r} + \cdots + \ee_{n - 2} + \ee_{n}
\end{eqnarray}
if the arrow between $n - 2$ and $n$ and the arrow between $n - 2$ and $n - 1$ in $Q^{0}$ are in the same direction, and let
\begin{eqnarray} \label{eqn:dd-2}
\dd'' = \ee_{r} + \cdots + \ee_{n - 2}, \hspace{0.5cm} \dd''' = \ee_{r} + \cdots + \ee_{n}
\end{eqnarray}
otherwise.  Also, let 
\begin{eqnarray}
F'_{\dd'''} = L[\gg_{\dd''}](F_{\dd'''}).
\end{eqnarray}
For some  $\aa', \aa'' \in \ZZ^{n}$, $\lambda', \lambda'' \in \frac{1}{2}\ZZ$, and $\dd' \in \Phi_{+}(B^{0}) \cup \{ 0 \}$ such that $\dd' < \dd$ and the $p$th component of $\dd'$ is 0, we have
\begin{eqnarray} \label{eqn:fpolyD-rec2}
F_{\dd} = q^{\lambda'}F_{\dd'}Z^{\aa'} + q^{\lambda''}F_{\dd''}F'_{\dd'''}Z^{\aa''}
\end{eqnarray} 

\end{enumerate}
\end{lemma}
\begin{proof}  We will show that there exists a sequence of cluster mutations which yields the cluster variable corresponding to $\dd$ by using the combinatorial representation given above.  Then Theorem \ref{thm:quantum-fpoly-rec} will guarantee that at least one of the equations for $F_{\dd}$ occurs.  

Note that the diagonals $\alpha_{1}, \ldots, \alpha_{n - 1}$ divide half of $\PP_{2n}$ into $n - 1$ triangles.  For $i = 1, \ldots, n - 2$, let $\Delta_{i}$ be the triangle containing $\alpha_{i}$ and $\alpha_{i + 1}$ as sides, and let $\Delta_{0}$ be the triangle with $\alpha_{1}$ as a side but not $\alpha_{2}. $   Let $v$ be the vertex opposite $\alpha_{r}$ in $\Delta_{r - 1}$, and for $i \in [1, n - 2]$, let $v_{i}$ be the vertex opposite $\alpha_{i}$ in $\Delta_{i}$.   

First, if $\alpha_{n - 1}$ and $\alpha_{n}$ are diameters of the same color, then we perform a mutation on $\mathcal{D}^{0}$ in direction $n$ so that $\alpha_{n}$ is mutated to $\tilde{\alpha}_{n - 1}$.  Then mutate  in directions $k = r, \ldots, n$.  Under this sequence of mutations, the diagonal $\alpha_{i}$ ($i \in [r, n - 2]$) is mutated to $\alpha_{i}' = [vv_{i}]$, and $\alpha_{n - 1}$ and $\tilde{\alpha}_{n - 1}$ are mutated to the diagonals  $\widetilde{[v\overline{v}]}$ and $[v\overline{v}]$.

Let $r \leq j_{s} < \ldots < j_{1} \leq n - 2$ be indices such that the arrow between $j_{\ell}$ and $j_{\ell} + 1$ is opposite the arrow between $r - 1$ and $r$ for $\ell = 1, \ldots, s$.  Next, we continue mutating in directions $k = j_{1}, \ldots, j_{s}$.  Let $w$ be the endpoint of the diagonal $\alpha_{r - 1}$ which is not equal to $v$.   Then $\alpha'_{j_{\ell}}$ is mutated to $\alpha''_{j_{\ell}} = [\overline{v}v_{j_{\ell + 1}}]$ (for $\ell \in [1, s - 1]$) and $\alpha_{j_{s}}'$ is mutated to $[\overline{v}w]$.  

Let $w_{i}$ be the vertex opposite $\alpha_{i}$ in the triangle $\Delta_{i - 1}$ for $i \in [p, r - 1]$.   Continue mutating in directions $k = r - 1, \ldots, p$.  Then $\alpha_{k}$ is mutated to $\alpha_{k}' = [\overline{v}w_{k}]$ for $k = p, \ldots, r - 1$.   Note that $\alpha_{p}'$ corresponds to the cluster variable with denominator vector $\dd$ since it crosses the diagonals $\alpha_{p}, \ldots, \alpha_{n}, \theta(\alpha_{r}), \ldots, \theta(\alpha_{n - 2})$.

Let $a_{1}, a_{2}$ be the endpoints of $\alpha_{p}$.  The diagonal $\alpha'_{p}$ is the diagonal of a quadrilateral whose sides are each in $\mathcal{D}'$ or sides of $\mathbb{P}_{2n}$.  To be more precise, the sides of the quadrilateral are $\beta_{1} = [\overline{v}a_{1}]$, $\beta_{2} = [\overline{v}a_{2}]$,  $\alpha_{p - 1}$ (if $p - 1 \geq 1$), and  $\widetilde{[v\overline{v}]}$ (if one of the $a_{i}$ equals $v$); the remaining sides (if any) of the quadrilateral are sides of $\mathbb{P}_{2n}$.

Observe that the only elements of $\mathcal{D}^{0}$ that the $\beta_{i}$ can intersect come from the list  $\alpha_{p + 1}, \ldots, \alpha_{n}, \theta(\alpha_{r}), \ldots, \theta(\alpha_{n - 2})$.   Thus, each $\beta_{i}$ corresponds to some $\dd' \in \Phi_{+}(B^{0})$ such that $\dd' < \dd$, and the $p$th component of $\dd'$ is equal to 0.

 If $a_{i} \neq v$ for $i = 1, 2$, then the diagonals $\beta_{1}, \beta_{2}$ are not diameters. It follows from Theorem \ref{thm:quantum-fpoly-rec} in conjunction with (\ref{exchmatD}) that there exists an equation of the form  (\ref{eqn:fpolyD-rec1}), where the $\dd', \dd''' \in \Phi_{+}(B^{0})$ in this equation correspond to the diagonals $\beta_{1}, \beta_{2}$.  If $a_{i} = v$ for some $i$, then observe that $[v\overline{v}]$ intersects the interiors of $\alpha_{r}, \ldots, \alpha_{n}$, which means that $[v\overline{v}]$ and $\widetilde{[v\overline{v}]}$ correspond to denominator vectors $\dd''$, $\dd'''$ as at (\ref{eqn:dd-1}) or (\ref{eqn:dd-2}).  In this case, we get an equation for $F_{\dd}$ of the form (\ref{eqn:fpolyD-rec2}).
\end{proof}

The equations for $F_{\dd}$ given in Lemma \ref{typeDrec} show that the terms of $F_{\dd}$ are subtraction-free Laurent polynomials in $Z_{1}, \ldots, Z_{n}$ with coefficients in $\ZZ[q^{\pm \frac{1}{2}}]$.  By Proposition \ref{prop:fpoly1},  setting $q = 1$ and $Z_{i} = u_{i}$ ($i \in [1, n]$) in $F_{\dd}$ yields $F_{\dd}^{cl}$, so it follows that for $\aa = (a_{1}, \ldots, a_{n}) \in \ZZ^{n}$, the monomial $Z^{\aa}$ occurs with nonzero coefficient in $F_{\dd}$ if and only if $u_{1}^{a_{1}}\ldots u_{n}^{a_{n}}$ occurs with nonzero coefficient in $F_{\dd}^{cl}$. 

To prove Theorem \ref{thm:quantum-classical}, it suffices to show that all nonzero coefficients of $F_{\dd}$ are of the form 
\begin{eqnarray} \label{coeff-formD}
q^{c_{1}}(1 + q^{d})^{c_{2}}
\end{eqnarray}
for some $c_{1}, c_{2} \in \frac{1}{2}\ZZ$.  Once this fact has been established, the proof of the theorem in type $\mbox{D}_{n}$ may be finished in the following way.  Suppose that the coefficient of $Z^{\aa}$ is of the form given at (\ref{coeff-formD}).  This expression may be rewritten as $q^{c_{1}'}(q^{-\frac{1}{2}} + q^{\frac{1}{2}})^{c_{2}}$ for
some $c_{1}' \in \frac{1}{2}\ZZ$.  Since setting $q = 1$ and $Z_{i} = u_{i}$ for $i \in [1, n]$ in $F_{\dd}$ yields $F_{\dd}^{cl}$, it follows that $c_{2} = \phi_{\dd}(\aa)$.   Using Proposition \ref{prop:fpoly-coeff}, we have that 
\begin{eqnarray}
q^{c_{1}'}(q^{-\frac{1}{2}} + q^{\frac{1}{2}})^{\phi_{\dd}(\aa)} = q^{-d \delta \cdot \gg_{\dd} \cdot \aa }q^{-c_{1}'}(q^{-\frac{1}{2}} + q^{\frac{1}{2}})^{\phi_{\dd}(\aa)}.
\end{eqnarray}
Thus, $c_{1}' = -\frac{d}{2} \delta \cdot \gg_{\dd} \cdot \aa$, as desired.

By Theorem  \ref{thm:fpoly-classical}, $Z^{\dd}$ occurs with nonzero coefficient in $F_{\dd}$, and $F_{\dd}$ has nonzero constant term.  Thus, in the expressions (\ref{eqn:fpolyD-rec1}) and (\ref{eqn:fpolyD-rec2}), one of $\aa'$ or $\aa''$ must have $p$th component 1, and the other must be equal to the 0-vector.   In (\ref{eqn:fpolyD-rec1}), this means that one of the expansions of $q^{\lambda'}F_{\dd'}Z^{\aa'}$, $q^{\lambda''}F_{\dd''}Z^{\aa''}$ contains all monomials $Z^{\aa}$ in $F_{\dd}$ with $Z_{p}$, and the other contains all monomials which do not have $Z_{p}$.    Consequently, if (\ref{eqn:fpolyD-rec1}) holds, then it follows immediately from the induction hypothesis that all monomials in $F_{\dd}$ have coefficient of the form at (\ref{coeff-formD}).
By similar reasoning, in the expression (\ref{eqn:fpolyD-rec2}), one of the expansions for $q^{\lambda'}F_{\dd'}Z^{\aa'}$, $q^{\lambda''}F_{\dd''}F'_{\dd'''}Z^{\aa''}$ contains all monomials with $Z_{p}$, and the other contains all monomials without $Z_{p}$.  By the induction hypothesis, it is clear that coefficients of the monomials in the expansion of $q^{\lambda'}F_{\dd'}Z^{\aa'}$ have the form (\ref{coeff-formD}).  The remainder of the proof of Theorem \ref{thm:quantum-classical} is devoted to showing the same statement  is true of the monomials in $F_{\dd''}F'_{\dd'''}$.

Fix the denominator vectors $\dd'', \dd'''$ as at (\ref{eqn:dd-1}) or (\ref{eqn:dd-2}), and fix $\aa = (a_{1}, \ldots, a_{n}) \in \ZZ^{n}$ such that the monomial $Z^{\aa}$ occurs with nonzero coefficient in $F_{\dd''}F'_{\dd'''}$.  Let $S_{\aa}$ denote the set of all possible pairs $(\vv, \ww) \in \ZZ^{n} \times \ZZ^{n}$ such that $\aa = \vv + \ww$, and such that $Z^{\vv}$ (resp. $Z^{\ww}$) occurs with nonzero coefficient in $F_{\dd''}$ (resp. $F'_{\dd'''}$).   The next goal  is to describe the elements that can occur in $S_{\aa}$.

 Let $S$ be the subgraph of $Q^{0}$ induced by $ \{ i \in [r, n - 2] : a_{i} = 1 \}$.  Let $\mathcal{C}(S)$ denote the set of components of $S$, excluding the component which contains $n - 2$ if at least one of the arrows between $n - 2$ and $n$ and between $n - 2$ and $n - 1$ in $Q^{0}$ is critical with respect to $(\dd, \aa)$.

\begin{lemma} \label{lemma:d-a} 
\begin{enumerate}
 \item $0 \leq a_{i} \leq 2$ for $i = [r, n - 2]$,  \\
         $0 \leq a_{i} \leq 1$ for $i = n - 1, n$,  \\
          $a_{i} = 0$ for $i  \in [1, r - 1]$.
  \item If $i \rightarrow j$ in $Q^{0}$ with $r \leq i, j \leq n - 2$, then $a_{j} \geq a_{i}$.
  \item If $n - 2 \rightarrow n$ in $Q^{0}$ (resp. $n - 2 \rightarrow n - 1$) and $a_{n - 2} = 2$, then $a_{n} = 1$ (resp. $a_{n - 1} = 1$).
    \item If $n \rightarrow n - 2$ (resp. $n - 1 \rightarrow n - 2$), then $a_{n - 2} \geq a_{n}$ (resp. $a_{n - 2} \geq a_{n - 1}$).
  \item At least one of the arrows between $n -2$ and $n - 1$ and between $n - 2$ and $n$ in $Q^{0}$ is not critical with respect to $(\dd, \aa)$.
\end{enumerate}
\end{lemma}
\begin{proof}  Let $(\vv, \ww) \in S_{\aa}$, and write $\vv = (v_{1}, \ldots, v_{n}), \ww = (w_{1}, \ldots, w_{n})$.  Statement (1) follows immediately from part (1) of Theorem \ref{thm:fpoly-classical}.  Part (2) of the same theorem implies that if $i \rightarrow j$ in $Q^{0}$ and $i, j \in [r, n - 2]$, then $v_{i} \leq v_{j}$ and $w_{i} \leq w_{j}$, so statement (2) of the lemma also follows.  If $n - 2 \rightarrow n$ in $Q^{0}$ and $a_{n - 2} = 2$, then $v_{n - 2} = w_{n - 2} = 1$, so using part (2) of Theorem \ref{thm:fpoly-classical} again, it follows that $1 = v_{n - 2} \leq v_{n} \leq a_{n}$ if (\ref{eqn:dd-1}) holds and $1 = w_{n - 2} \leq w_{n} \leq a_{n}$ if (\ref{eqn:dd-2}) holds.  Similar reasoning applies to the case when $n - 2 \rightarrow n - 1$ in $Q^{0}$.  This proves statement (3) of the lemma.   For statement (4), if $n \rightarrow n - 2$ in $Q^{0}$, then parts (1) and (2) of Theorem \ref{thm:fpoly-classical} imply that $v_{n} = 0$ and  $w_{n} \leq w_{n - 2} \leq a_{n - 2}$, so  $a_{n} \leq a_{n - 2}$.  Similar reasoning applies to the situation where $n - 1 \rightarrow n - 2$ in $Q^{0}$.   Finally, for the final statement of the lemma, suppose for the sake of contradiction that both of the arrows between $n -2$ and $n - 1$ and between $n - 2$ and $n$ are critical.    We consider the case when (\ref{eqn:dd-1}) holds and leave the other case as an easy check.  If $n - 2 \rightarrow n$ and $n - 2 \rightarrow n - 1$ in $Q^{0}$, then $v_{n - 2} \leq v_{n - 1}$ and $w_{n - 2} \leq w_{n}$.  The fact that the stated arrows are critical implies that $a_{n - 2} = 1$ and $a_{n} = a_{n - 1} = 0$.  This means that $v_{n - 1} = 0$ and $w_{n} = 0$, and either $v_{n - 2} = 1$ or $w_{n - 2} = 1$, contradiction.  Now suppose that $n \rightarrow n - 2$ and $n - 1 \rightarrow n - 2$ in $Q^{0}$.  Then $v_{n - 1} \leq v_{n - 2}$ and $w_{n} \leq w_{n - 2}$.  The fact that the arrows are critical means that $a_{n - 2} = 1$ and $a_{n} = a_{n - 1} = 1$, so $v_{n - 1} = w_{n} = 1$, and either $v_{n - 2} = 0$ or $w_{n - 2} = 0$, contradiction.
\end{proof}

\

Let 
\begin{eqnarray} \label{eqn:vv-ww1}
\vv^{(0)} & =  &\ee_{n - 1} +  \sum_{i \in [r, n - 2], \, a_{i} \geq 1} \ee_{i} \\
\nonumber \ww^{(0)} & =  & a_{n}\ee_{n} + \sum_{i \in [r, n - 2], \, a_{i} = 2} \ee_{i}
\end{eqnarray}
if $\dd'', \dd'''$ are as at (\ref{eqn:dd-1}) and $a_{n - 1} = 1$; let
\begin{eqnarray} \label{eqn:vv-ww2}
\vv^{(0)} & =  &  \sum_{i \in [r, n - 2], \, a_{i} = 2} \ee_{i} \\
\nonumber \ww^{(0)} & =  & a_{n} \ee_{n} +  \sum_{i \in [r, n - 2], \, a_{i} \geq 1} \ee_{i} 
\end{eqnarray}
if $\dd'', \dd'''$ are as at (\ref{eqn:dd-1}) and $a_{n - 1} = 0$; let
\begin{eqnarray} \label{eqn:vv-ww3}
\vv^{(0)} & =  &  \sum_{i \in [r, n - 2], \, a_{i} = 2} \ee_{i} \\
\nonumber \ww^{(0)} & =  & a_{n - 1}\ee_{n - 1} + a_{n}\ee_{n} +  \sum_{i \in [r, n - 2], \, a_{i} \geq 1} \ee_{i} 
\end{eqnarray}
if $\dd'', \dd'''$ are as at (\ref{eqn:dd-2}) and $a_{n - 1} = 1$; finally, let 
\begin{eqnarray}  \label{eqn:vv-ww4}
\vv^{(0)} & =  &  \sum_{i \in [r, n - 2], \, a_{i} \geq 1} \ee_{i} \\
\nonumber \ww^{(0)} & =  & a_{n - 1}\ee_{n - 1} + a_{n}\ee_{n} +  \sum_{i \in [r, n - 2], \, a_{i} = 2} \ee_{i} 
\end{eqnarray} 
if $\dd'', \dd'''$ are as at (\ref{eqn:dd-2}) and $a_{n - 1} = 0$.

\begin{lemma} \label{lemma:S-a}
If  $\vv^{(0)}, \ww^{(0)}$ are as at (\ref{eqn:vv-ww1}) or (\ref{eqn:vv-ww4}), then 
\begin{eqnarray}
S_{\aa} = \{ (\vv^{(0)} - \aa', \ww^{(0)} + \aa')  :  J \subset \mathcal{C}(S), \aa' = \sum_{C \in J} \ee_{C} \}.\end{eqnarray}
Otherwise, 
\begin{eqnarray}
S_{\aa} = \{ (\vv^{(0)} + \aa', \ww^{(0)} - \aa')  :  J \subset \mathcal{C}(S), \aa' = \sum_{C \in J} \ee_{C} \}.
\end{eqnarray}

\end{lemma}

\begin{proof}
Assume that the arrows between $n - 2$ and $n - 1$ and between $n - 2$ and $n$ are in the same direction, so that $\dd'', \dd'''$ are as at (\ref{eqn:dd-1}).  (The proof when $\dd'', \dd'''$ are as at (\ref{eqn:dd-2}) is similar.)  

Let $\aa' = \sum_{C \in J} \ee_{C}$ for some $J \subset \mathcal{C}(S)$, and let $\vv = \vv^{(0)} \mp \aa' = (v_{1}, \ldots, v_{n})$, $\ww = \ww^{(0)} \pm \aa' = (w_{1}, \ldots, w_{n})$ (taking the top signs in the situation of (\ref{eqn:vv-ww1}) and the bottom signs otherwise).   First, we will show that $Z^{\vv}$ occurs with nonzero coefficient in $F_{\dd''}$.    Let $i, j \in [r, n - 1]$ with $i \rightarrow j$ in $Q^{0}$.   Then we need to show that $v_{j} \geq v_{i}$. 

\

\textit{Case 1:}  $i, j \in S$ \\
Then $i, j$ are in the same component of $S$, which means that the $i$th and $j$th components of $\vv^{(0)}$ and hence of $\vv$ are equal.   

\

\textit{Case 2:} $i \in S$, $j \in [r, n- 2] - S$ \\
Then $a_{i} = 1$, so $a_{j} = 2$ by (1) and (2) of Lemma \ref{lemma:d-a}.  In this case, $v_{i} \leq 1 = v_{j}$.

\

\textit{Case 3:} $i \in [r, n - 2] -  S$, $j \in S$ \\
Then $a_{j} = 1$, so $a_{i} = 0$ by (1) and (2) of Lemma \ref{lemma:d-a}.  Thus, $v_{i} = 0 \leq v_{j}$.

\

\textit{Case 4:} $i, j \in [r, n - 2] - S$ \\
Then $a_{i} = 0$ or $a_{i} = 2$.  If $a_{i} = 0$, then $v_{i} = 0$, so $v_{i} \leq v_{j}$.  If $a_{i} = 2$, then $a_{j} = 2$ by (2) of Lemma \ref{lemma:d-a}, which means that $v_{j} = 1 \geq v_{i}$.

\

\textit{Case 5:} $n - 2 \rightarrow n - 1$ in $Q^{0}$.  \\
If $a_{n - 2} = 0$, then $v_{n - 2} = 0 \leq v_{n - 1}$.  If $a_{n - 2} = 2$, then $a_{n - 1} = 1$ by (3) of Lemma \ref{lemma:d-a}, so $v_{n - 1} = 1 \geq v_{n - 2}$.  If $a_{n - 2} = 1$ and $a_{n - 1} = 1$, then the $(n - 1)$th component of $\vv^{(0)}$ and hence of $\vv$ is equal to 1.  If $a_{n - 2} = 1$ and $a_{n - 1} = 0$, then the $(n - 1)$th and $(n - 2)$th components of $\vv^{(0)}$ are both equal to 0.  Since $n - 2 \rightarrow n - 1$ in $Q^{0}$ is critical with respect to $(\dd, \aa)$, the statement for $\vv^{(0)}$ also holds for $\vv$.

\

\textit{Case 6:} $n - 1 \rightarrow n - 2$ in $Q^{0}$ \\
If $a_{n - 1} = 0$, then $v_{n - 1} = 0 \leq v_{n - 2}$.  Suppose that $a_{n - 1} = 1$. Then $a_{n - 2} \geq 1$ by (4) of Lemma \ref{lemma:d-a}, so the $(n - 1)$th and $(n - 2)$th components of $\vv^{(0)}$ are both equal to 1.     If $a_{n - 2} = 1$, then  $n - 1 \rightarrow n - 2$ in $Q^{0}$ is critical with respect to $(\dd, \aa)$, so the $(n - 1)$th and $(n - 2)$th components of $\vv$ are also equal to 1.  If $a_{n - 2} = 2$, then $v_{n - 2} = 1$, so $v_{n - 1} \leq v_{n - 2}$.

\

Next, we need to prove that $Z^{\ww}$ occurs with nonzero coefficient in $F_{\dd'''}'$.  This means proving that if $i, j \in [r, n - 2] \cup \{ n \}$ with $i \rightarrow j$ in $Q^{0}$, then $w_{j} \geq w_{i}$.  If $i, j \in [r, n - 2]$, then the proof is similar to the proof that $v_{j} \geq v_{i}$ above given in Cases 1 to 4.

Suppose $n - 2 \rightarrow n$ in $Q^{0}$.  If $a_{n - 2} = 0$ or $2$, then the proof is similar to Case 5 above.  If $a_{n - 2} = 1$ and $a_{n} = 0$, then (5) of Lemma \ref{lemma:d-a} implies that $n - 2 \rightarrow n - 1$ is not critical, so $a_{n - 1} = 1$.  Thus, $w_{n - 2} = 0 \leq w_{n}$.   If $a_{n} = 1$, then $w_{n} = 1 \geq w_{n - 2}$.

Suppose $n \rightarrow n - 2$ in $Q^{0}$.   If $a_{n} = 0$, then $w_{n} = 0 \leq w_{n - 2}$.  Suppose that $a_{n} = 1$. Then $w_{n} = 1$, and  $a_{n - 2} \geq 1$ by (4) of Lemma \ref{lemma:d-a}.    If $a_{n - 2} = 2$, then $w_{n - 2} = 1 \geq w_{n}$.  If $a_{n - 2} = 1$, then $n \rightarrow n - 2$ is critical, so by (5) of Lemma \ref{lemma:d-a}, we must have $a_{n - 1} = 0$, which means that $w_{n - 2} = 1 \geq w_{n}$, as desired.  This proves that $(\vv, \ww) \in S_{\aa}$.

Now suppose that $(\vv, \ww) \in S_{\aa}$ with $\vv = (v_{1}, \ldots, v_{n})$, $\ww = (w_{1}, \ldots, w_{n})$.   Clearly, the $i$th component of $\vv^{(0)}$ is equal to the $i$th component of $\vv$ whenever $a_{i} = 0$ or $2$, and for $i = n - 1$ and $i = n$.   Note that if $i, j \in S$ and $i \rightarrow j$ in $Q^{0}$, then $v_{j} \geq v_{i}$ and $w_{i} \geq w_{j}$; since $v_{i} + w_{i} = 1$ and $v_{j} + w_{j} = 1$, one may show that  $v_{i} = v_{j}$, $w_{i} = w_{j}$.   It follows that if $i, j$ are in the same component of $S$, then $v_{i} = v_{j}$.   

It remains to show that if either the arrow between $n - 2$ and $n - 1$ in $Q^{0}$ or the arrow between $n - 2$ and $n$ in $Q^{0}$ is critical, then the $(n - 2)$th components of $\vv$ and $\vv^{(0)}$ are equal.  Suppose that $n - 2 \rightarrow n - 1$ and $n - 2 \rightarrow n$ in $Q^{0}$.  If $n - 2 \rightarrow n - 1$ is critical, then $n - 2 \rightarrow n$ is not critical by (5) of Lemma \ref{lemma:d-a}, which means that $a_{n - 2} = 1$, $a_{n - 1} = 0$, and $a_{n} = 1$.  Since $v_{n - 2} \leq v_{n - 1}$, this means that $v_{n - 2} = 0$.    If $n - 2 \rightarrow n$ is critical, then $n - 2 \rightarrow n - 1$ is not critical, and this means that $a_{n - 2} = 1$, $a_{n - 1} = 1$, and $a_{n} = 0$.  Since $w_{n - 2} \leq w_{n} = 0$, it follows that $w_{n - 2} = 0$ and hence, $v_{n - 2} = 1$.  The case where $n - 1 \rightarrow n - 2$ and $n \rightarrow n - 2$ in $Q^{0}$ is similar.
\end{proof}

\

For $(\vv, \ww) \in S_{\aa}$, let 
\begin{eqnarray}
\psi_{\dd''}(\vv) & = & -\frac{d}{2}(\gg_{\dd''}\cdot \vv) \\
\psi'_{\dd'''}(\ww) & = & -\frac{d}{2}(\gg_{\dd'''}\cdot \ww) - d(\gg_{\dd''}\cdot \ww). \\
\end{eqnarray}
The coefficients of $Z^{\vv}$ in $F_{\dd''}$ and $Z^{\ww}$ in $F'_{\dd'''}$ are given by $q^{\psi_{\dd''}(\vv)}$ and $q^{\psi'_{\dd'''}(\ww)}$, respectively.

Define a skew-symmetric bilinear form $\langle \,\cdot\, , \cdot \,\rangle: \ZZ^{n} \times \ZZ^{n} \rightarrow \frac{1}{2}\ZZ$ by 
\begin{eqnarray}\langle \ee_{i}, \ee_{j} \rangle = \frac{db_{ij}^{0}}{2}.
\end{eqnarray} 

Write
\begin{eqnarray}
\tilde{\psi}(\vv, \ww) = \psi_{\dd''}(\vv) + \psi'_{\dd'''}(\ww) + \langle \vv, \ww \rangle.
\end{eqnarray}
By (\ref{eqn:Z-monomial}), 
\begin{eqnarray}
q^{\psi_{\dd'}(\vv)}Z^{\vv} \cdot q^{\psi'_{\dd'''}(\ww)}Z^{\ww} = q^{\tilde{\psi}(\vv, \ww)}Z^{\aa}.
\end{eqnarray}

\begin{lemma} \label{lemma:psiD}
Let $[s_{1}, s_{2}] \in \mathcal{C}(S)$, and let $(\vv, \ww) \in S_{\aa}$ with $\vv = (v_{1}, \ldots, v_{n})$ and $\ww = (w_{1}, \ldots, w_{n})$ such that $v_{i} = 1$ for all $i \in [s_{1}, s_{2}]$.  Write $
\aa' = \sum_{i \in [s_{1}, s_{2}]} \ee_{i}$.
Then 
\begin{eqnarray}
\tilde{\psi}(\vv - \aa', \ww + \aa') = d + \tilde{\psi}(\vv, \ww).
\end{eqnarray}
\end{lemma}
\begin{proof}
Note that 
\begin{eqnarray}  \label{eqn:phi-tilde}
\tilde{\psi}(\vv - \aa', \ww + \aa') - \tilde{\psi}(\vv, \ww) = -\frac{d}{2}(\gg_{\dd''} +  \gg_{\dd'''}) \cdot \aa' + \langle \aa, \aa' \rangle.
\end{eqnarray}

We show that the right hand side of this expression is equal to $d$.  
Assume that $\dd'', \dd'''$ are as at (\ref{eqn:dd-1}).  (The proof when $\dd'', \dd'''$ are as at (\ref{eqn:dd-2}) is similar.)  Also, assume that $s_{1} \geq 2$.  (The proof below may be easily modified in the case that $s_{1} = 1$ by omitting any expression that contains $s_{1} - 1$ as a subscript.)  First, suppose that $s_{2} < n - 2$.  By Theorem \ref{thm:g-vec-classical}, 
\begin{eqnarray}
(-\gg_{\dd''} - \gg_{\dd'''})\cdot \aa' =  2(s_{2} - s_{1} + 1) - \sum_{\begin{array}{cc} i \in [r, n - 2] \cup \{ n \} \\ j \in [s_{1}, s_{2}] \end{array}}  [-b_{ji}^{0}]_{+} \\
\nonumber - \sum_{i \in [r, n - 1], \, j \in [s_{1}, s_{2}]} [-b_{ji}^{0}]_{+}
\end{eqnarray}
Since $[-b^{0}_{i, i + 1}]_{+} = 1$ or $[-b^{0}_{i + 1, i}]_{+} = 1$ (but not both) for each $i \in [s_{1}, s_{2} - 1]$, we have
\begin{eqnarray}
\sum_{i, j \in [s_{1}, s_{2}]} [-b^{0}_{ji}]_{+} = s_{2} - s_{1}.
\end{eqnarray}
Thus,
\begin{eqnarray}
(-\gg_{\dd''} - \gg_{\dd'''})\cdot \aa' = 2 - 2[-b^{0}_{s_{1}, s_{1} - 1}]_{+}\epsilon_{s_{1}} - 2[-b^{0}_{s_{2}, s_{2} + 1}]_{+},
\end{eqnarray}
where $\epsilon_{s_{1}} = 1$ if $s_{1} - 1 \geq r$, and $\epsilon_{s_{1}} = 0$ otherwise. 
Next, note that if $s_{2} \rightarrow s_{2} + 1$ in $Q^{0}$ (i.e., $b^{0}_{s_{2} + 1, s_{2}} = 1 = -b^{0}_{s_{2}, s_{2} + 1}$), then $a_{s_{2}} = 1$ implies  $a_{s_{2} + 1} = 2$ by (2) of Lemma \ref{lemma:d-a}.  On the other hand, if $s_{2} + 1 \rightarrow s_{2}$ in $Q^{0}$ (i.e., $b^{0}_{s_{2} + 1, s_{2}} = -1 = -b^{0}_{s_{2}, s_{2} + 1}$), then $a_{s_{2}} = 1$ implies $a_{s_{2} + 1} = 0$ by the same property.  In either case, we have $a_{s_{2} + 1}b^{0}_{s_{2} + 1, s_{2}}  = 2[-b^{0}_{s_{2}, s_{2} + 1}]_{+}$.  Likewise, it is easy to show that $a_{s_{1} - 1}b_{s_{1} - 1, s_{1}}^{0} = 2[-b^{0}_{s_{1}, s_{1} - 1}]_{+}$ if $s_{1} - 1 \geq r$ and $a_{s_{1} - 1} = 0$ otherwise.  Thus, 
\begin{eqnarray}
\langle \aa, \aa' \rangle & = & \frac{d}{2}(a_{s_{1} - 1}a_{s_{1}}b_{s_{1} - 1, s_{1}}^{0}\epsilon_{s_{1}} + a_{s_{2}}a_{s_{2} + 1}b_{s_{2}, s_{2} + 1}^{0}) \\
& =  & d[-b_{s_{1}, s_{1} - 1}^{0}]_{+}\epsilon_{s_{1}} + d[-b^{0}_{s_{2}, s_{2} + 1}]_{+}.  
\end{eqnarray}
This proves that the right hand side of (\ref{eqn:phi-tilde}) is equal to $d$ in this case.

Next, we consider the case where $s_{2} = n - 2$.  By reasoning as in the previous case, one may show that 
\begin{eqnarray}
(-\gg_{\dd''} - \gg_{\dd'''})\cdot \aa' = 2 - 2[-b^{0}_{s_{1}, s_{1} - 1}]_{+}\epsilon_{s_{1}} - [-b^{0}_{n - 2, n}]_{+} - [-b^{0}_{n - 2, n - 1}]_{+}.
\end{eqnarray}
Note that the arrows between $n - 2$ and $n$ and between $n - 2$ and $n - 1$ in $Q^{0}$ are not critical with respect to $(\dd, \aa)$ since $[s_{1}, n - 2] \in \mathcal{C}(S)$. Using the fact that $a_{n - 2} = 1$, one may show that 
\begin{eqnarray}
a_{n}b_{n, n - 2}^{0} & = & [-b^{0}_{n - 2, n}]_{+} \\
a_{n - 1}b_{n, n - 1}^{0} & = & [-b^{0}_{n - 2, n - 1}]_{+}.
\end{eqnarray}
Therefore,
\begin{eqnarray}
\langle \aa, \aa' \rangle & = &  \frac{d}{2}(a_{s_{1} - 1}a_{s_{1}}b_{s_{1} - 1, s_{1}}^{0}\epsilon_{s_{1}} + a_{n - 2}a_{n}b_{n, n - 2}^{0} + a_{n - 1}a_{n}b_{n, n - 1}^{0}) \\
& = & d[-b^{0}_{s_{1}, s_{1} - 1}]_{+}\epsilon_{s_{1}} + \frac{d}{2} [-b^{0}_{n - 2, n}]_{+} + \frac{d}{2}[-b^{0}_{n - 2, n - 1}]_{+}.
\end{eqnarray}
This proves the desired assertion.

\end{proof}

Now we are ready to complete the proof of the theorem for type $\mbox{D}_{n}$.  Assume that  $\vv^{(0)}, \ww^{(0)}$ are as at (\ref{eqn:vv-ww1}).  (The proofs in the other cases are similar.)  By Lemma \ref{lemma:S-a}, the coefficient of $Z^{\aa}$ in the expansion of $F_{\dd''}F'_{\dd'''}$ is 
\begin{eqnarray} \label{eqn:coeff-za}
\sum_{(\vv, \ww) \in S_{\aa}} q^{\tilde{\psi}(\vv, \ww)} = \sum q^{\tilde{\psi}(\vv^{(0)} - \aa', \ww^{(0)} + \aa')},
\end{eqnarray}
where the summation on the right hand side ranges over $\aa' = \sum_{C \in J} \ee_{C}$ such that $J \subset \mathcal{C}(S)$.     Using Lemma \ref{lemma:psiD} and induction on the cardinality of $J$, it is easy to show that
\begin{eqnarray}
\tilde{\psi}(\vv^{(0)} - \aa', \ww^{(0)} + \aa') = \tilde{\psi}(\vv^{(0)}, \ww^{(0)}) + d|J|.
\end{eqnarray}
Using the binomial theorem, the right hand side of (\ref{eqn:coeff-za}) may be rewritten as 
\begin{eqnarray}
\sum_{i = 0}^{|\mathcal{C}(S)|} (q^{\tilde{\psi}(\vv^{(0)}, \ww^{(0)}) + di})\left(\begin{array}{c} |\mathcal{C}(S)| \\ i \end{array}\right) = q^{\tilde{\psi}(\vv^{(0)}, \ww^{(0)})}(1 + q^{d})^{|\mathcal{C}(S)|}.
\end{eqnarray}
This proves that the coefficient of $Z^{\aa}$ in $F_{\dd''}F'_{\dd'''}$ has the form at (\ref{coeff-formD}), as desired.

\subsection{Type $\mbox{B}_{n}$}

Let $B^{0} = (b^{0}_{ij})$ be an acyclic $n \times n$ exchange matrix of type $\mbox{B}_{n}$.


 First, we will prove the theorem for $\dd$ of the form $\ee_{p} + \ldots + \ee_{r}$, where $1 \leq p \leq r \leq n$.   It suffices to show that for $\aa = (a_{1}, \ldots, a_{n}) \in \ZZ^{n}$, $Z^{\aa}$ occurs with nonzero coefficient in $F_{\dd}$ if and only if $u_{1}^{a_{1}}\ldots u_{n}^{a_{n}}$ occurs with nonzero coefficient in $F_{\dd}^{cl}$, and in this case, the coefficient of $Z^{\aa}$ in $F_{\dd}$ is of the form $q^{c}$ for some $c \in \frac{1}{2}\ZZ$.  Then Theorem \ref{thm:quantum-classical}  follows from Proposition \ref{prop:fpoly-coeff}.

\begin{lemma} \label{lemma:typeC-rec1} Let $\dd = \ee_{p} + \ldots + \ee_{r} \in \Phi_{+}(B^{0})$.   Then 
\begin{eqnarray}  \label{eqn:typeC-rec1}
F_{\dd} = q^{\lambda'}F_{\dd'}Z^{\aa'} + q^{\lambda''}F_{\dd''}Z^{\aa''}
\end{eqnarray}
for some $\lambda', \lambda'' \in \frac{1}{2}\ZZ$,  $\aa', \aa'' \in \ZZ_{\geq 0}^{n}$,  and $\dd', \dd'' \in \Phi_{+}(B^{0}) \cup \{ 0 \}$ such that $\dd', \dd'' < \dd$ and the $r$th component of $\dd'$ and $\dd''$ are both 0.
\end{lemma}
\begin{proof} A similar strategy as in the proof of Lemma \ref{typeDrec} will be used.  Note that the diagonals $\alpha_{1}, \ldots, \alpha_{n}$ divide half of $\PP_{2n + 2}$ into $n$ triangles.  For $i = 1, \ldots, n - 1$, let $\Delta_{i}$ be the triangle containing $\alpha_{i}$ and $\alpha_{i + 1}$ as sides, and let $\Delta_{0}$ be the triangle with $\alpha_{1}$ as a side but not $\alpha_{2}. $   For $i \in [1, n - 1]$, let $v_{i}$ be the vertex opposite $\alpha_{i}$ in $\Delta_{i}$.   Let $v_{n}$ be the endpoint of $\theta(\alpha_{n - 1})$ which is not on the diagonal $\alpha_{n}$.  Let $w_{i}$ be the vertex opposite $\alpha_{i}$ in the triangle $\Delta_{i - 1}$ for $i \in [1, n]$.   

Mutate the initial cluster in directions $k = p, \ldots, r$, and consider the corresponding sequence of mutations of $\mathcal{D}^{0}$.  When the mutation in direction $k \in [p, r]$ occurs, the new diagonal obtained is $[w_{p}v_{k}]$ if $k < n$ and $[w_{p}\overline{w_{p}}]$ if $k = n$.   Note that the elements of $\mathcal{D}^{0}$ that the diagonal $[w_{p}v_{k}]$ $(k < n)$ intersects are precisely $\alpha_{p}, \ldots, \alpha_{k}$, and the elements that the diagonal $[w_{p}\overline{w_{p}}]$ intersects are precisely $\alpha_{p}, \ldots, \alpha_{n}, \theta(\alpha_{n - 1}), \ldots, \theta(\alpha_{p})$.  In either case, the $k$th cluster variable in the final cluster has denominator vector $\ee_{p} + \cdots + \ee_{k}$.  Observe that $[w_{p}v_{k}]$ is not a diameter unless $k = r = n$.  By using (\ref{eqn:bc2}) or (\ref{eqn:bc3}) in conjunction with Theorem \ref{thm:quantum-fpoly-rec}, it is easy to show that an equation of the form (\ref{eqn:typeC-rec1}) holds.
\end{proof}

Proceed by induction on the difference $r - p$.  It is easy to show that $F_{\ee_{i}} = q^{d}Z_{i} + 1$ for $i \in [1, n - 1]$ and $F_{\ee_{n}} = q^{d/2}Z_{n} + 1$ using Theorem \ref{thm:quantum-fpoly-rec}.    Now suppose that $r - p > 1$, and assume the theorem has been proven for all $\ee_{p'} + \ldots + \ee_{r'} \in \Phi_{+}(B^{0})$ such that $r' - p' < r - p$.   Let $\dd = \ee_{p} + \cdots + \ee_{r}$.   By reasoning as in the type $\mbox{D}_{n}$ case, one may show that for $\aa = (a_{1}, \ldots, a_{n}) \in \ZZ^{n}_{\geq 0}$, $Z^{\aa}$ occurs with nonzero coefficient in $F_{\dd}$ if and only if $u_{1}^{a_{1}}\ldots u_{n}^{a_{n}}$ occurs with nonzero coefficient in $F_{\dd}^{cl}$.  Furthermore, one of the expressions  $q^{\lambda'}F_{\dd'}Z^{\aa'}, q^{\lambda''}F_{\dd''}Z^{\aa''}$ in Lemma \ref{lemma:typeC-rec1} contains all of the terms in $F_{\dd}$ with $Z_{r}$ in it, and the other expression contains all of the terms in $F_{\dd}$ without $Z_{r}$.    It is known that the coefficient of every monomial in $F_{\dd'}$ and $F_{\dd''}$ is a power of $q$, so the same statement  is true of $F_{\dd}$.  This finishes the proof of the theorem in the case where $\dd = \ee_{p} + \cdots + \ee_{r}$.

For the remainder of the proof of the theorem, we turn our attention to denominator vectors of the form $\dd = \sum_{i = p}^{n} \ee_{i} + \sum_{j = r}^{n} \ee_{j} \in \Phi_{+}(B^{0})$, where $p < r$.  The base of the induction occurs when $n - p = 0$, in which case, $\dd = \ee_{n}$, and the theorem has already been proven.

Now let $1 \leq p < r \leq n$, and assume that Theorem \ref{thm:quantum-classical} has been proven for $\dd = \sum_{i = p'}^{n} \ee_{i} + \sum_{j = r'}^{n} \ee_{j}$ for $1 \leq p' < r' \leq n$ such that $n - p' < n - p$.  It suffices to show that for $\aa = (a_{1}, \ldots, a_{n}) \in \ZZ^{n}$, $Z^{\aa}$ occurs with nonzero coefficient in $F_{\dd}$ if and only if $u_{1}^{a_{1}}\ldots u_{n}^{a_{n}}$ occurs with nonzero coefficient in $F_{\dd}^{cl}$, and in this case, the coefficient of $Z^{\aa}$ in $F_{\dd}$ is of the form 
\begin{eqnarray} \label{typeCcoeff}
q^{c_{1}}(1 + q^{d})^{\rho_{\dd}(\aa)}(1 + q^{2d})^{c_{2}}.
\end{eqnarray}
Once this is done, it is easy to finish the proof of the theorem as follows.  By Proposition \ref{prop:fpoly1}, setting $q = 1$ and $Z_{i} = u_{i}$ for $i \in [1, n]$ in $F_{\dd}$ yields $F_{\dd}^{cl}$.  This forces $c_{2} = \phi_{\dd}(\aa) - \rho_{\dd}(\aa)$.  Note that the expression at (\ref{typeCcoeff}) can be rewritten as 
\begin{eqnarray}
q^{c_{1}'}(q^{-\frac{d}{2}} + q^{\frac{d}{2}})^{\rho_{\dd}(\aa)}(q^{-d} + q^{d})^{\phi_{\dd}(\aa) - \rho_{\dd}(\aa)}
\end{eqnarray}
for some $c_{1}' \in \frac{1}{2}\ZZ$.  By applying Proposition \ref{prop:fpoly-coeff}, it follows that $c_{1}' = -\frac{d}{2}\delta \cdot \gg_{\dd} \cdot \aa$, as desired.

\begin{lemma} \label{typeCrec} Let $\dd = \ee_{p} + \cdots + 2\ee_{r} + \cdots + 2\ee_{n} \in \Phi_{+}(B^{0})$.  At least one of the two statements below  is true:
\begin{enumerate}
\item For some $\lambda', \lambda'' \in \frac{1}{2}\ZZ$, $\aa', \aa'' \in \ZZ^{n}$, and $\dd', \dd'' \in \Phi_{+}(B^{0}) \cup \{ 0 \}$ such that $\dd', \dd'' < \dd$ and the $p$th component of both vectors is 0, the quantum $F$-polynomial $F_{\dd}$ satisfies

\begin{eqnarray} \label{eqn:fpolyC-rec1}
F_{\dd} = q^{\lambda'}F_{\dd'}Z^{\aa'} + q^{\lambda''}F_{\dd''}Z^{\aa''}
\end{eqnarray}

\item For some  $\aa', \aa'' \in \ZZ^{n}$, $\lambda', \lambda'' \in \frac{1}{2}\ZZ$, and $\dd' \in \Phi_{+}(B^{0}) \cup \{ 0 \}$ such that $\dd' < \dd$ and the $p$th component of $\dd'$ is 0, we have
\begin{eqnarray} \label{eqn:fpolyC-rec2}
F_{\dd} = q^{\lambda'}F_{\dd'}Z^{\aa'} + q^{\lambda''}F_{\dd''}F'_{\dd''}Z^{\aa''}
\end{eqnarray} 
where 
\begin{eqnarray} \label{eqn:fpoly-rec2}
\dd'' & =  &\ee_{r} + \cdots + \ee_{n},  \\
F'_{\dd''} & = & L[\gg_{\dd''}](F_{\dd''}).
\end{eqnarray}
\end{enumerate}
\end{lemma}
\begin{proof}  
Similar reasoning as in Lemma \ref{typeDrec} will be used. Use the same notation as established in the proof of Lemma \ref{lemma:typeC-rec1}.  First, mutate  in directions $k = r, \ldots, n$.  

Let $r \leq j_{s} < \ldots < j_{1} \leq n - 1$ be indices such that the arrow between $j_{\ell}$ and $j_{\ell} + 1$ in $Q^{0}$ is opposite the arrow between $r - 1$ and $r$ for $\ell = 1, \ldots, s$.  Next, we continue mutating in directions $k = j_{1}, \ldots, j_{s}$.  Let $w$ be the endpoint of the diagonal $\alpha_{r - 1}$ which is not equal to $w_{r}$.   Then $\alpha'_{j_{\ell}}$ is mutated to $\alpha''_{j_{\ell}} = [\overline{w_{r}}v_{j_{\ell + 1}}]$ (for $\ell \in [1, s - 1]$) and $\alpha_{j_{s}}'$ is mutated to $[\overline{w_{r}}w]$.  

 Continue mutating in directions $k = r - 1, \ldots, p$.  Then $\alpha_{k}$ is mutated to $\alpha_{k}' = [\overline{w_{r}}w_{k}]$ for $k = p, \ldots, r - 1$.   Note that $\alpha_{p}'$ corresponds to the cluster variable with denominator vector $\dd$ since $\alpha_{p}'$ crosses the diagonals $\alpha_{p}, \ldots, \alpha_{n}, \theta(\alpha_{r}), \ldots, \theta(\alpha_{n - 1})$, and $\theta(\alpha_{p}')$ crosses the diagonals  $\theta(\alpha_{p}), \ldots, \theta(\alpha_{n - 1}), \alpha_{n}, \alpha_{r}, \ldots, \alpha_{n - 1}$.

Let $\mathcal{D}'$ be the final collection of diagonals obtained through this process.  Let $c_{1}, c_{2}$ be the endpoints of $\alpha_{p}$.   The diagonal $\alpha'_{p}$ is the diagonal of a quadrilateral whose sides are each in $\mathcal{D}'$ or sides of $\mathbb{P}_{2n + 2}$.  To be more precise, the sides of the quadrilateral are $\beta_{1} = [\overline{w_{r}}c_{1}]$, $\beta_{2} = [\overline{w_{r}}c_{2}]$,  and $\alpha_{p - 1}$ (if $p - 1 \geq 1$); the remaining sides of the quadrilateral (if any) are sides of $\mathbb{P}_{2n + 2}$.
 
Observe that the only elements of $\mathcal{D}^{0}$ that the $\beta_{i}$ can intersect come from the list  $\alpha_{p + 1}, \ldots, \alpha_{n}, \theta(\alpha_{r}), \ldots, \theta(\alpha_{n - 1})$, and $\theta(\beta_{i})$ can only intersect diagonals in the list  $\theta(\alpha_{p + 1}), \ldots, \theta( \alpha_{n}), \alpha_{n}, \alpha_{r}, \ldots, \alpha_{n - 1}$.   Thus, each $\beta_{i}$ corresponds to some $\dd' \in \Phi_{+}(B^{0})$ such that $\dd' < \dd$, and the $p$th component of $\dd'$ is equal to 0.

Use Theorem \ref{thm:quantum-fpoly-rec} in conjunction with (\ref{eqn:bc1}) or (\ref{eqn:bc3}).  If $c_{i} \neq w_{r}$ for $i = 1, 2$, then the diagonals $\beta_{1}, \beta_{2}$ are not diameters, and we get an equation for $F_{\dd}$ of the form (\ref{eqn:fpolyC-rec2}) where the $\dd', \dd''' \in \Phi_{+}(B^{0})$  correspond to the diagonals $\beta_{1}, \beta_{2}$.  If $c_{i} = w_{r}$ for some $i$, then observe that $[w_{r}\overline{w_{r}}]$ intersects the interiors of $\alpha_{r}, \ldots, \alpha_{n}$.  In this case, (\ref{eqn:fpoly-rec2}) holds.
\end{proof}

 Using Lemma \ref{typeCrec}  and a similar argument as in the type $\mbox{D}_{n}$,  it is easy to show that  for $\aa = (a_{1}, \ldots, a_{n}) \in \ZZ^{n}$, $Z^{\aa}$ occurs with nonzero coefficient in $F_{\dd}$ if and only if $u_{1}^{a_{1}}\ldots u_{n}^{a_{n}}$ occurs with nonzero coefficient in $F_{\dd}^{cl}$.

We will consider the coefficients of $F_{\dd''}F_{\dd''}'$ for $\dd'', \dd''$ as in Lemma \ref{typeCrec}.  Let $\aa = (a_{1}, \ldots, a_{n}) \in \ZZ^{n}$ such that $Z^{\aa}$ occurs with nonzero coefficient in the expansion of $F_{\dd''}F_{\dd''}'$.   For $\aa \in \ZZ^{n}$, let $\rho(\aa) = 1$ if the $n$th component of $\aa$ is 1, and let $\rho(\aa) = 0$ otherwise.  The next goal is to show that the coefficient of $Z^{\aa}$ in $F_{\dd''}F_{\dd''}'$ is of the form 
\begin{eqnarray} \label{coeff-formC}
q^{c_{1}}(1 + q^{d})^{\rho(\aa)}(1 + q^{2d})^{c_{2}}
\end{eqnarray} 
for some $c_{1}, c_{2} \in \frac{1}{2}\ZZ$. 

Let $S$ be the subgraph of $Q^{0}$ induced by $\{ i \in [1, n] : a_{i} = 1 \}$.   Let $\mathcal{C}(S)$ denote the set of components of the graph $S$.  Let 
\begin{eqnarray} 
\label{v0} \vv^{(0)} & = & \displaystyle \sum_{i \in [1, n], \, a_{i} \geq 1} \ee_{i} \\
\label{w0} \ww^{(0)} & = & \displaystyle \sum_{i \in [1, n], \, a_{i} = 2} \ee_{i}.
\end{eqnarray}
Let $S_{\aa}$ denote the set of $(\vv, \ww) \in \ZZ^{n} \times \ZZ^{n}$ such that $\vv + \ww = \aa$, and $Z^{\vv}, Z^{\ww}$ each occur with nonzero coefficient in $F_{\dd''}$ (and thus in $F_{\dd''}'$).

We will need to use the following properties of $\aa$, which may be easily proven using the known formula for $F_{\dd''}^{cl}$.

\begin{lemma}  \label{lemma:typeC-propa}
\begin{enumerate}
 \item $0 \leq a_{i} \leq 2$ for $i \in [r, n]$, $a_{i} = 0$ otherwise.
 \item If $i \rightarrow j$ in $Q^{0}$, $i, j \in [r, n]$, then $a_{i} \leq a_{j}$.
\end{enumerate}
\end{lemma}
\begin{proof}
Let $(\vv, \ww) \in S_{\aa}$ with $\vv = (v_{1}, \ldots, v_{n}), \ww = (w_{1}, \ldots, w_{n})$.  Part (1) of Theorem  \ref{thm:fpoly-classical} implies that $0 \leq v_{i}, w_{i} \leq 1$ for $i \in [r, n]$ and $v_{i}, w_{i} = 0$ for $i \in [1, r - 1]$, so statement (1) of the lemma follows.  Part (2) of Theorem \ref{thm:fpoly-classical} implies that if $i, j \in [r, n]$ with $i \rightarrow j$ in $Q^{0}$, then $v_{i} \leq v_{j}$ and $w_{i} \leq w_{j}$.  The second part of the lemma follows immediately.
\end{proof}

The next lemma explicitly gives the elements of the set $S_{\aa}$.

\begin{lemma} \label{lemma:setS_a}
\begin{eqnarray} \label{setS_a}
S_{\aa} = \{  (\vv^{(0)} - \aa', \ww^{(0)} + \aa') : J \subset \mathcal{C}(S), \aa' = \sum_{C \in J} \ee_{C} \}.
\end{eqnarray}
\end{lemma} 
\begin{proof}  Let $(\vv, \ww) \in \ZZ^{n} \times \ZZ^{n}$, and write $\vv = (v_{1}, \ldots, v_{n})$ and $\ww = (w_{1}, \ldots, w_{n})$.  

First, suppose that $(\vv, \ww) = (\vv^{(0)} - \aa', \ww^{(0)} + \aa')$ with $\aa'$ as in the right hand side of (\ref{setS_a}).    It suffices to prove that if $i \rightarrow j$ in $Q^{0}$, $i, j \in [r,  n]$, then $v_{i} \leq v_{j}$ and $w_{i} \leq w_{j}$.   If $a_{i} = 0$, then $v_{i} = 0 \leq v_{j}$.  If $a_{i} = 2$, then $a_{i} \leq a_{j} = 2$, so $v_{j} = 1 \geq v_{i}$.  Now suppose that $a_{i} = 1$.  Then $i, j$ are in the same component in $S$, so $v_{i} = v_{j}$.    Similarly, one shows that $w_{i} \leq w_{j}$.  

Next, suppose that $(\vv, \ww) \in S_{\aa}$, and write $\vv^{(0)} = (v^{(0)}_{1}, \ldots, v^{(0)}_{n})$.   Let $\aa' = \vv^{(0)} - \vv = (a_{1}', \ldots, a_{n}')$.   It is easy to see that $v_{i} = v_{i}^{(0)}$ whenever $a_{i} = 0$ or $a_{i} = 2$, so $a_{i}' = 0$ for such indices $i$.   If $i \rightarrow j$ in $Q^{0}$ and $i, j$ are vertices in $S$, then $a_{i} = a_{j} = 1$, and  one may show that $v_{i} = v_{j} = 1$ and $w_{i} = w_{j} = 0$, or $v_{i} = v_{j} = 0$ and $w_{i} = w_{j} = 1$. Thus, if $i, j$ are vertices in $S$ connected an edge, then $a_{i}' = a_{j}'$.  Consequently, the same equality holds if $i, j$ are vertices in the same component in $S$.   Thus, $\aa'$ is a sum of elements $\ee_{C}$, where $C \in \mathcal{C}(S)$, as desired.
\end{proof}

For $(\vv, \ww) \in S_{\aa}$,  let
\begin{eqnarray}
\label{psi} \psi_{\dd''}(\vv) & = & -\frac{d}{2}(\delta \cdot \gg_{\dd''}\cdot \vv) \\
\label{psiprime} \psi'_{\dd''}(\ww) & = & - \frac{3d}{2}(\delta \cdot \gg_{\dd''}\cdot \ww). 
\end{eqnarray}
The coefficients of $Z^{\vv}$ in $F_{\dd''}$ and $Z^{\ww}$ in $F'_{\dd''}$ are given by $q^{\psi_{\dd''}(\vv)}$ and $q^{\psi'_{\dd''}(\ww)}$, respectively.

Define a skew-symmetric bilinear form $\langle \,\cdot\, , \cdot \,\rangle: \ZZ^{n} \times \ZZ^{n} \rightarrow \frac{1}{2}\ZZ$ by 
\begin{eqnarray}\langle \ee_{i}, \ee_{j} \rangle = d\cdot \sgn(b_{ij}^{0})
\end{eqnarray} 

Write
\begin{eqnarray} \label{def:tilde-psi}
\tilde{\psi}(\vv, \ww) = \psi_{\dd''}(\vv) + \psi'_{\dd''}(\ww) + \langle \vv, \ww \rangle.
\end{eqnarray}
From (\ref{eqn:Z-monomial}), it follows that  
\begin{eqnarray} \label{psi-tilde-prop}
q^{\psi_{\dd'}(\vv)}Z^{\vv} \cdot q^{\psi'_{\dd''}(\ww)}Z^{\ww} = q^{\tilde{\psi}(\vv, \ww)}Z^{\aa}.
\end{eqnarray}

\begin{lemma} \label{lemma:psi}
Let $[s_{1}, s_{2}] \in \mathcal{C}(S)$, and let $(\vv, \ww) \in S_{\aa}$ with $\vv = (v_{1}, \ldots, v_{n})$ such that $v_{i} = 1$ for all $i \in [s_{1}, s_{2}]$.  Write $
\aa' = \sum_{i \in [s_{1}, s_{2}]} \ee_{i}$.  If $s_{2} < n$, then 
\begin{eqnarray}
\tilde{\psi}(\vv - \aa', \ww + \aa') = 2d + \tilde{\psi}(\vv, \ww).
\end{eqnarray}
Otherwise, if $s_{2} = n$, then 
\begin{eqnarray}
\tilde{\psi}(\vv - \aa', \ww + \aa') = d + \tilde{\psi}(\vv, \ww).
\end{eqnarray}
\end{lemma}
\begin{proof}
Note that 
\begin{eqnarray}  \label{eqn:phi-tilde-C}
\tilde{\psi}(\vv - \aa', \ww + \aa') - \tilde{\psi}(\vv, \ww) = -d \delta \cdot \gg_{\dd''} \cdot \aa' + \langle \aa, \aa' \rangle.
\end{eqnarray}

We show that the right hand side of this expression is equal to $2d$ if $s_{2} < n$ and equal to $d$ if $s_{2} = n$.  
Assume that $s_{1} \geq 2$.  (The proof below may be easily modified in the case that $s_{1} = 1$ by omitting any expression that contains $s_{1} - 1$ as a subscript.)  First, suppose that $s_{2} \leq n - 1$.  By Theorem \ref{thm:g-vec-classical}, 
\begin{eqnarray}
-\delta \cdot \gg_{\dd''} \cdot \aa' =  2(s_{2} - s_{1} + 1) - \sum_{\begin{array}{cc} i \in [r, n] \\ j \in [s_{1}, s_{2}] \end{array}} 2 [-b_{ji}^{0}]_{+} 
\end{eqnarray}
Since $[-b^{0}_{i, i + 1}]_{+} = 1$ or $[-b^{0}_{i + 1, i}]_{+} = 1$ (but not both) for each $i \in [s_{1}, s_{2} - 1]$, we have
\begin{eqnarray}
\sum_{i, j \in [s_{1}, s_{2}]} [-b^{0}_{ji}]_{+} = s_{2} - s_{1}.
\end{eqnarray}
Thus,
\begin{eqnarray}
-\delta \cdot \gg_{\dd''} \cdot \aa'  = 2 - 2[-b^{0}_{s_{1}, s_{1} - 1}]_{+}\epsilon_{s_{1}} - 2[-b^{0}_{s_{2}, s_{2} + 1}]_{+},
\end{eqnarray}
where $\epsilon_{s_{1}} = 1$ if $s_{1} - 1 \geq r$, and $\epsilon_{s_{1}} = 0$ otherwise. 

Note that if $s_{2} \rightarrow s_{2} + 1$ in $Q^{0}$ (i.e. $b^{0}_{s_{2}, s_{2} + 1} = -1 = -b^{0}_{s_{2} + 1, s_{2}})$, then $a_{s_{2}} = 1$ implies that $a_{s_{2} + 1} = 2$ by Lemma \ref{lemma:typeC-propa}.  If $s_{2} + 1 \rightarrow s_{2}$ in $Q^{0}$ (i.e. $b^{0}_{s_{2}, s_{2} + 1} = 1 = -b^{0}_{s_{2} + 1, s_{2}})$, then $a_{s_{2}} = 1$ implies that $a_{s_{2} + 1} = 0$, again by the same lemma.  In either case, this means that $a_{s_{2} + 1}\sgn(b^{0}_{s_{2} + 1, s_{2}}) = 2[-b^{0}_{s_{2}, s_{2} + 1}]_{+}$.  Likewise, it is easy to show that $a_{s_{1} - 1}\sgn(b^{0}_{s_{1} - 1, s_{1}}) = 2[-b^{0}_{s_{1}, s_{1} - 1}]_{+}$ when $s_{1} - 1 \geq r$. It follows that 
\begin{eqnarray}
\langle  \aa, \aa'  \rangle & = & d(a_{s_{1}}a_{s_{1} - 1}\sgn(b^{0}_{s_{1} - 1, s_{1}})\epsilon_{s_{1}} + a_{s_{2}}a_{s_{2} + 1}\sgn(b^{0}_{s_{2} + 1, s_{2}})) \\
& = & 2d[-b^{0}_{s_{1}, s_{1} - 1}]_{+}\epsilon_{s_{1}} + 2d[-b^{0}_{s_{2}, s_{2} + 1}]_{+}.
\end{eqnarray}
This proves that the right hand side of (\ref{eqn:phi-tilde-C}) is equal to $2d$.

Now suppose that $s_{2} = n$.   In this case, $-\delta \cdot \gg_{\dd''} \cdot \aa'$ equals
\begin{eqnarray}
2(n - s_{1}) + 1 - 2[-b^{0}_{n - 1, n}]_{+} - [-b^{0}_{n, n - 1}]_{+} - \sum_{\begin{array}{cc} i \in [r, n - 1] \\ j \in [s_{1}, n - 1] \end{array}} 2 [-b_{ji}^{0}]_{+} \\
= 3  -2[-b^{0}_{n - 1, n}]_{+} - [-b^{0}_{n, n - 1}]_{+} - 2[-b^{0}_{s_{1}, s_{1} - 1}]_{+}\epsilon_{s_{1}}.
\end{eqnarray}
Either $b^{0}_{n - 1, n} = -1$ and $b^{0}_{n, n - 1} = 2$, or $b^{0}_{n - 1, n} = 1$ and $b^{0}_{n, n - 1} = -2$, which means that $-2[-b^{0}_{n - 1, n}]_{+} - [-b^{0}_{n, n - 1}]_{+} = -2$.
Thus, 
\begin{eqnarray}
-\delta \cdot \gg_{\dd''} \cdot \aa' = 1 - 2[-b^{0}_{s_{1}, s_{1} - 1}]_{+}\epsilon_{s_{1}}.
\end{eqnarray}
On the other hand, 
\begin{eqnarray}
\langle \aa, \aa' \rangle = da_{s_{1}}a_{s_{1} - 1}\sgn(b^{0}_{s_{1} - 1, s_{1}})\epsilon_{s_{1}} = 2d[-b^{0}_{s_{1}, s_{1} - 1}]_{+}\epsilon_{s_{1}}.
\end{eqnarray}
Thus, in the case that $s_{2} = n$, we have that the right hand side of (\ref{eqn:phi-tilde-C}) is equal to $d$. \end{proof}

 Observe that the coefficient of $Z^{\aa}$ in $F_{\dd''}F'_{\dd''}$ is 
\begin{eqnarray} \label{eqn:coeff-za-typeC}
\sum_{(\vv, \ww) \in S_{\aa}} q^{\tilde{\psi}(\vv, \ww)} = \sum q^{\tilde{\psi}(\vv^{(0)} - \aa', \ww^{(0)} + \aa')},
\end{eqnarray}
where the summation on the right hand side ranges over $\aa' = \sum_{C \in J} \ee_{C}$ such that $J \subset \mathcal{C}(S)$.     Using Lemma \ref{lemma:psi} and induction on the cardinality of $J$, it is easy to show that
\begin{eqnarray}
\tilde{\psi}(\vv^{(0)} - \aa', \ww^{(0)} + \aa') = \tilde{\psi}(\vv^{(0)}, \ww^{(0)}) + 2d|J|  
\end{eqnarray}
if the $n$th component of $\aa'$ is 0, and 
\begin{eqnarray}
\tilde{\psi}(\vv^{(0)} - \aa', \ww^{(0)} + \aa') = \tilde{\psi}(\vv^{(0)}, \ww^{(0)}) + 2d(|J| - 1) + d
\end{eqnarray}
if the $n$th component of $\aa'$ is 1.

\

If the $n$th component of $\aa$ is 0 or 2, then the right hand side of (\ref{eqn:coeff-za-typeC}) may be rewritten as 
\begin{eqnarray}
 \sum_{i = 0}^{|\mathcal{C}(S)|} (q^{\tilde{\psi}(\vv^{(0)}, \ww^{(0)}) + 2di})\left(\begin{array}{c} |\mathcal{C}(S)| \\ i \end{array}\right) = q^{\tilde{\psi}(\vv^{(0)}, \ww^{(0)})}(1 + q^{2d})^{|\mathcal{C}(S)|}
\end{eqnarray}

If the $n$th component of $\aa$ is 1, then the summation in the right hand side of (\ref{eqn:coeff-za-typeC}) may be split into two summations, one which ranges over all $\aa'$ as above such that the $n$th component of $\aa'$ is 1, and the other over $\aa'$ where the $n$th component is 1.  Then the right hand side of (\ref{eqn:coeff-za-typeC}) equals 
\begin{eqnarray}
\sum_{i = 0}^{|\mathcal{C}(S) - 1|} q^{\tilde{\psi}(\vv^{(0)} - \aa', \ww^{(0)} + \aa')} & + & q^{d}\sum_{i = 0}^{|\mathcal{C}(S) - 1|} q^{\tilde{\psi}(\vv^{(0)} - \aa', \ww^{(0)} + \aa')}  \\
& = & (1 + q^{d}) \sum_{i = 0}^{|\mathcal{C}(S)| - 1} (q^{\tilde{\psi}(\vv^{(0)}, \ww^{(0)}) + 2di})\left(\begin{array}{c} |\mathcal{C}(S)| - 1 \\  i \end{array}\right)  \\
& = & q^{\tilde{\psi}(\vv^{(0)}, \ww^{(0)})}(1 + q^{2d})^{|\mathcal{C}(S)| - 1}(1 + q^{d}).
\end{eqnarray}
This proves that the coefficient of $Z^{\aa}$ in $F_{\dd''}F'_{\dd''}$ has the form at (\ref{coeff-formC}), as desired.

Now we complete the proof of the theorem.  Suppose that (\ref{eqn:fpolyC-rec2}) holds.    (The reasoning when (\ref{eqn:fpolyC-rec1}) holds is similar.)  By reasoning as in the type $\mbox{D}_{n}$ case, one may show that one of the expansions of  $q^{\lambda'}F_{\dd'}Z^{\aa'}, q^{\lambda''}F_{\dd''}F'_{\dd''}Z^{\aa''}$ contains all the terms in $F_{\dd}$ with $Z_{p}$, while the other contains all terms in $F_{\dd}$ without $Z_{p}$.    Thus, it suffices to show that the coefficients in the expansions of $q^{\lambda'}F_{\dd'}Z^{\aa'}, q^{\lambda''}F_{\dd''}F'_{\dd''}Z^{\aa''}$ are of the form (\ref{typeCcoeff}).

First, consider the expression $q^{\lambda''}F_{\dd''}F'_{\dd''}Z^{\aa''}$.    Note  that the $n$th component of $\aa''$ is 0.  Thus, for $\aa \in \ZZ^{n}$ such that $Z^{\aa - \aa''}$ occurs with nonzero coefficient in $F_{\dd''}F'_{\dd''}$, we have $\rho_{\dd}(\aa) = \rho(\aa - \aa'')$.  Since it is known that the coefficient of $Z^{\aa - \aa''}$ in $F_{\dd''}F'_{\dd''}$ is of the form $q^{c_{1}}(1 + q^{d})^{\rho(\aa - \aa'')}(1 + q^{2d})^{c_{2}}$, it follows that the coefficient of $Z^{\aa}$ in the expansion $q^{\lambda''}F_{\dd''}F'_{\dd''}Z^{\aa''}$ is of the form at (\ref{typeCcoeff}).

Consider the coefficients in the expression $q^{\lambda'}F_{\dd'}Z^{\aa'}$.  Suppose that is known that the coefficients of $F_{\dd'}$ are of the form given at (\ref{typeCcoeff}).    Either $\dd' = \ee_{p'} + \cdots + 2\ee_{r'} + \cdots + 2\ee_{n}$ for some $p < p' < r' \leq n$, or $\dd' = \ee_{p'} + \cdots + \ee_{r'}$ for some $p < p' \leq r' \leq n$.  In the former case, the $n$th component of $\aa'$ is equal to 0.   Thus, for all $\aa \in \ZZ^{n}$ such that $Z^{\aa - \aa'}$ occurs with nonzero coefficient in $F_{\dd'}$, it follows that $\rho_{\dd'}(\aa - \aa') = \rho_{\dd}(\aa)$.  In the latter case, all coefficients of $F_{\dd'}^{cl}$ are 0 or 1, which means that $\rho_{\dd}(\aa) = 0$, $\phi_{\dd}(\aa) = 0$ for all $\aa \in \ZZ^{n}$ such that $Z^{\aa - \aa'}$ occurs with nonzero coefficient in $F_{\dd'}$.  This proves that the coefficients of $F_{\dd}$ has the form at (\ref{typeCcoeff}), as desired.

\subsection{Type $\mbox{C}_{n}$}

Let $B^{0} = (b^{0}_{ij})$ be an acyclic $n \times n$ exchange matrix of type $\mbox{C}_{n}$.


First, we consider the case when $\dd = \ee_{p} + \cdots + \ee_{r} \in \Phi_{+}(B^{0})$ $(1 \leq p \leq r \leq n$).  Once the lemma below is proven, the proof for this type of denominator vector follows by similar reasoning as in the type $\mbox{B}_{n}$ case.

\begin{lemma} \label{lemma:typeB-rec1} Let $\dd = \ee_{p} + \ldots + \ee_{r} \in \Phi_{+}(B^{0})$.   Then 
\begin{eqnarray}  \label{eqn:typeB-rec1}
F_{\dd} = q^{\lambda'}F_{\dd'}Z^{\aa'} + q^{\lambda''}F_{\dd''}Z^{\aa''}
\end{eqnarray}
for some $\lambda', \lambda'' \in \frac{1}{2}\ZZ$, $\aa', \aa'' \in \ZZ_{\geq 0}^{n}$, and $\dd', \dd'' \in \Phi_{+}(B^{0}) \cup \{ 0 \}$ such that $\dd', \dd'' < \dd$ and the $p$th component of $\dd'$ and $\dd''$ are both 0.
\end{lemma}
\begin{proof}   We use a similar strategy to Lemma \ref{typeDrec}. Use the same notation as established in the proof of  Lemma \ref{lemma:typeC-rec1}.
Mutate the initial cluster in directions $k = r, \ldots, p$, and consider the corresponding sequence of mutations applied to  $\mathcal{D}^{0}$.  When the mutation in direction $k \in [p, r]$ occurs, the new diagonal obtained is $[w_{k}v_{r}]$.  This diagonal intersects the diagonals $\alpha_{k}, \ldots, \alpha_{r}$ and no other elements in $\mathcal{D}^{0}$.  Thus, the new $k$th cluster variable obtained has denominator vector $\ee_{k} + \cdots + \ee_{r}$.    The lemma follows from (\ref{eqn:bc3}) and Theorem \ref{thm:quantum-fpoly-rec}.
\end{proof}

For the remaining $\dd \in \Phi_{+}(B^{0})$, it suffices to prove that for $\aa = (a_{1}, \ldots, a_{n}) \in \ZZ^{n}$, $Z^{\aa}$ occurs with nonzero coefficient in $F_{\dd}$ if and only if $u_{1}^{a_{1}}\ldots u_{n}^{a_{n}}$ occurs with nonzero coefficient in $F_{\dd}^{cl}$, and in this case, the coefficient of each monomial $Z^{\aa}$ is of the form
\begin{eqnarray} \label{typeBcoeff}
q^{c_{1}}(1 + q^{d})^{c_{2}}
\end{eqnarray}
for some $c_{1}, c_{2} \in \frac{1}{2}\ZZ$.  After this is done, the rest of the proof of Theorem \ref{thm:quantum-classical} is similar to the type $\mbox{D}_{n}$ case.

First, we will consider denominator vectors of the form $\dd = 2\ee_{p} + \cdots + 2\ee_{n - 1} + \ee_{n}$.

\begin{lemma}  \label{lemma:typeB1} Let $\dd = 2\ee_{p} + \cdots + 2\ee_{n - 1} + \ee_{n} \in \Phi_{+}(B^{0})$.    Then 
there exist $\lambda', \lambda'' \in \frac{1}{2}\ZZ$, $\aa', \aa'' \in \ZZ^{n}$, and $\dd', \dd''$ with each of $\dd', \dd''$  either equal to 0 or of the form $\ee_{\ell} + \cdots + \ee_{r}$ for some $p \leq \ell \leq r \leq n - 1$, such that 
\begin{eqnarray}
F_{\dd} = q^{\lambda_{1}}F_{\dd'}F_{\dd'}'Z^{\aa'} + q^{\lambda_{2}}F_{\dd''}F_{\dd''}'Z^{\aa''},
\end{eqnarray}
where 
\begin{eqnarray}
F_{\dd'} & = & L[\gg_{\dd'}](F_{\dd'}) \\
F_{\dd''} & = & L[\gg_{\dd''}](F_{\dd''}) 
\end{eqnarray}
Consequently, for $\aa = (a_{1}, \ldots, a_{n}) \in \ZZ^{n}_{\geq 0}$, $Z^{\aa}$ occurs with nonzero coefficient in $F_{\dd}$ if and only if $u_{1}^{a_{1}}\ldots u_{n}^{a_{n}}$ occurs with nonzero coefficient in $F_{\dd}^{cl}$.
\end{lemma}
\begin{proof}
Use the same notation as in Lemma \ref{lemma:typeC-rec1}.  Mutate the initial cluster in directions $k = p, \ldots, n$.   In this case,  the $k$th cluster variable in the final cluster has denominator vector $\ee_{p} + \cdots + \ee_{k}$ $(k \in [p, n - 1])$, while the $n$th cluster variable has denominator vector $2\ee_{p} + \cdots + 2\ee_{n - 1} + \ee_{n}$.   The lemma follows from Theorem \ref{thm:quantum-fpoly-rec} and (\ref{eqn:bc2}).
\end{proof}

We need to show that the coefficients of $F_{\dd'}F_{\dd'}'$ and $F_{\dd''}F_{\dd''}'$ are of the form given at (\ref{typeBcoeff}).   Let $\dd'' = \ee_{p} + \cdots + \ee_{r}$, where $p \leq  r \leq n - 1$, and let $\aa = (a_{1}, \ldots, a_{n}) \in \ZZ^{n}_{\geq 0}$ such that $Z^{\aa}$ occurs with nonzero coefficient in $F_{\dd}$.  Using the vectors $\dd''$ and  $\aa$, define $S$, $\vv^{(0)}$, $\ww^{(0)}$, $S_{\aa}$, $\mathcal{C}(S)$, $\psi_{\dd''}$, $\psi'_{\dd''}$ as in the previous subsection.   Observe that Lemmas \ref{lemma:typeC-propa} and  \ref{lemma:setS_a} hold in this setting using the same proof.

Define a skew-symmetric bilinear form $\langle \,\cdot\, , \cdot \,\rangle: \ZZ^{n} \times \ZZ^{n} \rightarrow \frac{1}{2}\ZZ$ by 
\begin{eqnarray}\langle \ee_{i}, \ee_{j} \rangle = \frac{d}{2}b_{ij}^{0}
\end{eqnarray} 
if $(i, j) \neq (n, n - 1)$, and 
\begin{eqnarray}
\langle \ee_{n}, \ee_{n - 1} \rangle = d b_{n, n - 1}^{0}.
\end{eqnarray}
Define $\tilde{\psi}$ as at (\ref{def:tilde-psi}) using this bilinear form.  Then (\ref{psi-tilde-prop}) holds in this setting as well.

\begin{lemma} \label{lemma:psiB}
Let $[s_{1}, s_{2}] \in \mathcal{C}(S)$, and let $(\vv, \ww) \in S_{\aa}$ with $\vv = (v_{1}, \ldots, v_{n})$ and $\ww = (w_{1}, \ldots, w_{n})$ such that $v_{i} = 1$ for all $i \in [s_{1}, s_{2}]$.  Write $
\aa' = \sum_{i \in [s_{1}, s_{2}]} \ee_{i}$.  Then
\begin{eqnarray}
\tilde{\psi}(\vv - \aa', \ww + \aa') = d + \tilde{\psi}(\vv, \ww).
\end{eqnarray}
\end{lemma}
\begin{proof}
The proof will use similar reasoning to the proof of Lemma \ref{lemma:psi}.  It suffices to show that 
\begin{eqnarray} \label{eqn-typeB}
-d \delta \cdot \gg_{\dd''} \cdot \aa' + \langle \aa, \aa' \rangle = d.
\end{eqnarray}
Note that $-\delta \cdot \gg_{\dd''} \cdot \aa' = -\gg_{\dd''} \cdot \aa'$ since the $n$th component of $\aa'$ is 0.
First, consider the case where $s_{2} < n - 1$.  Then one can show that 
\begin{eqnarray}
-\delta \cdot \gg_{\dd''} \cdot \aa' & = & (s_{2} - s_{1} + 1) - \sum_{i \in [p, r], j \in [s_{1}, s_{2}]} [-b_{ji}^{0}]_{+} \\
& = & 1 - [-b^{0}_{s_{1}, s_{1} - 1}]_{+}\epsilon_{s_{1}} - [-b^{0}_{s_{2}, s_{2} + 1}]_{+}\epsilon_{s_{2}},
\end{eqnarray}
where $\epsilon_{s_{1}} = 0$ if $s_{1} - 1 < r$ and $\epsilon_{s_{1}} = 1$ otherwise, and $\epsilon_{s_{2}} = 0$ if $s_{2} + 1 > r$ and $\epsilon_{s_{2}} = 1$ otherwise.
Furthermore,
\begin{eqnarray}
\langle \aa, \aa' \rangle = a_{s_{1} - 1}a_{s_{1}} \epsilon_{s_{1}}\cdot \frac{d}{2}b^{0}_{s_{1} - 1, s_{1}} + a_{s_{2} + 1}a_{s_{2}} \epsilon_{s_{2}}\cdot \frac{d}{2}b^{0}_{s_{2} + 1, s_{2}}.
\end{eqnarray}
It is easy to show that $a_{s_{1} - 1}b^{0}_{s_{1} - 1, s_{1}} = 2[-b^{0}_{s_{1}, s_{1} - 1}]_{+}$ if $s_{1} - 1 \geq p$ and 
$a_{s_{2} + 1}b^{0}_{s_{2} + 1, s_{2}} = 2[-b^{0}_{s_{2}, s_{2} + 1}]_{+}$ if $s_{2} + 1 \leq r$.  Thus, (\ref{eqn-typeB}) holds in this case.

Now suppose that $s_{2} = n - 1$.  Then $r = n - 1$, and 
\begin{eqnarray}
-\delta \cdot \gg_{\dd''} \cdot \aa' & = & (n - s_{1}) - \sum_{i \in [p, n - 1], \, j \in [s_{1}, n - 1]}  [-b_{ji}^{0}]_{+} \\
& = & 1 - [-b^{0}_{s_{1}, s_{1} - 1}]_{+}\epsilon_{s_{1}}.
\end{eqnarray}
Also, 
\begin{eqnarray}
\langle \aa, \aa' \rangle = a_{s_{1} - 1}a_{s_{1}} \epsilon_{s_{1}}\cdot \frac{d}{2} b^{0}_{s_{1} - 1, s_{1}}.
\end{eqnarray}
Again, (\ref{eqn-typeB}) holds.
\end{proof}

By reasoning as in the type $\mbox{D}_{n}$ case, it is easy to show that the coefficient of $Z^{\aa}$ in $F_{\dd''}F_{\dd''}'$ has the form at (\ref{typeBcoeff}).  It then follows from Lemma \ref{lemma:typeB1} that the same is true of the coefficients of $F_{\dd}$.

Next, we consider denominator vectors of the form 
\begin{eqnarray}
\sum_{i = p}^{n} \ee_{i} + \sum_{j = r}^{n - 1} \ee_{j} = \ee_{p} + \cdots + 2\ee_{r} + \cdots + 2\ee_{n - 1} + \ee_{n} \in \Phi_{+}(B^{0}),
\end{eqnarray}
where $p < r$.   Proceed by induction on $n - p$.   The base of the induction occurs when  $n - p = 1$, in which case $\dd = \ee_{n - 1} + \ee_{n}$, and Theorem \ref{thm:quantum-classical} is already proven.  Now suppose that $\dd = \ee_{p} + \cdots + 2\ee_{r} + \cdots + 2\ee_{n - 1} + \ee_{n}$, and the theorem has been proven for denominator vectors $\ee_{p'} + \cdots + 2\ee_{r'} + \cdots + 2\ee_{n - 1} + \ee_{n}$ such that $p < p' < r' \leq n - 1$.

\begin{lemma}  Let $\dd = \ee_{p} + \cdots + 2\ee_{r} + \cdots + 2\ee_{n - 1} + \ee_{n} \in \Phi_{+}(B^{0})$.
Then there exist $\lambda_{1}, \lambda_{2} \in \frac{1}{2}\ZZ$, $\cc', \cc'' \in \ZZ_{\geq 0}^{n}$, and $\dd', \dd'' \in \Phi_{+}(B^{0}) \cup \{ 0 \}$ satisfying $\dd', \dd'' < \dd$ and the $p$th component of $\dd'$ and $\dd''$ are 0, such that
\begin{eqnarray}
F_{\dd} = q^{\lambda_{1}}F_{\dd'}Z^{\cc'} + q^{\lambda_{2}}F_{\dd''}Z^{\cc''}.
\end{eqnarray}  \label{eqn:typeCrec3}
Consequently, for $\aa = (a_{1}, \ldots, a_{n}) \in \ZZ^{n}_{\geq 0}$, $Z^{\aa}$ occurs with nonzero coefficient in $F_{\dd}$ if and only if $u_{1}^{a_{1}}\ldots u_{n}^{a_{n}}$ occurs with nonzero coefficient in $F_{\dd}^{cl}$.
\end{lemma}
\begin{proof}
The proof is similar to that of Lemma \ref{typeCrec}.  Use the same notation as defined in that proof and the same sequence of mutations, the diagonal $\alpha_{p}'$ obtained now corresponds to the denominator vector $\dd$, and the diagonals $\beta_{1}, \beta_{2}$ correspond to $\dd', \dd'' \in \Phi_{+}(B^{0})$ such that $\dd', \dd'' < \dd$ and the $p$th component of each vector is equal to 0.   To finish the lemma, use (\ref{eqn:bc1}) or (\ref{eqn:bc3}) together with Theorem \ref{thm:quantum-fpoly-rec}.
\end{proof}

Use the expression for $F_{\dd}$ given by (\ref{eqn:typeCrec3}).  By arguing as in the type $\mbox{D}_{n}$ case, one can show that all the terms in $F_{\dd}$ containing $Z_{p}$ are contained in one of the expansions of $q^{\lambda_{1}}F_{\dd'}Z^{\cc'}, q^{\lambda_{2}}F_{\dd''}Z^{\cc''}$, and all of the terms not containing $Z_{p}$ are contained in the other expansion.  By induction or by using the theorem in the already established cases above,  the coefficients of $F_{\dd'}$, $F_{\dd''}$ are of the form (\ref{typeBcoeff}).  Therefore, the same is true of the coefficients of $F_{\dd}$.  This completes the proof of the theorem for type $\mbox{C}_{n}$.

\section*{Acknowledgments}
The author would like to thank her advisor Andrei Zelevinsky for his many helpful comments and suggestions.

\end{document}